\documentclass[12pt]{article}

\usepackage{color}
\usepackage{latexsym}

\usepackage{amsmath}
\usepackage{amssymb}
\usepackage{amsthm}
\usepackage{mathrsfs}
\usepackage{enumerate}
\usepackage{setspace}
\usepackage{graphicx}
\usepackage{caption}
\usepackage{subcaption}
\usepackage{url}
\usepackage{float}
\usepackage{pdflscape}
\usepackage{titlesec}
\allowdisplaybreaks

\usepackage[authoryear]{natbib}

\usepackage{wrapfig}

\newtheorem{defin}{Definition}[section]
\newtheorem{prop}{Proposition}[section]
\newtheorem{thm}{Theorem}[section]
\newtheorem{lemma}{Lemma}[section]

\newtheorem{rem}{Remark}[section]

\newcommand{\E}{\mathbb{E}}
\newcommand{\R}{\mathbb{R}}
\newcommand{\Var}{\mathrm{Var}}
\newcommand{\cum}{\mathrm{cum}}
\newcommand{\I}[1]{\mathbb{I}_{\{ #1 \}}}
\newcommand{\Cov}{\mathrm{Cov}}

\newcommand{\Op}{O_{\mathbb{P}}}

\newcommand{\p}{\mathbb{P}}
\newcommand{\pr}{^\prime}
\newcommand{\bv}{\Big\vert}
\newcommand{\tnull}{t_0}
\newcommand{\T}{n}

\newcommand{\Ntheta}{\mathcal{N}_{\tnull,T}}

\newcommand{\gn}{\varphi_\T}
\newcommand{\unu}{\vartheta}
\newcommand{\omjn}{\omega_{j,\T}}

\newcommand{\eps}{\varepsilon}
\newcommand{\Z}{\mathbb{Z}}
\newcommand{\weak}{\stackrel{\mathcal{D}}{\longrightarrow}}

\newcommand{\Ind}[1]{\mathbb{I}_{\{#1\}}}

\begin{document}
\def\spacingset#1{\renewcommand{\baselinestretch}%
{#1}\small\normalsize} \spacingset{1}

\bibpunct[, ]{(}{)}{;}{a}{,}{, }

{\title{\sc $\,$\vspace{-33mm} \\ Quantile Spectral Analysis \\ for Locally Stationary Time Series}}

\author{
Stefan {\sc Birr}$^{\rm a *}$,
Stanislav {\sc Volgushev}$^{\rm b *}$,
Tobias {\sc Kley}$^{\rm c *}$, \\
Holger  {\sc Dette}$^{\rm a}$\thanks{ Supported by the Sonderforschungsbereich ``Statistical modelling of nonlinear dynamic processes" (SFB~823, Teilprojekt A1, C1) of the Deutsche Forschungsgemeinschaft.},
and Marc {\sc Hallin}$^{\rm d}$\thanks{Acad\' emie Royale de Belgique,  CentER (Tilburg University), and ECORE.  Supported by     the IAP research network grant~P7/06 of the Belgian government (Belgian Science
Policy) and a Cr\' edit aux Chercheurs of the  Fonds de la Recherche Scientifique-FNRS.}
\vspace{3mm} \\
$^{\rm a}$Ruhr-Universit\" at Bochum \vspace{1mm}\\
$^{\rm b}$Cornell University Ithaca \vspace{1mm}\\
$^{\rm c}$London School of Economics and Political Science
  \\
$^{\rm d}$ECARES,  Universit\' e Libre de Bruxelles 
}
\date{}
\maketitle
\vspace{-5mm}
\begin{abstract} 
Classical spectral methods are subject to two fundamental limitations:  they 
only can account for covariance-related serial dependencies, and they require 
second-order stationarity.  Much attention has been devoted lately to 
quantile-based spectral methods that go beyond covariance-based serial 
dependence features. At the same time, covariance-based methods relaxing stationarity  into  much 
weaker {\it local stationarity} conditions have been developed for a variety of 
time-series models. Here, we are combining those two approaches by proposing 
quantile-based spectral methods for locally stationary processes. We therefore 
introduce a time-varying version of the copula spectra that have 
been  recently proposed in the literature, along with a suitable local lag-window estimator.
We propose a new definition of local {\it strict} stationarity that allows  us to 
handle completely general non-linear processes without any moment assumptions,
thus accommodating our quantile-based concepts and methods. We establish a central limit 
theorem for the new estimators, and illustrate the power of the proposed 
 methodology by means of a simulation study. Moreover, in  two empirical studies
(namely of the Standard \& Poor's 500 series and a temperature dataset recorded in Hohenpeissenberg)
we demonstrate that the new approach detects   important variations 
in  serial dependence structures both across time and across quantiles. Such variations remain completely undetected, and are 
actually undetectable, via classical covariance-based spectral methods.  
\end{abstract}



\spacingset{1.4}

\section{Introduction}\label{Secintro}
\def\theequation{1.\arabic{equation}}
\setcounter{equation}{0}

For more than a century, spectral methods have been among the favorite tools of time-series analysis. The concept of  {\it periodogram} was proposed and discussed as early as~1898 by Schuster, who coined the term in a study   (\cite{schuster1898}) of meteorological series.
 The modern mathematical foundations of the approach were laid between 1930 and 1950 by such big names as Wiener, Cram\' er, Kolmogorov, Bartlett, and Tukey. The main reason for the unwavering success  of spectral methods is that  they are entirely {\it model-free}, hence fully nonparametric; as such, they can be considered a precursor to the subsequent development of nonparametric techniques in the area and, despite their age,  they  still are part of the leading group of methods in the field.

The classical spectral approach to time series analysis, however, remains deeply marked by two major restrictions:

\begin{description}
\item[(i)] as a second-order theory, it is essentially  limited to  modeling  first- and second-order dynamics: being entirely covariance-based, it  cannot accommodate heavy tails and infinite variances, and cannot account for any dynamics in conditional skewness, kurtosis, or tail behavior;
\item[(ii)] the assumption of second-order stationarity is pervasive: except for processes that, possibly after some adequate transformation such as differencing or cointegration, are second-order stationary,  observations exhibiting time-varying  distributional features are ruled out.

\end{description}

The first of these two limitations recently has attracted much attention, and new quantile-related  spectral analysis tools have been proposed, which  do not require second-order moments, and are able to capture serial features that cannot be accounted for by the classical second-order approach. Pioneering contributions in that direction are \cite{hong1999}  and  \cite{li2008}, who coined the names of {\it Laplace spectrum} and {\it Laplace periodogram}. The Laplace spectrum concept was further studied by  \cite{hagemann2011}, and extended into {\it cross-spectrum} and spectral {\it kernel} concepts by  \cite{dhkv2014}, who also introduced {\it copula-based} versions of the same. Those copula spectral quantities are indexed by couples $(\tau_1,\tau_2)$ of quantile levels, and their  collections (for~$(\tau_1,\tau_2)\in [0,1]^2$) account for any features of the joint distributions  of  pairs $(X_t,  X_{t-k})$ in a strictly stationary process~$\{X_t\}$, without requiring any distributional assumptions such as 
the existence of finite moments.

%
That thread of literature also includes \cite{li2012,li2014}, \cite{kvdh2014}, and \cite{leerao2012}. Somewhat  different approaches were taken by \cite{hong2000}, \cite{dmz2013}, and several others; in the time domain, \cite{linwha2007}, \cite{davmik2009}, and \cite{hlow2014} introduced the related concepts of {\it quantilograms} and {\it extremograms}. Strict or second-order stationarity, however, are essential in all those contributions.

The  pictures in Figure~\ref{SPFig}  show that the copula-based  spectral methods developed in \cite{dhkv2014}
(where we refer to for details)  indeed successfully account for serial features that remain out of reach in the traditional approach. The series considered in Figure~\ref{SPFig} is the classical S\&P500 index series, with $T=12092$ observations  from 1962 through~2013; more precisely, that series contains the differences of logarithms of daily opening and closing prices for about 51 years. That series is generally accepted to be   white noise, yielding  perfectly flat   periodograms. When rank-based copula periodograms are substituted for the classical ones, however, the picture looks quite different. Three rank-based copula periodograms  are shown in Figure~\ref{SPFig}, for the quantile levels $0.1$, $0.5$ and $0.9$, respectively. The central one, corresponding to the central part of the marginal distribution, is compatible with the assumption of white noise. But the more extreme ones (associated with  the quantile levels~$0.1$  and~$0.9$) yield a peak at the origin, pointing at a strong dependence in the tails which is definitely not present in the median part of the (marginal) distribution.

\begin{figure}
    \begin{center}
     \includegraphics[width=\linewidth]{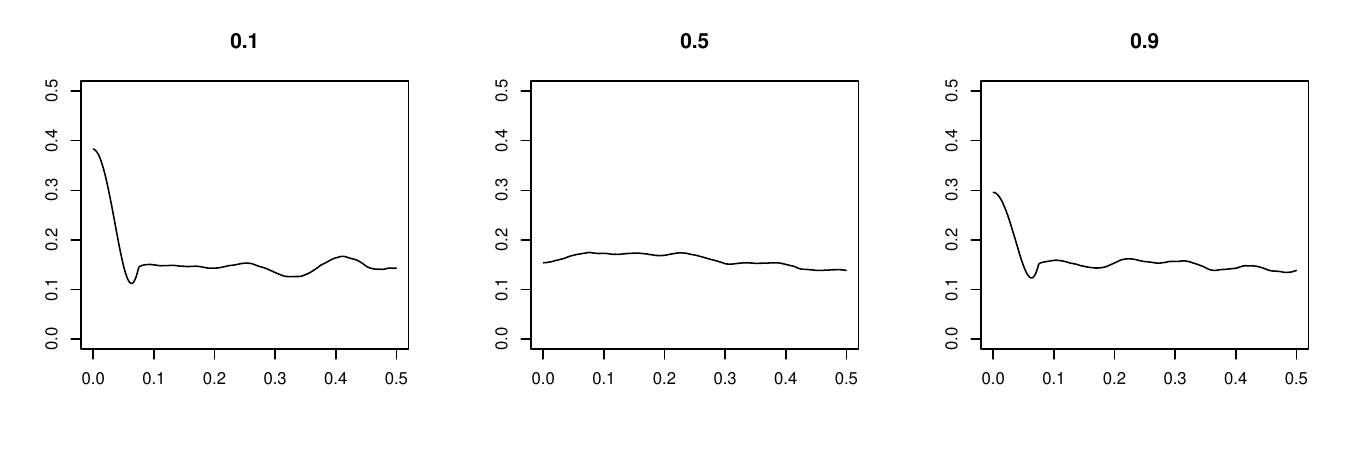}
    \end{center}
    \vspace{-10mm}
    \caption{ \small S\&P500, 1962-2013: the smoothed rank-based copula periodograms  for $\tau_1 =\tau_2 
    = 0.1$, $0.5$ and $0.9$, respectively. All curves are plotted against $\omega / 2\pi$.\vspace{-3mm}
    \label{SPFig}}
\end{figure}

%

Now, all periodograms  in Figure~\ref{SPFig} were computed from the complete series (51 years, $1\leq t\leq 12092$),  under the presumption of   stationarity (more precisely, stationarity  in distribution,  for all $k$, of the couples $(X_t, X_{t-k})$). Is that assumption likely to hold true?
The wavelet-based test proposed by \cite{Nason2013} reveals significant changes in the behavior of that time series, with most significant changes taking place around 1975, 1997 and during the period 2007-2013. Moreover, 
two rank-based copula periodograms for $\tau_1 = \tau_2 = 0.1$ computed before and after  the year $ 2007$ are shown in Figure \ref{Fig2007}. 
We observe differences in the height of the peak at the origin  before and after the year $2007$. 
These finding raise questions about  the second limitation of traditional spectral methods,  (second-order) stationarity. It has  motivated the development of a rich strand of literature, mainly along four (largely overlapping) lines:

\begin{description}
\item[(a)] {\it models with time-dependent parameters}: inherently parametric,  those models  are mimicking the traditional ones, but with parameters varying over time---see  \citet{subba1970} for a prototypical contribution, \cite{azrmel2006} for an in-depth study of the time-varying ARMA case; \vspace{-2mm}
\item[(b)] the {\it evolutionary spectral methods}, initiated by \cite{priestley1965}, where the process under study admits a spectral representation with {\it time-varying transfer function}---a second-order characterization, thus;   \vspace{-2mm}
\item[(c)] {\it piecewise stationary processes}, in relation with change-point analysis: see, e.g.,  \cite{dlry2005};\vspace{-2mm}
\item[(d)] the {\it locally stationary process} approach initiated by \cite{dahlhaus1997},  based on the assumption that, over a short period of time (that is, locally in time), the process under study behaves approximately as  a stationary one;  related concepts have been developed recently  
 by \cite{dahlrao2006}, 
   \cite{zhouwu2009a, zhouwu2009b}, \cite{roueff} and \cite{vogt2012}; wavelet-based versions also have been considered, as in \cite{nason2000}. %
%
We refer to Dahlhaus~(2012)  for a  survey of this approach.

\end{description}

\begin{figure} 
\centering{\includegraphics[width= 0.8\linewidth]{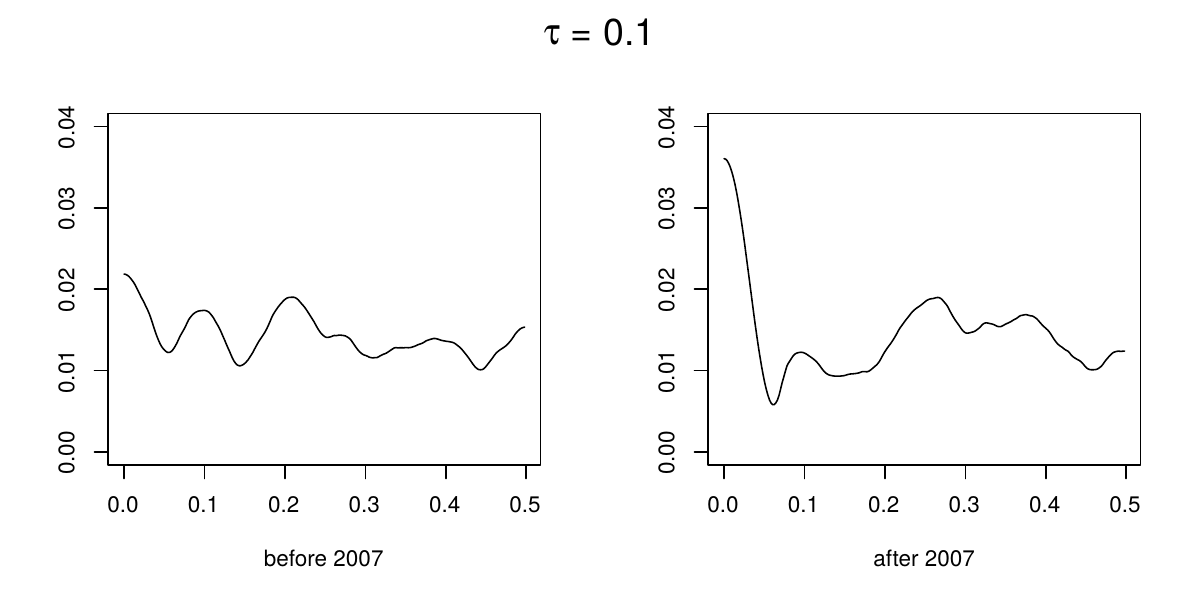}}
\caption{\small  S\&P500: the smoothed rank-based copula periodograms  for $\tau_1 =\tau_2 = 0.1$ calculated for 512 days before and 512 day after 01.01.2007. Both curves are plotted against $\omega / 2 \pi.$}
\label{Fig2007}
\end{figure}


Those four approaches, as already mentioned, are not without overlaps: the original concept by Dahlhaus~(1997)  is based on time-varying (second-order) spectral representations, turned into  time-domain linear MA($\infty$) ones by \cite{dahlpol2006};  \cite{dahlrao2006}  and \cite{fssr2008} deal with locally stationary ARCH models; although much more general,  \cite{zhouwu2009a, zhouwu2009b} also assume a form of time-varying nonlinear MA$(\infty)$ representation, and hence also resort to~(a).  Most references  require moment assumptions, either by nature (being based on a spectral representation), or by the nature of the stationary approximation they are considering.
%

In this paper, we  address the two limitations (i) and (ii) of traditional spectral analysis simultaneously by developing a locally stationary version
of the quantile-related spectral analysis proposed in \cite{dhkv2014}.  At the same time we provide a thorough theoretical underpinning for the proposed approach.
While  adopting the locally stationary ideas of (d), however, we turn them into a fully non-parametric 
and moment-free approach, adapted to the nature of quantile- and copula-based spectral concepts (see \cite{harvey2010} for a related, time-domain, attempt).   The definitions of local stationarity existing in the literature indeed are not general enough to accommodate quantile spectra,  and we therefore  formulate  a new concept of {\it local strict  stationarity}. Contrary to \cite{dahlpol2006} and \cite{zhouwu2009a,zhouwu2009b},   who deal  with time-varying (linear or nonlinear) moving averages, to \cite{dahlhaus1997}, which is based on time-varying second-order spectra,   to \cite{vogt2012}, where the approximation is in terms of random variables and requires finite moments of order $\rho >0$,  our approximation  is directly based on joint distributions, and does not involve any moments  nor specific data-generating processes. 
Its very general nature allows us to 
  extend to the quantile context  the definitions of a local spectrum, and to establish a central limit theorem for our local lag-window estimators. 
 
The \textit{time-varying 
copula spectrum} and its estimators are introduced in Section~\ref{sec:strict}  and
Section~\ref{sec:3}, respectively. In Section~\ref{sec:prac}, we illustrate the application of the new methodology by means of a small simulation study and two real-life examples,
while  the  theoretical properties of time-varying copula spectra and a corresponding lag-window estimator are investigated in Section~\ref{theory}. In particular, a central limit theorem for our local lag-window estimator is established. The proofs and additional information concerning the simulation studies and the datasets analyzed in Section \ref{SecSandP} and \ref{SecHPB} are deferred to an online supplement.

\section{Local strict stationarity and local copula spectra}\label{sec:strict}
\def\theequation{2.\arabic{equation}}
\setcounter{equation}{0}

 \subsection{Locally strictly stationary processes}

 Consider a series $(X_{1},\ldots ,X_{T})$ of length $T$ as being part of a triangular array $(X_{t,T},\ 1\leq t\leq T)$, $T\in\mathbb{N}$, of   finite-length 
  realizations   of nonstationary  processes \mbox{$\{X_{t,T},\ t\in\mathbb{Z}\}$}, $T\in\mathbb{N}$.
The intuitive idea behind all definitions  of local stationarity  consists in   the assumption  that those processes  have an approximately stationary behavior over a short period of time.  More formally,   one usually  assume the existence of a collection, indexed by~$\vartheta\in(0,1)$,  of stationary processes $\{X^\vartheta_{t},\ t\in~\mathbb{Z}\}$ such that  the nonstationary process~$\{X_{t,T},\ t\in\mathbb{Z}\}$ can be approximated (in a suitable way),  in the vicinity of time~$t $,  by the stationary process~$\{X^\vartheta_{t},\ t\in~\! \mathbb{Z}\}$ associated with~$\vartheta= t/T$.\vspace{1mm}

The exact nature of this approximation has to be adapted to the specific problem under study. If the objective is a locally stationary extension of classical spectral analysis, only the autocovariances Cov$(X_{t,T},X_{s,T})$ have to be approximated. In the quantile-related context considered here, the joint distributions of $X_{t,T}$ and $X_{s,T}$ are the feature of interest, and    traditional autocovariances are to be  replaced with autocovariances of indicators, of the form Cov$(\I{X_{t,T}\leq q_{t,T}(\tau_1)},\I{X_{s,T}\leq q_{s,T}(\tau_2)})$, where~$q_{t,T}(\tau_1)$ and $q_{s,T}(\tau_2)$ stand for the $\tau_1$-quantile of $X_{t,T}$ and the $\tau_2$-quantile of $X_{s,T}$, respectively, with  $\tau_1,\tau_2 \in (0,1)$; see \cite{li2008,li2012}, \cite{hagemann2011}, or \cite{dhkv2014}. Such  covariances only depend on the bivariate copulas of  $X_{t,T}$ and $X_{s,T}$.

In a strictly stationary context, this leads to the so-called {\it Laplace spectrum}, first considered by \cite{li2008}  for a   strictly stationary process $\{Y_t , t \in \mathbb{Z}\}$ with marginal   median zero. That spectrum is defined as
\[
 \mathcal{C}_{0.5,0.5}  (\omega) := \frac{1}{2\pi}\sum_{k \in \mathbb{Z}} e^{-i\omega k}\text{Cov}(\I{Y_0\leq 0},\I{Y_{-k}\leq 0}),\qquad\omega\in(-\pi,\pi ].
\]
Li's  concept was extended  by  \cite{hagemann2011},   \cite{li2012} and   \cite{dhkv2014}   to general quantile levels. The most general version, which also takes into account {\it cross}-covariances of indicators, was introduced by \cite{dhkv2014}. Denoting by $q $ the marginal quantile function of~$Y_t$, they define the {\it copula spectral density kernel}   as
\[
\mathcal{C}_{\tau_1,\tau_2}(\omega) := \frac{1}{2\pi}\sum_{k\in \mathbb{Z}} e^{-i\omega k}\text{Cov}(\I{Y_0\leq q(\tau_1)},\I{Y_{-k}\leq q(\tau_2)}) ,\quad \tau_1,\tau_2 \in (0,1),\ \omega\in(-\pi,\pi ].
\]

Those definitions all  heavily rely on the strict stationarity of the underlying time series; without this assumption, actually, they do not make much sense. 
 It seems natural, thus, to look for some adequate notion of local stationarity that can be employed to characterize the notion of a local copula-based spectrum.
However, the definitions of local stationarity previously considered in the literature are  placing  unnecessarily strong restrictions on the classes of processes that can be considered. In particular, \cite{dahlpol2006}, \cite{dahlrao2006} and \cite{vogt2012} 
  rely on moment assumptions that are neither desirable nor natural in a quantile context, and  are  not required  for the definition of copula spectra. We therefore  introduce a new concept of \textit{local strict stationarity} which completely avoids moment assumptions. 
 That  concept is not totally unrelated to the existing ones, though, and we also show that, under adequate conditions, processes that fit into the framework of \cite{dahlrao2006} or \cite{dahlpol2006} are locally strictly stationary in the new sense; see Section~\ref{sec:strictstatmodels} for details. Similar results certainly also could be obtained for the  \cite{zhouwu2009a, zhouwu2009b}  concept, but they are less obvious and,  in order to not overload the paper, we do not pursue into that direction. 


The copula spectral density kernels of a stationary process $\{Y_{t}\}$  are defined in terms of its bivariate marginal distribution functions.
Therefore, it is  natural to use  bivariate marginal distribution functions when evaluating, in the definition of local stationarity,  the distance between the non-stationary process $\{X_{t,T}\}$ 
and its stationary approximation $\{X^\unu_{t} \}$.

\begin{defin}   \label{def:strstat}
A triangular array $\{ (X_{t,T})_{t \in \mathbb{Z}} \}_{T \in \mathbb{N}}$ of processes is called {\em locally strictly stationary (of order two)} if there exists a constant $L>0$ and, for every~$\unu\in (0,1)$,  a strictly stationary process $\{X^\unu_{t} , t \in \mathbb{Z}\}$ such that, for every $1\leq r,s \leq T,$
 \begin{equation}\label{eq1}
    \big\Vert F_{r,s;T}(\cdot,\cdot) - G^{\unu}_{r-s}(\cdot,\cdot)\big\Vert_{\infty} \leq L \Big(\max(|r/T-\unu|,|s/T-\unu|) + {1}/{T}\Big),
\end{equation}
where $\Vert  \cdot  \Vert_\infty$ stands for the supremum norm, while $F_{r,s;T}(\cdot,\cdot)$ and $G^{\unu}_{k}(\cdot,\cdot)$ denote the joint distribution functions of $(X_{r,T}, X_{s,T})$ and $(X^\unu_{0},X^\unu_{-k})$, respectively.
\end{defin}
\noindent Here, {\textquoteleft of order two\textquoteright}  refers to the fact  that (\ref{eq1}) is based on bivariate distributions only. Letting $y$ tend to infinity in $F_{r,s;T}(x,y)$ and $G^{\unu}_{k}(x,y)$, we get an analogous condition for the marginal distributions~$F_{t;T}$ and~$G^\unu$ of $X_{t,T}$ and $X^{\unu}_{0}$, namely
\begin{equation}\label{eq2}
\big\|F_{t;T}(\cdot)-G^\unu(\cdot)\big\|_\infty \leq L \big\vert t/T - \unu\big\vert +  L/{T}.
\end{equation}

Intuitively, (\ref{eq1}) and (\ref{eq2}) imply  that the univariate and bivariate distribution functions~$F_{t;T}$ and $F_{r,s;T}$ of the process~$\{X_{t,T}\}$ are allowed to change smoothly over time. A crucial advantage of this definition is its nonparametric nature, as 
  it does not depend on any specific data-generating mechanism. 

Whenever the data-generating process can be described in terms of a parametric model, strict stationarity in the sense of Definition \ref{def:strstat} holds if the underlying parameters change smoothly over time. Familiar  examples   include MA$(\infty)$, ARCH$(\infty)$ and GARCH$(p,q)$ models with time-varying coefficients. Sufficient conditions  for   local strict stationarity of  those models  are discussed in Section~\ref{sec:strictstatmodels}, where we also provide explicit forms of the strictly stationary approximating processes.  
 

\subsection{Local copula spectral density kernels}
Turning to the definition of  a localized version of copula spectral density kernels,  first consider the copula cross-covariance kernels associated with the strictly stationary process~$\{X^\unu_{t} , t \in~\mathbb{Z}\}$, $\unu \in (0,1)$. The  {\it lag-$h$-copula cross-covariance kernel}  of  %
 $\{X^\unu_{t}\}$, as defined in \cite{dhkv2014},   is
\begin{equation*}
 \gamma^\unu_h(\tau_1,\tau_2) := \Cov(\I{X^\unu_{t}{} \leq q^\unu(\tau_1)},\I{X^\unu_{t-h}{}\leq q^\unu(\tau_2)}) ,\quad \tau_1,\tau_2 \in (0,1) ,
\end{equation*}
where $q^\unu(\tau)$ denotes $X^\unu_{t}$'s marginal quantile of order $\tau$.


 The cross-covariances involved in the above definition always exist, and their collection  (for~$\tau_1,\tau_2 \in (0,1)$ and given lag $h$) provides a canonical characterization of the joint copula of $(X^\unu_{t}{}, X^\unu_{t-h}{})$, hence, an approximate (in the sense of (\ref{eq1})) description of the joint copula of all couples of the form~$(X_{t,T}, X_{t-h,T})$. Therefore we also call~$\gamma^\unu_h(\tau_1,\tau_2)$  the {\it time-varying lag-$h$-copula cross-covariance kernel} of~$\{X_{t,T}\}.$ If we assume that,  for all~$\tau_1,\tau_2 \in (0,1)$,  the lag-$h$-covariance kernels $ \gamma^\unu_h(\tau_1,\tau_2)$ are  absolutely summable, we moreover can define the {\it local} or {\it time-varying copula spectral density kernel} of~$\{X_{t,T}\}$ as
\begin{eqnarray} \label{def:tvf}
 \mathfrak{f}^\unu(\omega,\tau_1,\tau_2) := \frac{1}{2 \pi} \sum_{h = - \infty}^{\infty} \gamma^\unu_h(\tau_1,\tau_2) e^{-ih\omega},\ \ \tau_1,\tau_2 \in (0,1), \ \ \omega\in(-\pi,\pi ].
\end{eqnarray}
The time-varying cross-covariance kernel then admits the representation
\begin{equation} \nonumber 
\gamma^\unu_h(\tau_1,\tau_2)
 = \int_{- \pi}^{\pi} e^{ih\omega}
\mathfrak{f}^\unu(\omega,\tau_1,\tau_2 )  d\omega , \qquad\omega\in(-\pi,\pi ] ,\quad \tau_1,\tau_2 \in (0,1) .
\end{equation}

\noindent {Comparing those representations with the local spectral densities of Dahlhaus~(1997), we see that the autocovariances of the approximating processes there are replaced by   copulas. This indicates that the local spectral density kernels~(\ref{def:tvf}) can be viewed as a completely non-parametric generalization of classical $L^2$-based tools. In particular, those kernels can capture  pairwise serial dependencies of arbitrary forms. For more detailed comparisons, we refer   to \cite{dhkv2014} and \cite{kvdh2014}.} The usefulness of the concepts discussed here for data analysis is provided, via simulation and the analysis of two real datasets, the classial S$\&$P 500 and a meteorological one,  in Section~\ref{sec:prac}.

\section{Estimation of local copula spectra}\label{sec:3}
\def\theequation{3.\arabic{equation}}
\setcounter{equation}{0}

Given observations $X_{1,T},\dots,X_{T,T}$, the classical approach to the estimation of the time-varying spectral density of a locally stationary time series consists in  considering a subset of~$\T$ data points centered around a time point $\tnull.$ 

To formalize ideas, let $m_T$ be a sequence of positive integers diverging to infinity as~$T \to~\!\infty$, and   define the discrete neighborhood
$\mathcal{N}_{\tnull,T} := \{ 1 \leq t \leq T : |t_0-t|< m_T \},$
with  cardinality~$\T = \T(m_T,T)$. 
 Denoting by $\omjn = 2\pi j/\T, 1 \leq j \leq \lfloor \frac{\T+1}{2}\rfloor$ the positive Fourier frequencies, let $\varphi_n: \omega\mapsto \varphi_n(\omega):= \omega_{j,n}$ be the piecewise constant function  mapping $\omega\in  (0,\pi)$ to the closest Fourier frequency, i.e. to  the frequency $\omega_{j,n}$  such that $\omega\in  (\omega_{j,n}-\frac{2\pi}{n},\, \omega_{j,n}+\frac{2\pi}{n}]$. Defining  
\[ T(k) :=\{t \in \Ntheta: t+k \in \Ntheta\},\quad 
\tilde F_{t_0;T}(x) := \frac{1}{2 T^{4/5}}\sum_{|t-t_0|\leq T^{4/5}} \I{X_{t,T} \leq x},  
\]
and $ \hat q_{t_0,T}(\tau) := \tilde F_{t_0;T}^{-1}(\tau)$, consider the 
 {\it local lag-window estimator} (at the Fourier frequen-\linebreak cies $\omjn = 2\pi j/\T$)
\begin{eqnarray} \label{esti}
\hat{\mathfrak{f}}_{t_0,T}(\omjn,\tau_1,\tau_2) &:=& \frac{1}{2\pi}\sum_{|k|\leq n-1} K(k/B_n) e^{-i \omjn k}\nonumber \\
&&\quad \times\  \frac{1}{n} \sum_{t \in T(k)} \Big(\I{X_{t,T} \leq \hat q_{t_0,T}(\tau_1)} - \tau_1 \Big)\Big( \I{X_{t+k;T} \leq \hat q_{t_0,T}(\tau_2)} - \tau_2 \Big),
\end{eqnarray}
  where $B_n \rightarrow \infty$ as $\T \rightarrow \infty$ and~$K: \R \to \R$ is 
 continuous in $x=0$ and satisfies~$K(0) = 1$ and  $\lim_{|x| \to \infty} K(x) = 0$. In order to extend this estimator $\hat{f}_{\tnull,T}(\cdot,\tau_1,\tau_2)$ to the interval~$(0,\pi),$ let~$
\hat{f}_{\tnull,T}(\omega,\tau_1,\tau_2) := \hat{f}_{\tnull,T}(\gn(\omega),\tau_1,\tau_2).
$
In Section \ref{sec:asyth} below, we prove that, under mild conditions on the bandwidth parameters and the underlying time series, the local lag-window estimator is  consistent for  
the copula spectral density $\mathfrak{f}^\unu(\omega,\tau_1,\tau_2)$ and asymptotically normally distributed.  This is a novel result even in the stationary case, as \cite{kvdh2014} consider an estimator based on smoothed periodograms instead. 

Before we address  the asymptotic theory for  the new estimators, we illustrate   their properties  and advantages  by means of a brief simulation study
and a detailed analysis  of two real-life datasets.

\section{Simulations and an empirical study}\label{sec:prac}
\def\theequation{4.\arabic{equation}}
\setcounter{equation}{0}

One important practical aspect of the estimation of a quantile spectral density is the choice of a local window length $n$ and a smoothing parameter $B_n$. In Section~\ref{sec:asyth} and Theorem~\ref{thm:asymain}, we derive the asymptotic distribution of the estimator, which allows to derive an expression for the smoothing parameters that minimizes the asymptotic mean squared error (see Remark~\ref{rem:bw} for additional details). Those expressions, of course, cannot be readily used in practice, since they depend on the actual  time-varying copula spectral densities and their derivatives, which are unknown. Estimating such derivatives is even more difficult than estimating the original spectral density, and a plug-in approach to bandwidth selection therefore seems difficult to implement. For the estimation of local $L^2$-spectra, an interesting alternative has been proposed by \cite{craomb2002}. Unfortunately, that approach relies on wavelets instead of local windows for   localization in time; whether  it can be implemented here is not clear. For the implementation of our methodology, we propose to study  different local window lengths and bandwidth parameters and in the simulation study we illustrate  the performance of the estimators for different window lengths.    

\subsection{Heatmaps: calibrating the color scale} \label{Seccalibr}

Plots of  time-varying spectral densities and their estimators are provided in the form of  {\it time-frequency heatmaps}.  The vertical axis in all those plots represents frequencies ($\omega/2\pi$, ranging from 0 to 0.5), the horizontal axis the span of time $1,\ldots , T$ over which the time-varying spectral quantities are estimated. All $3\times 3$ figures in this section show the real and imaginary parts for different combinations of quantile orders, organized as shown in Table~\ref{tableOrga}.
\begin{table}
\begin{center}
 \begin{tabular}{c|c|c} 
$\mathfrak{\hat f}_{t_0,T}({\omega,\alpha, \alpha})$ & ${\Im\mathfrak{\hat f}_{t_0,T}({\omega,\beta, \alpha})}$ & ${\Im \mathfrak{\hat f}_{t_0,T}({\omega,\gamma, \alpha})}$
\\[10pt]
\hline
${\Re\mathfrak{\hat f}_{t_0,T}({\omega,\beta, \alpha})}$ & ${\mathfrak{\hat f}_{t_0,T}({\omega,\beta, \beta})}$ & ${\Im\mathfrak{\hat f}_{t_0,T}({\omega,\gamma, \beta})}$
\\[10pt]
\hline
${\Re\mathfrak{\hat f}_{t_0,T}({\omega,\gamma, \alpha})}$ & ${\Re\mathfrak{\hat f}_{t_0,T}({\omega,\gamma,\beta})} $& ${\mathfrak{\hat f}_{t_0,T}({\omega,\gamma, \gamma})}$\\[10pt]
\end{tabular}
\caption{\small Patterns for the $3\times 3$ time-frequency heatmaps in Figures 5-8 and 11; throughout, we use~$\alpha = 0.1$, $\beta = 0.5$ and $\gamma = 0.9,$ with $t_0 \in {\cal T}_0 \subset \{1,\dots,T\}$ and $\omega \in (0,\pi)$. For example,  the top-right corner, in all those figures,  displays a time-frequency plot of the imaginary parts of the collection  $\big(\mathfrak{\hat f}_{t_0,T}(\omega,0.9,0.1)\big)_{t_0 \in {\cal T}_0,\omega \in \Omega}$.}\label{tableOrga}\vspace{-7mm}
\end{center}
\end{table}
The spectral values themselves  (for $\tau_1 = \tau_2 =\tau$), or  their real and imaginary parts (for~$\tau_1\neq\tau_2$) are represented via a continuous $(\tau_1,\tau_2)$-dependent color code, ranging from   cyan and  light blue (for {\it small} values)  to dark blue, yellow, orange, and red (for {\it large} values). 
As   explained below, this color code also has an interpretation in terms of significance of certain $p$-values. This latter interpretation requires a preliminary calibration  step, though. Indeed, being {\textquoteleft small\textquoteright},  for a $(\tau_1=\tau_2=\tau)-$periodogram value (which by nature is nonnegative real) does not  have the same meaning as  being {\textquoteleft small\textquoteright}  for  the imaginary or the real part of some $(\tau_1\pr,\tau_2\pr)-$cross-periodogram (for which negative values are possible):  in order to make inter-frequency comparisons possible, a meaningful color code therefore has to be $(\tau_1,\tau_2)$-specific. For this purpose we  introduce a distribution-free simulation-based calibration that fully exploits the properties of copula-based quantities.

To explain the idea behind this calibration  step, consider plotting, for some sub\-set~${\cal T}_0\times \Omega$  (with ${\cal T}_0 \subset \{1,...,T\}$ and~$ \Omega \subset (0,\pi)$), a collection  $\big(\Re \mathfrak{\hat f}_{t_0,T}(\omega,\tau_1,\tau_2)\big)_{t_0 \in {\cal T}_0,\omega \in \Omega}$ of  the real parts {(the imaginary parts are dealt with  exactly the same way)} of estimators    computed from the  realization~$X_1,....,X_T$ of some time series of interest. Assume that a bandwidth $B_n$ and a window length $n$ are  used for the estimation. A color is then assigned to each  value of~$\Re \mathfrak{\hat f}_{t_0,T}(\omega,\tau_1,\tau_2)$ along the following steps:

\begin{description}
\item[(i)] simulate $M = 10^4$ independent realizations $(U_{1,m},\ldots ,U_{n,m})$, $m=1,...,M$ of an i.i.d.\ sequence of random variables of length $n$ (uniform over $[0,1]$, for instance -- but, our method being  distribution-free, this is not required);
\item[(ii)] for each of those $M$ realizations, compute the estimator $\mathfrak{\hat f}_{t_0,T}^{U,m}(\omega,\tau_1,\tau_2)$ of the local spectral density based on the same bandwidth $B_n$; 
note that the number~$n$ of observations in each replication 
  equals the window length used for our original collection;
\item[(iii)] define,  for each $m=1,...,M=10^4$,  the quantities
$ Q_{\max}^m(\tau_1,\tau_2) := \max_{\omega} \Re \mathfrak{\hat f}_{t_0,T}^{U,m}(\omega,\tau_1,\tau_2)$ {and} 
$ Q_{\min}^m(\tau_1,\tau_2) := \min_{\omega} \Re \mathfrak{\hat f}_{t_0,T}(\omega,\tau_1,\tau_2)$; 
 obtain the empirical $99.5\%$ quan\-tile $q_{\max}(\tau_1,\tau_2)$ of $(Q_{\max}^{m}(\tau_1,\tau_2))_{m=1,...,M}$,  and the empirical $0.5\%$ quantile $q_{\min}(\tau_1,\tau_2)$ \linebreak of $(Q_{\min}^m(\tau_1,\tau_2))_{m=1,...,M}$, respectively.
\end{description}


The color palette then  is  set  as follows: all points $(t_0 ,\omega )  \in {\cal T}_0\times \Omega$ with $ \Re \mathfrak{\hat f}_{t_0,T}(\omega,\tau_1,\tau_2)$ value in~$[q_{\min}(\tau_1,\tau_2),q_{\max}(\tau_1,\tau_2)]$ receive dark blue color. Next, letting \vspace{-2mm}
\begin{align*}
v_{\min}(\tau_1,\tau_2) &:= \min(\min_{t_0,\omega}\Re \mathfrak{\hat f}_{t_0,T}(\omega,\tau_1,\tau_2),q_{\min}(\tau_1,\tau_2) - (q_{\max}(\tau_1,\tau_2)-q_{\min}(\tau_1,\tau_2))),\\
v_{\max}(\tau_1,\tau_2) &:= \max(\max_{t_0,\omega}\Re  \mathfrak{\hat f}_{t_0,T}(\omega,\tau_1,\tau_2),q_{\max}(\tau_1,\tau_2) + (q_{\max}(\tau_1,\tau_2)-q_{\min}(\tau_1,\tau_2))), 
\end{align*}\vspace{-9mm}

\noindent all points $(t_0 ,\omega ) $ for which $ \Re \mathfrak{\hat f}_{t_0,T}(\omega,\tau_1,\tau_2)$ lies in  the interval $[v_{\min}(\tau_1,\tau_2),q_{\min}(\tau_1,\tau_2)]$ receive a color ranging, according to a linear scale,   from cyan to light  and  dark blue, while  the colors for the interval~$[q_{\max}(\tau_1,\tau_2),v_{\max}(\tau_1,\tau_2)]$ similarly  range  from dark blue  to yellow and  red.

%

%
%

All our time-frequency heat diagrams thus have the following interpretation. For each given choice of $(\tau_1,\tau_2)$ and a timepoint~$t_0$, the probability, under the hypothesis of (strong) white noise, that the real (resp., the imaginary) part at time~$t=t_0$ of the smoothed $(\tau_1,\tau_2)$-time-varying periodogram lies entirely in the dark blue area is approximately~$0.01.$ Hence, the presence of light blue, cyan or orange-red zones in a diagram indicates a significant (at probability level $1\%$) deviation from white noise behavior. The location of those zones moreover tells us where in the spectrum, and when in the period of observation, those significant deviations take place, along with an   evaluation of their magnitude. The correspondence between the actual size of the estimate and the colors used is provided by the color scale on the right-hand side of each diagram. Note that here and in the sequel, we use the terminology 'white noise' to denote i.i.d.\ (and not just uncorrelated)  variables. 

This calibration method yields a universal distribution-free and model-free color scaling   which also  provides (as far as  dark blue regions are concerned)  a hypothesis testing  interpretation of the results. The same color code is used for the empirical analyses in Sections~\ref{SecSandP} and~\ref{SecHPB}, as well as for the simulations in Section~\ref{SecSimu}. 
Currently, an R-package containing the codes  used here   is in preparation (a preliminary version is available 
 upon request).

\subsection{Simulations}\label{SecSimu} 
This section provides a numerical illustration  of the performance of the new estimators of the time-varying copula spectral densities in two time-varying models that have been considered elsewhere in the literature. For both models, six time-frequency heat plots, labeled (a)-(f), of time-varying copula spectral densities are provided, for each  combination of the quantile levels $0.1$, $0.5$, and $0.9$, using the color code described in Section~\ref{Seccalibr}:
\smallskip

\begin{description}
\item[(a)] the actual time-varying copula spectral densities  and
\item[(b)-(f)] the local lag-window estimators of the copula spectral densities for different window length $n.$
\end{description}

Currently, we do not have simple closed-form expressions for the actual spectra, and we doubt such expressions are possible (but for the theoretical definition~(\ref{def:tvf})). This is in contrast with classical $L^2$ spectral analysis where, at least for linear processes, explicit representation  for  the spectra are readily available. Such  lack of simple analytic expressions is   not   surprising since, even for linear processes, the impact of  the linear representation coefficients on  joint distributions (as opposed to covariances) is bound to be quite complicated, and  crucially depends on   innovation densities. From a practical point of view, this is not a major drawback, though, as for any given linear representation very good approximations of the copula spectra can be obtained within a few minutes via  simulations. The actual copula spectral densities in (a) were obtained by simulating, for each~$t_0$ in~$ {\cal T}_0$, $R=1000$ independent replications, all of  length~$2^{11}$,    of the strictly stationary approximation~$(X^{t_0/T}_t)_{t=1,...,2^{11}}$, computing the corresponding lag-window estimators~$\hat{ \mathfrak{f}}^r_{t_0,T}(\omega,\tau_1,\tau_2)$,  say, for~$r=1,...,R$, and averaging them   (over $r=1,...,R$)  for each fixed $(t_0 ,\omega ) \in {\cal T}_0\times \Omega$.   

The estimators in (b)-(f) are computed from one realization, of length $T=2^{13}$, of the (nonstationary) process under consideration with a bandwidth $B_n = 10$ and local window lengths $n = 128, 256, 512, 1024, 2048$. Additional examples with time series of length $T = 2^{11}, 2^{12}$ are available in Section \ref{app:shortts} of  the online appendix. Our findings indicate that, for   shorter time-series lengths, 
   estimating `fastly changing'  dependence structures may become difficult. If the changes are very smooth, as in the QAR example of Section~\ref{sec:p4}, the results for short time series are still reasonable.  
For $K$,  we used the Parzen window 
\[K(u) = ( 1-6u^2+6|u|^3 )\I{ |u| \leq 0.5} +    2(1-|u|)^3 \I{ 0.5 \leq |u| \leq 1}.\] In each case, the sets~${\cal T}_0$ and $\Omega$ were chosen as~${\cal T}_0 := \{32k|k= \lceil n/2 \rceil ,\dots,\lfloor T-n/2 \rfloor\}$ and~$\Omega := \{2\pi j/n|j=1,...,(n-2)/2\}.$


%
%
%


\subsubsection{Cauchy tvAR(2) } \label{sec:p2}
In Figure~\ref{Figb}, we display heatmapss for a time-varying AR(2) process with equation\vspace{-1mm}
\begin{equation} \label{extvAR}
X_{t,T} = 1.8\cos(1.5-\cos(2\pi t/T))X_{t-1,T} - 0.81X_{t-2,T} + Z_t\vspace{-1mm}
\end{equation}
and i.i.d.~noise~$Z_t$ with Cauchy distribution. Its strictly stationary approximation at $t_0=\vartheta T$, for $0\leq \vartheta \leq 1$,~is\vspace{-2mm}
\begin{equation}\label{exAR}
X_t^{\vartheta} = 1.8\cos(1.5-\cos(2\pi \vartheta))X_{t-1}^{\vartheta} - 0.81X_{t-2}^{\vartheta} + \zeta_t\vspace{-1mm}
\end{equation}
where the $\zeta_t$'s are  i.i.d.~noise~ with the same Cauchy density as the $Z_t$'s.

The form of the equation is taken from \cite{dahlhaus2012}, where we  replaced the Gaussian innovations with  Cauchy ones, thus violating  the moment assumptions of classical spectral analysis. The resulting  process exhibits a time-varying periodicity which is clearly visible in the heat diagrams associated with the real parts of  its time-varying copula spectral densities, displayed in the lower triangular parts  of Figures~\ref{Figb}(a)-(f).
The imaginary parts of the spectra are shown in the upper triangular parts of  the same figures; note that, due to time-irreversibility (see Hallin et al. 1988), those imaginary parts exhibit significant yellow regions in  the actual spectral density (a). The peaks are, however, very narrow,   thus quite  difficult to estimate, and essentially disappear in the estimated versions (b)-(f).  The proposed lag-window estimator nevertheless is able to recover the structure of the spectral densities over a broad range of window lengths. 

Also note the significant  peak around zero appearing in the diagrams associated with extreme quantiles ($\tau_1 , \tau_2 = 0.1$ and $0.9$),  indicating persistence in tail events---a phenomenon that totally escapes  traditional analyses. 
The change over time is rather fast and therefore the influence of the window length on the estimator is clearly visible. A very short window length, like $n=128$ in $(b),$ makes it very difficult to reconstruct the copula spectral densities for the extreme quantiles $(\tau = (0.1,0.1) \text{ or } \tau = (0.9,0.9) )$ still the periodic peak remains quite significant, while a much  larger one ($n=2048$ in (f)) 
 one.      
 leads to a loss of details. 
The estimators remain stable, though,  
over a broad range of window lengths ($n=256 - 1024$).

\subsubsection{tvQAR(1)} \label{sec:p4} Figure~\ref{Figd} shows the same heat diagrams for the QAR(1) (Quantile Autoregression) model of order one \vspace{-2mm}
\[
X_{t,T} = [(1.9U_t - 0.95)(t/T) + (-1.9U_t + 0.95)(1-(t/T))]X_{t-1,T} + (U_t-1/2),\vspace{-2mm}
\]
where the $U_t$'s are    i.i.d.\  uniform over $[0,1]$ (see \cite{QAR}).  The corresponding strictly stationary approximation at $t_0=\vartheta T$, $0\leq \vartheta \leq 1$,  is \vspace{-1mm}
\begin{equation}
\label{locQAR}
X_t^{\vartheta}= [(1.9V_t - 0.95)\vartheta + (-1.9V_t + 0.95)(1-\vartheta)]X_{t-1}^{\vartheta} + (V_t-1/2)\vspace{-1mm}
\end{equation}
where the $V_t$'s are i.i.d.~uniform over $[0,1]$. The gradient of the   coefficient function in (\ref{locQAR})  changes slowly from $1.9\, U_t - 0.95$ to~$-1.9\, U_t + 0.95$, so that the spectral densities associated with the lower quantiles for small values of~$t_0/T$ are the same as those associated with the upper quantiles for $1-t_0/T$,  and vice versa.

This behavior, which cannot be detected via  classical spectral methods, is quite visible here. 
 Comparing the plots for $\tau = (0.5,0.1)$ and $\tau = (0.9,0.5)$, we see that the real parts reflect the behavior of the time-varying coefficient functions;   mirroring one of them at a vertical axis in $\vartheta = 0.5$ yields the other one. On the other hand, the imaginary parts are time-varying but stable over different quantile combinations.  

Comparing Figure~4 with Figure~3 reveals  completely different reactions to variations of  window lengths. 
   The tvAR$(2)$ case in Figure~3 indeed consists in a strong signal   rapidly  changing over time, whereas the signal in the tvQAR(1)  of Figure~4 is rather weak and therefore harder to detect, with, however, a much smoother evolution in time. As a consequence, a larger window length yields better results in the estimation of the time varying spectral densities. 
 The best results are obtained for  $n=2048$, 
and the estimator displays most of the details found in the actual  spectral density.

\begin{figure}[H]
\begin{minipage}[b]{.48\textwidth}
  \includegraphics[width = 72mm, height = 41mm]{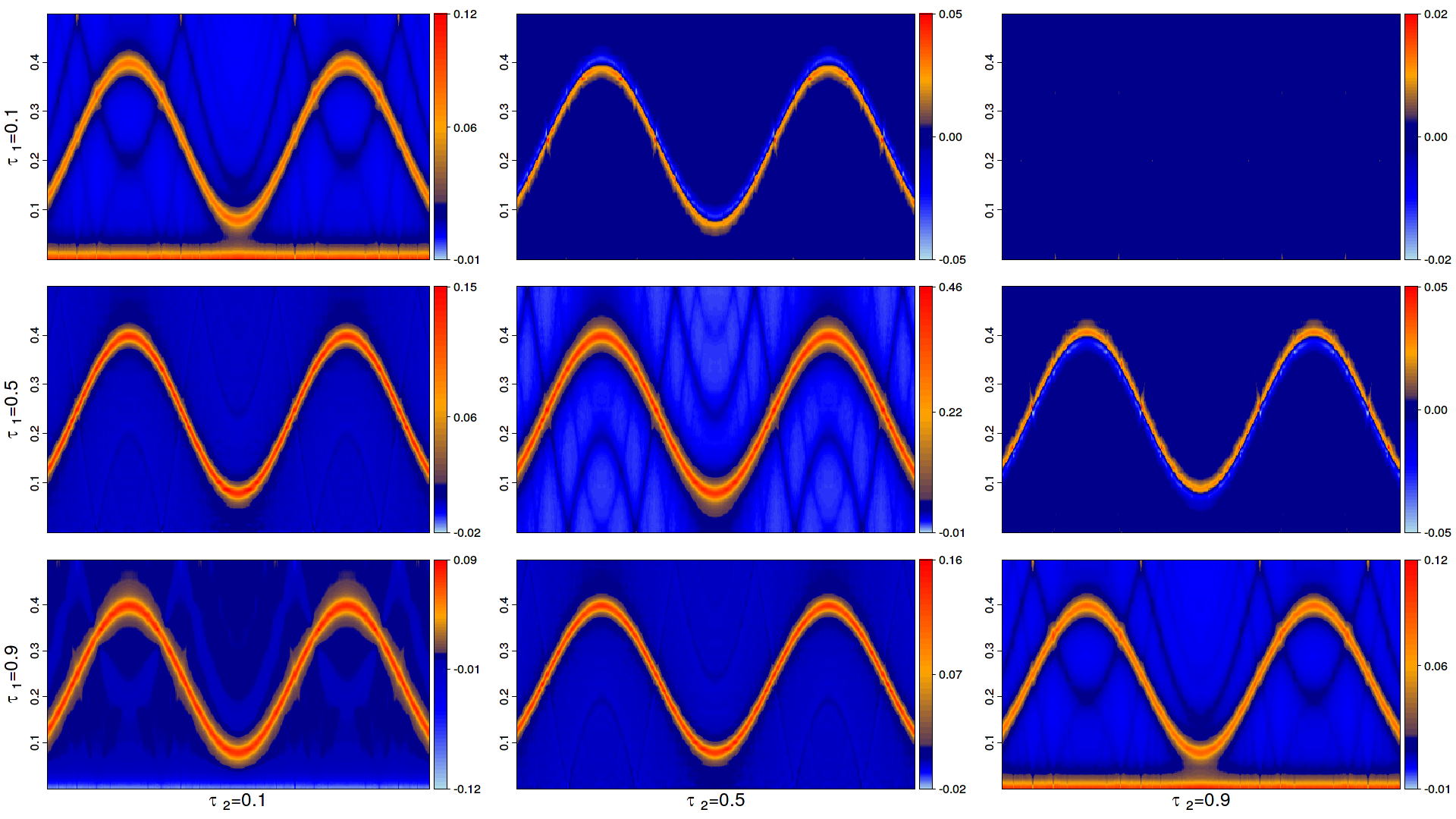}
\subcaption{\footnotesize Actual  copula spectral densities (simulated)}
\end{minipage}
\begin{minipage}[b]{.48\textwidth}
  \includegraphics[width = 72mm, height = 41mm]{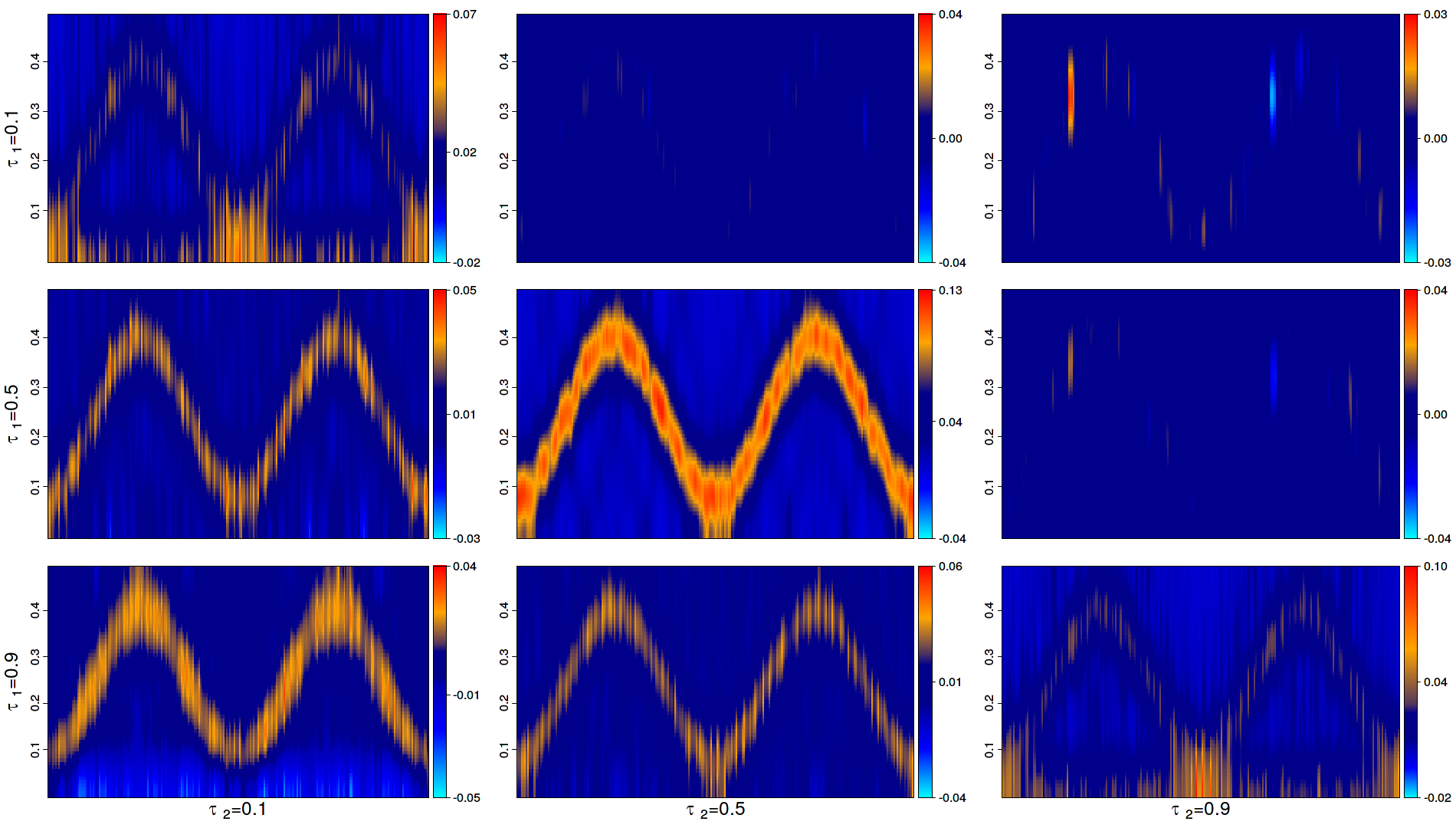}
\subcaption{\footnotesize Estimated copula spectral densities, $n = 128$}
\end{minipage}
\begin{minipage}[b]{.48\textwidth}
  \includegraphics[width = 72mm, height = 41mm]{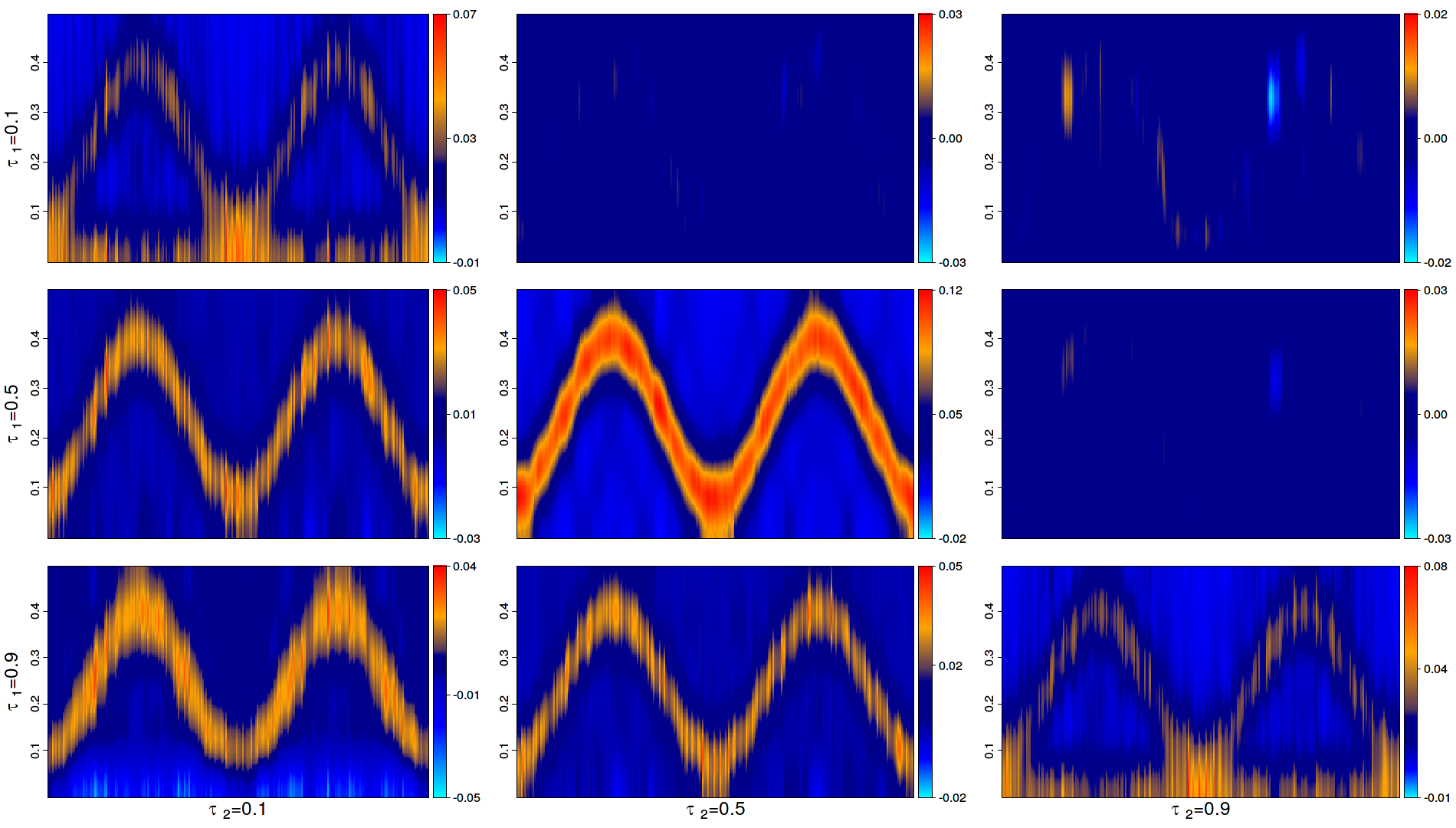}
\subcaption{\footnotesize Estimated copula spectral densities, $n = 256$}
\end{minipage}
\begin{minipage}[b]{.48\textwidth}
  \includegraphics[width = 72mm, height = 41mm]{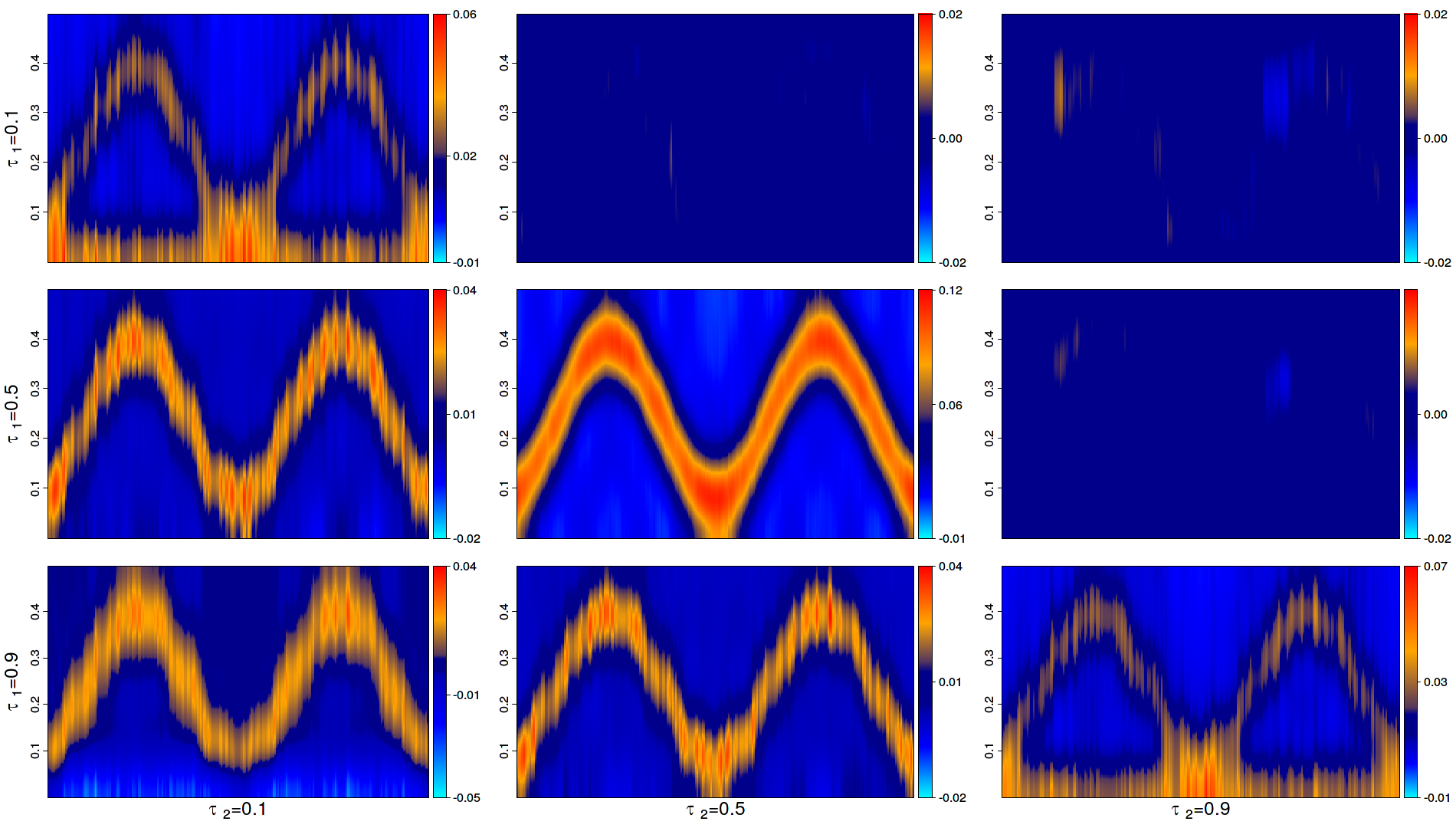}
\subcaption{\footnotesize Estimated copula spectral densities, $n = 512$}
\end{minipage}
\begin{minipage}[b]{.48\textwidth}
  \includegraphics[width = 72mm, height = 41mm]{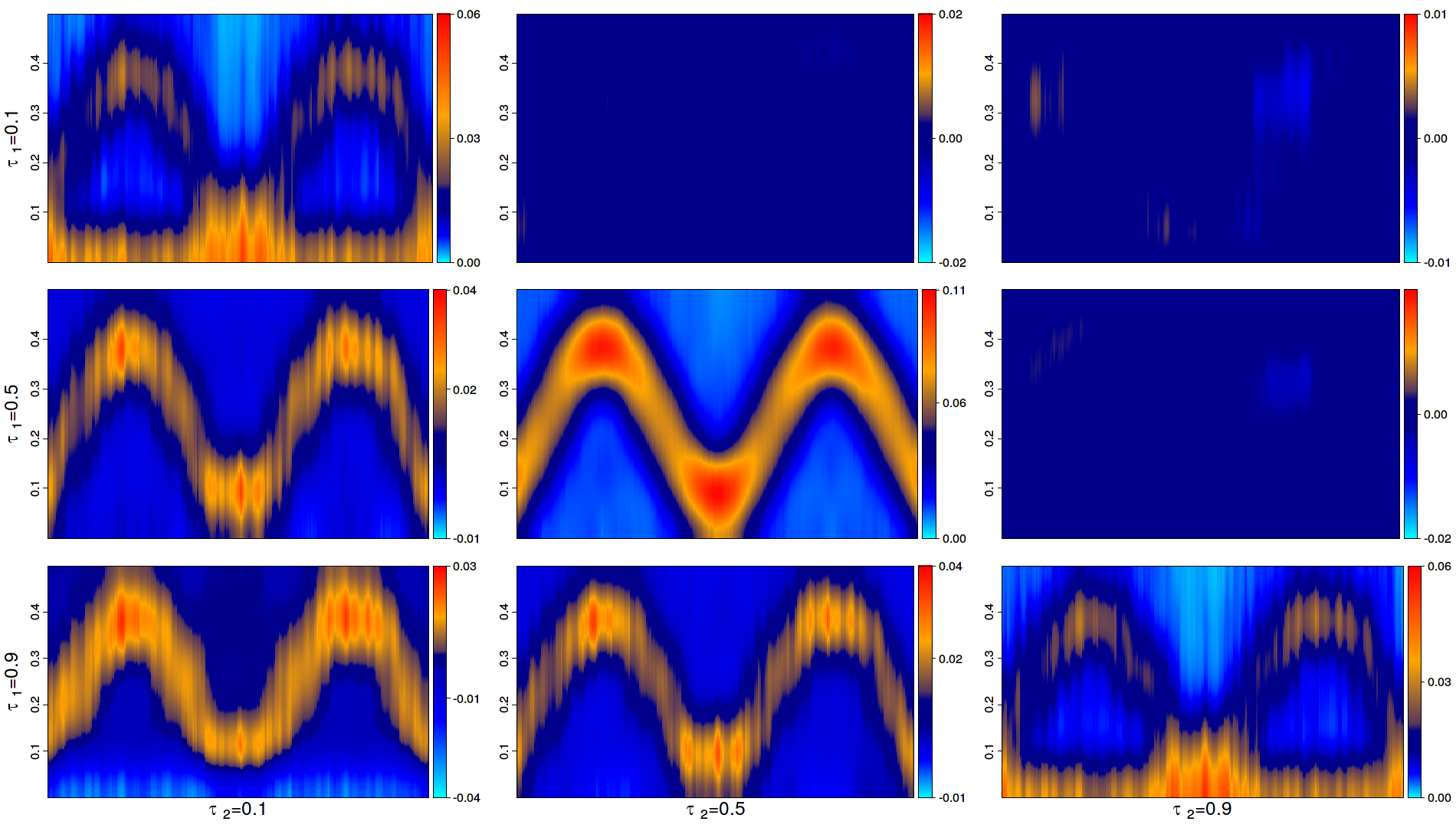}
\subcaption{\footnotesize Estimated copula spectral densities, $n = 1024$}
\end{minipage}\hspace{0.5cm}
\begin{minipage}[b]{.48\textwidth}
  \includegraphics[width = 72mm, height = 41mm]{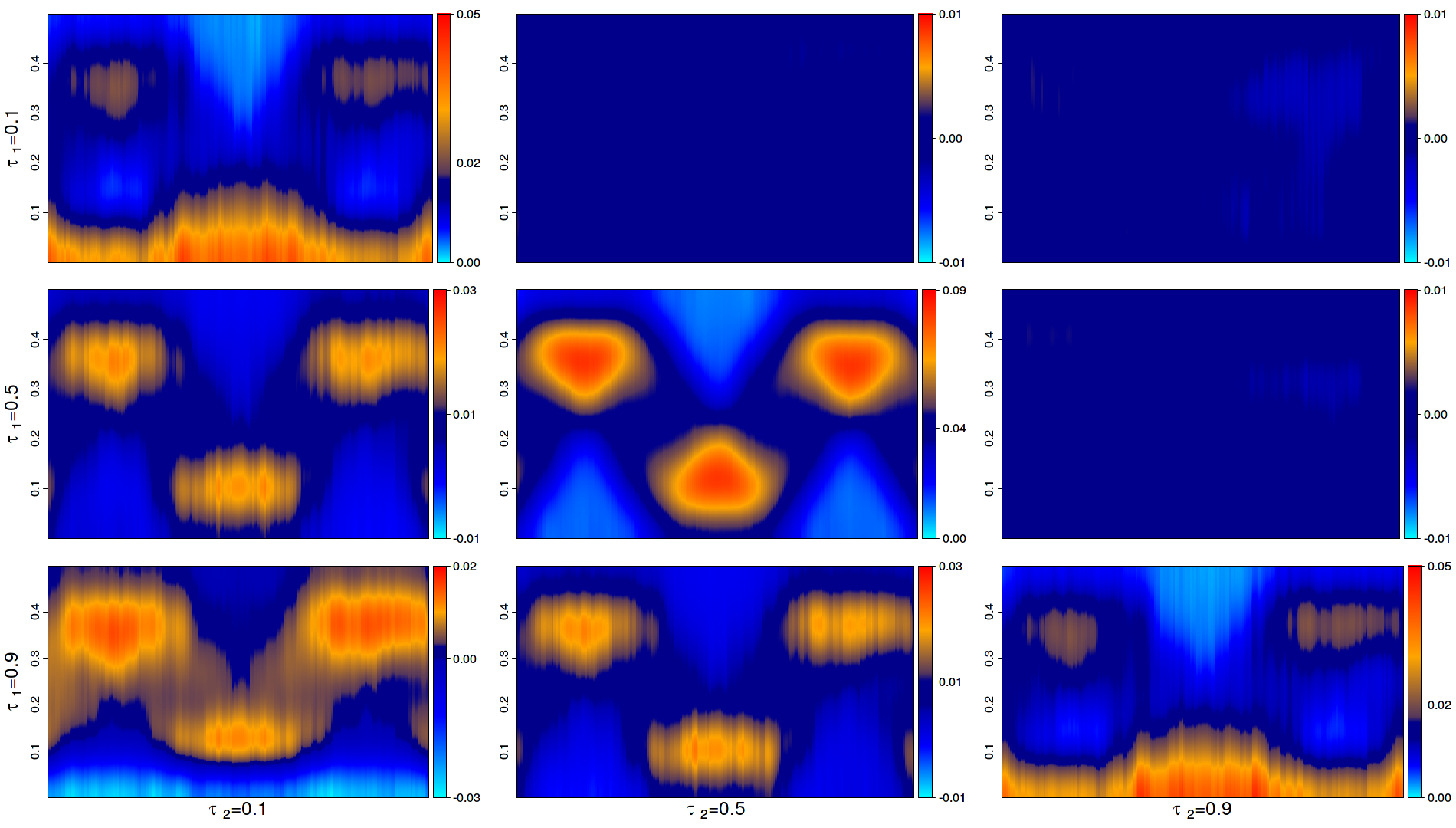}
\subcaption{\footnotesize Estimated copula spectral densities, $n = 2048$}
\end{minipage}

\caption{\small Heatmaps of the Cauchy time-varying AR(2) process described in Section \ref{sec:p2} and the corresponding estimators, 
for  various window lenghts. 
}\label{Figb}
\end{figure}
\newpage
\begin{figure}[H]
\begin{minipage}[b]{.48\textwidth}
  \includegraphics[width = 72mm, height = 41mm]{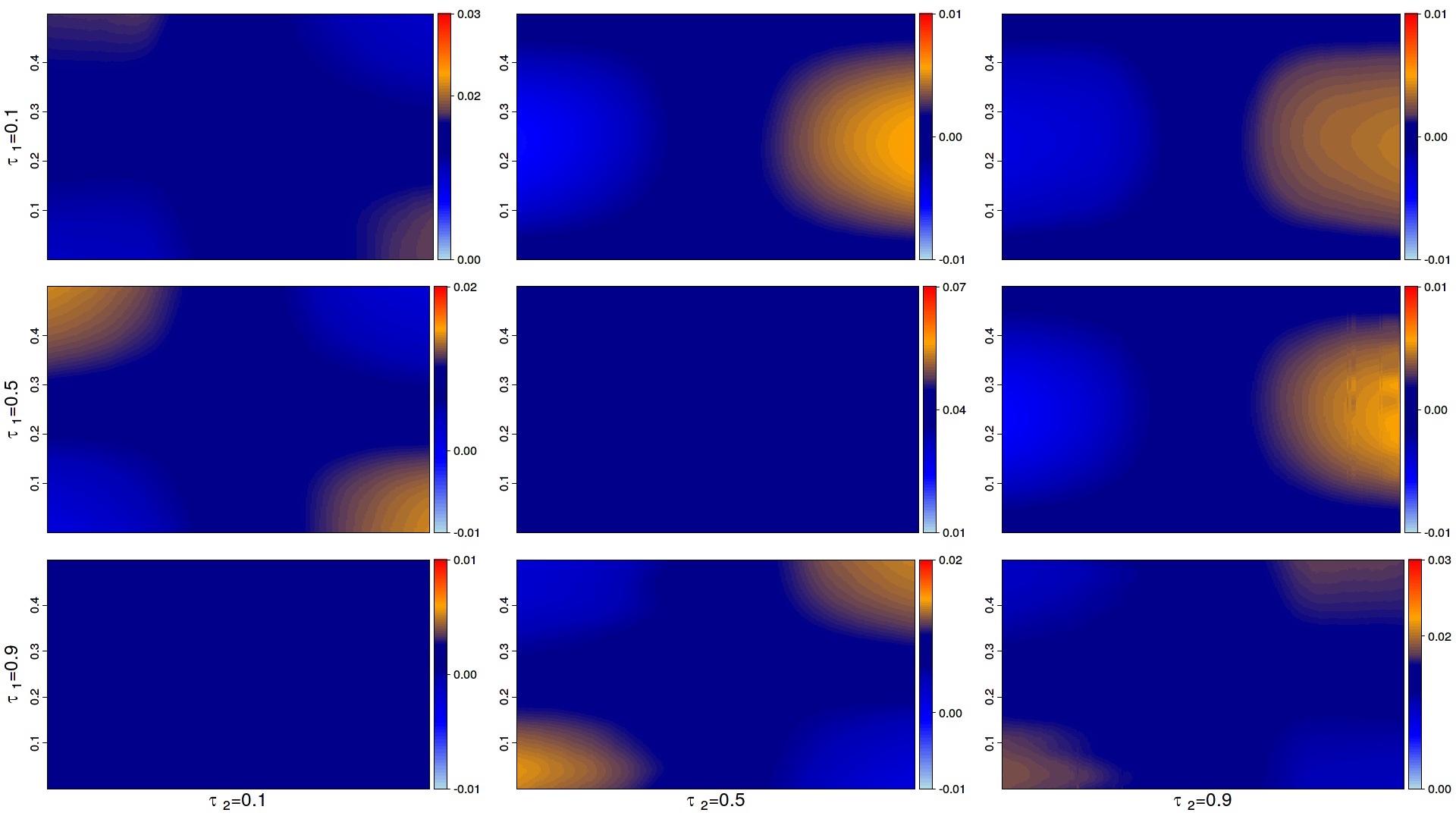}
\subcaption{\footnotesize Actual  copula spectral densities (simulated)}
\end{minipage}
\begin{minipage}[b]{.48\textwidth}
  \includegraphics[width = 72mm, height = 41mm]{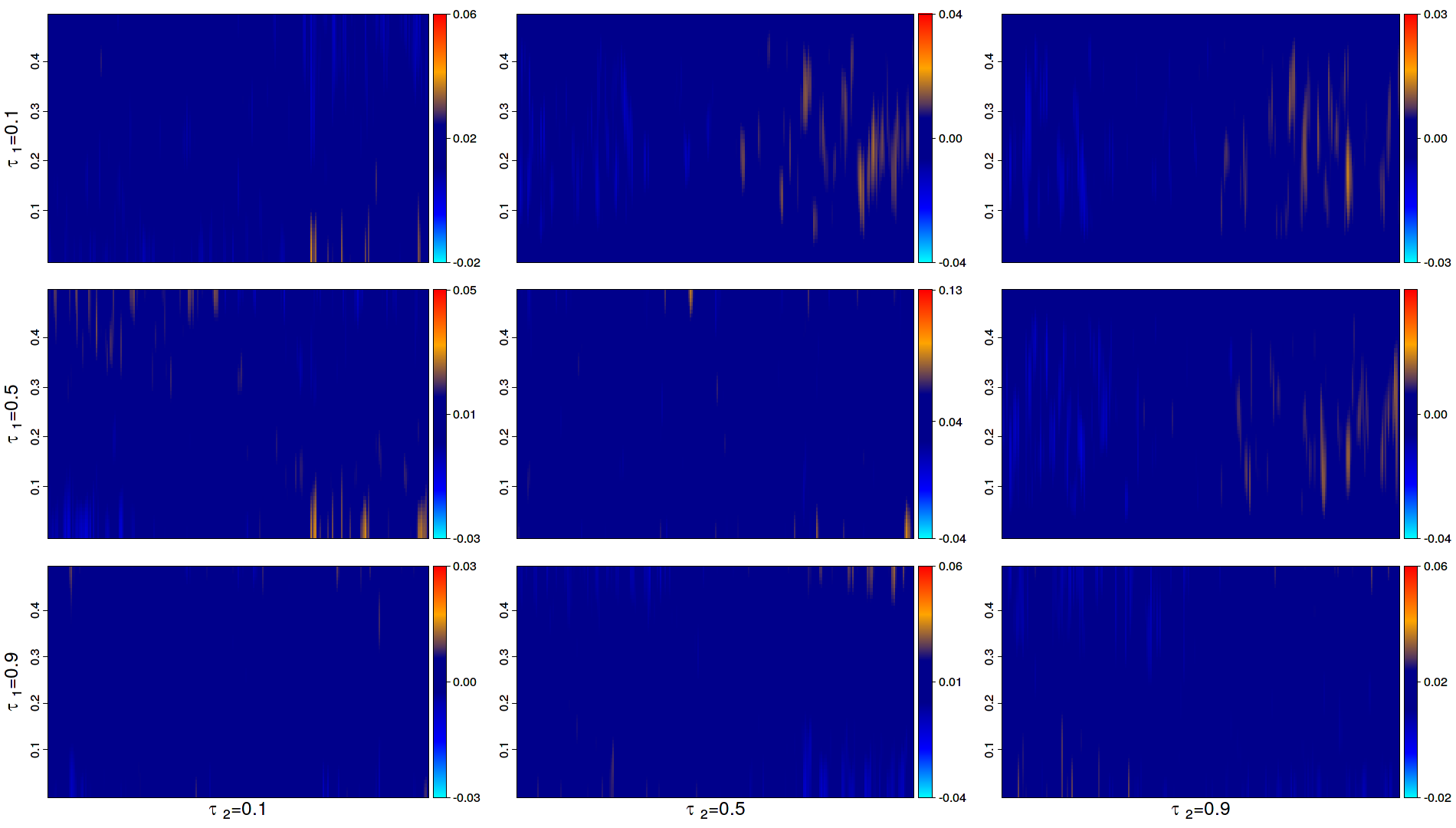}
\subcaption{\footnotesize Estimated copula spectral densities, $n = 128$}
\end{minipage}
\begin{minipage}[b]{.48\textwidth}
  \includegraphics[width = 72mm, height = 41mm]{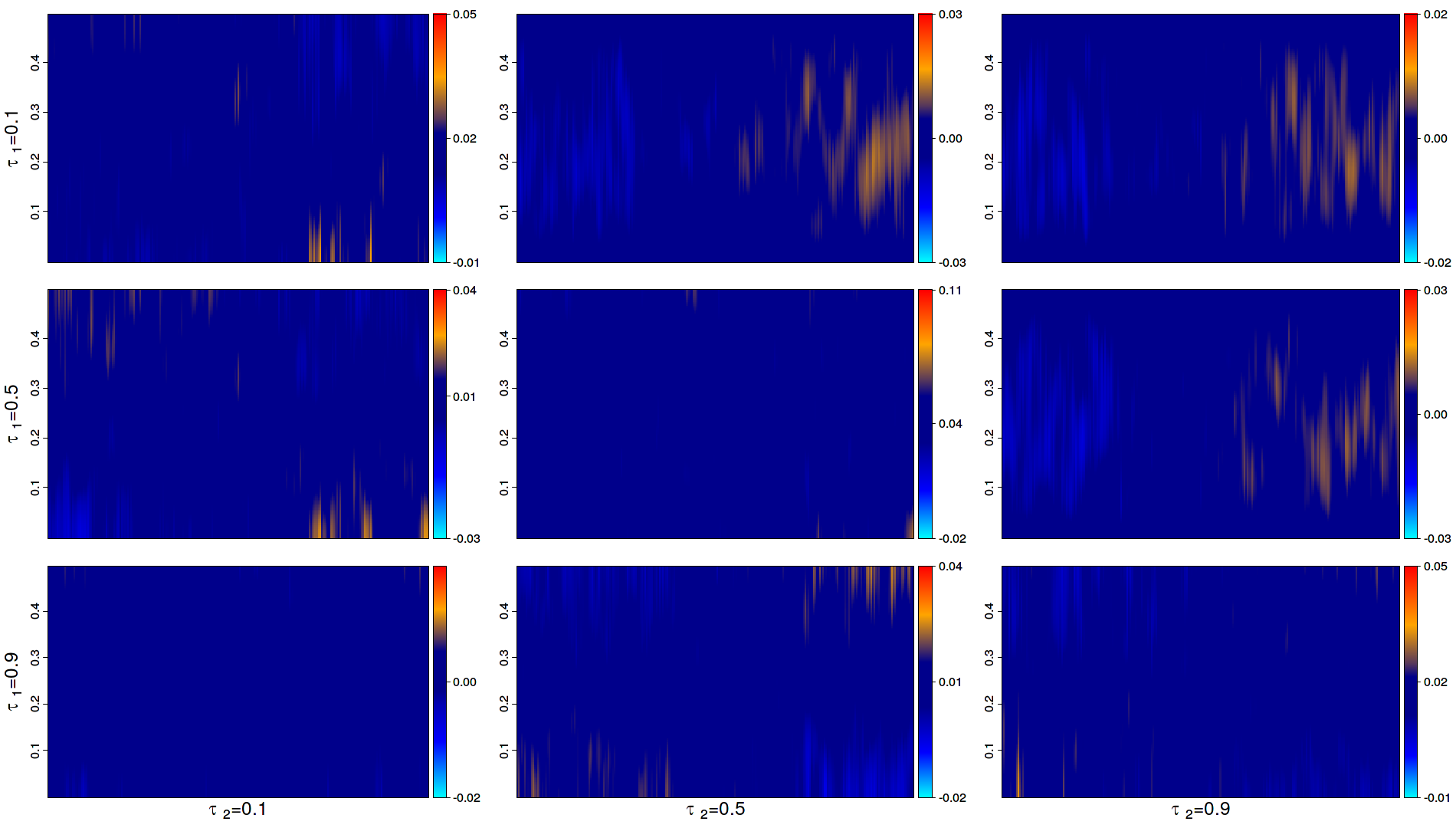}
\subcaption{\footnotesize Estimated copula spectral densities, $n = 256$}
\end{minipage}
\begin{minipage}[b]{.48\textwidth}
  \includegraphics[width = 72mm, height = 41mm]{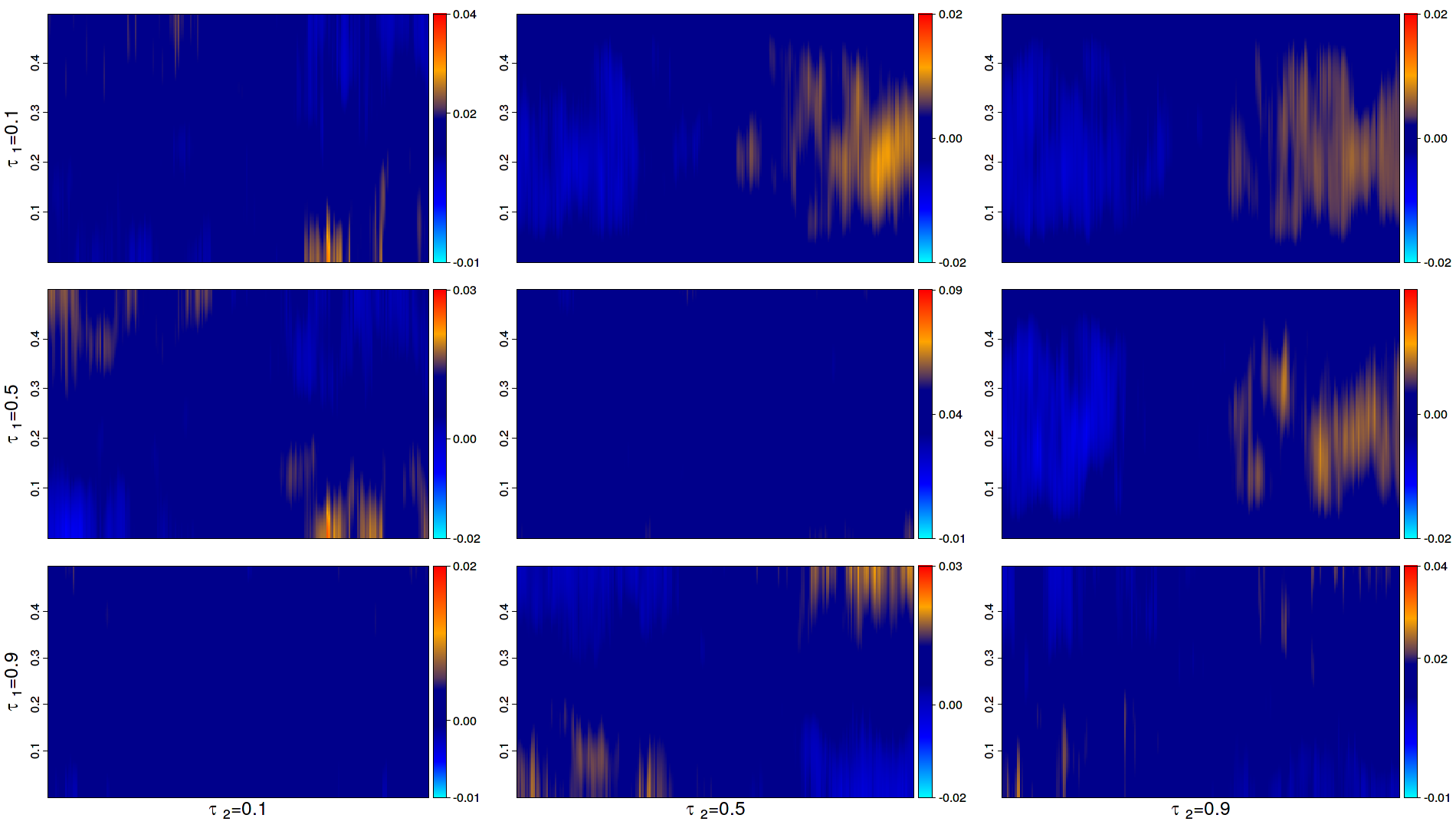}
\subcaption{\footnotesize Estimated copula spectral densities, $n = 512$}
\end{minipage}
\begin{minipage}[b]{.48\textwidth}
  \includegraphics[width = 72mm, height = 41mm]{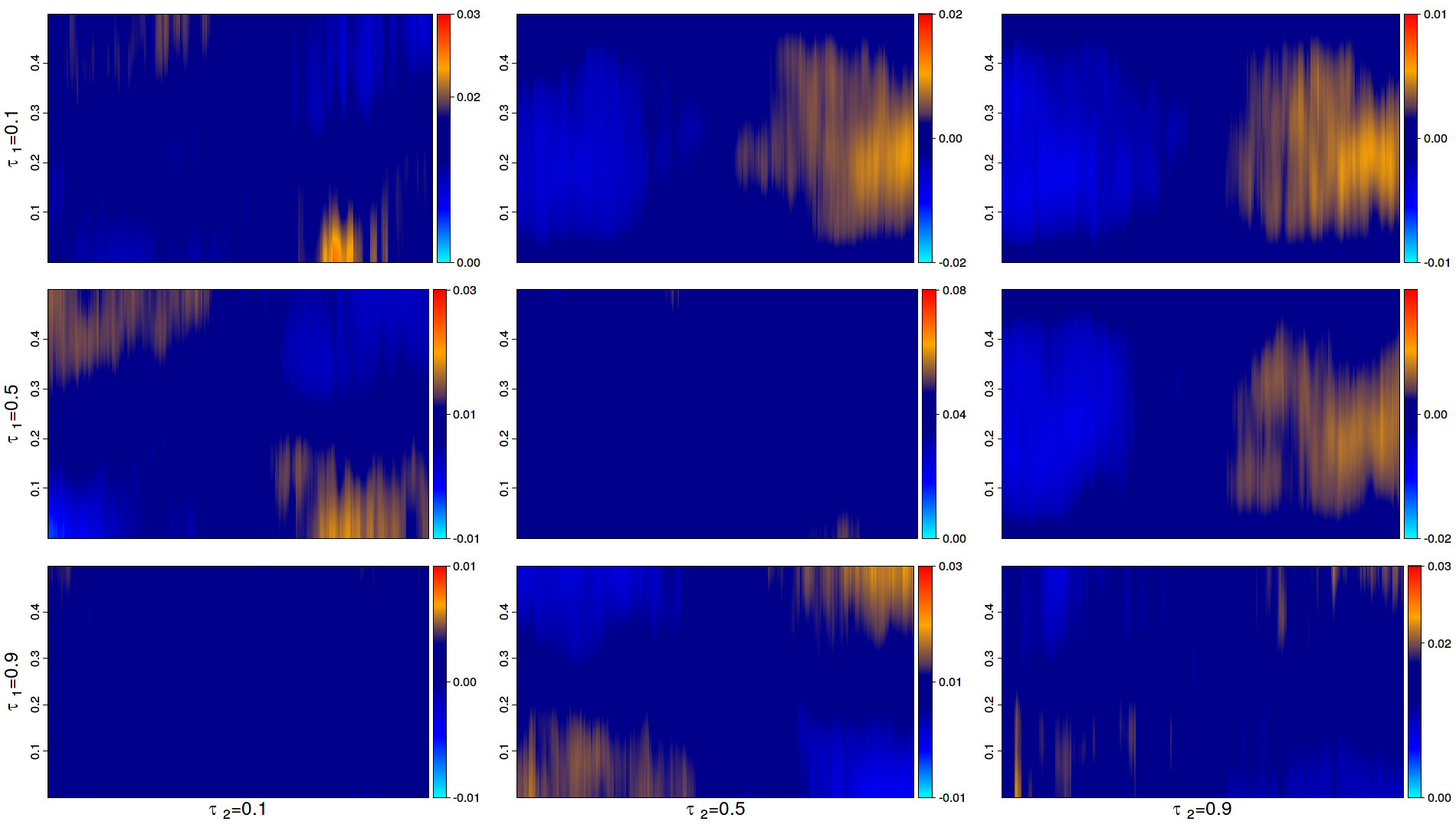}
\subcaption{\footnotesize Estimated copula spectral densities, $n = 1024$}
\end{minipage}
\hspace{5mm}
\begin{minipage}[b]{.48\textwidth}
  \includegraphics[width = 72mm, height = 41mm]{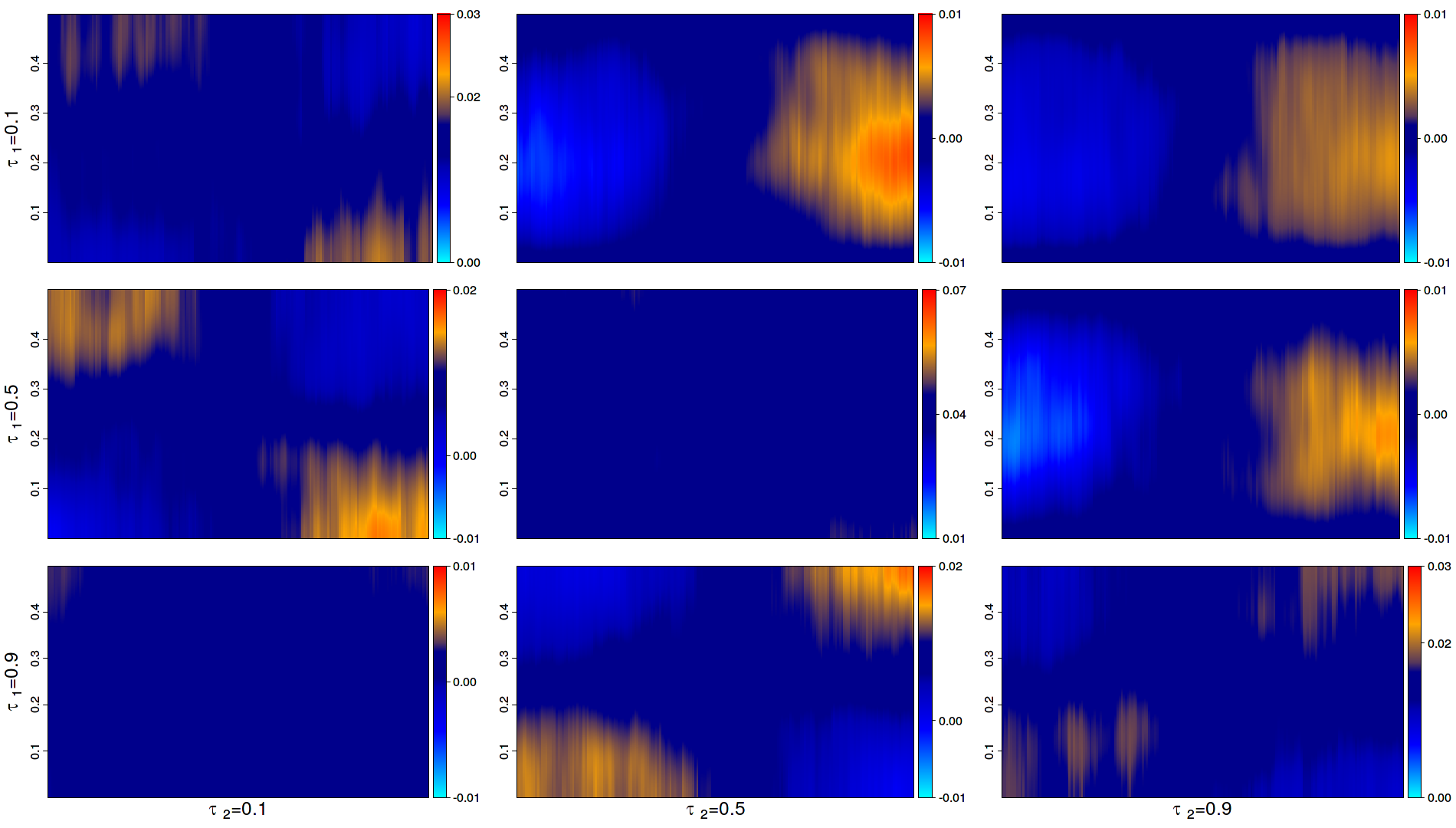}
\subcaption{\footnotesize Estimated copula spectral densities, $n = 2048$}
\end{minipage}
\caption{\small Heatmaps of the time-varying QAR(1) process described in Section \ref{sec:p4} and the corresponding estimators, 
for  various  window lenghts.
\vspace{-5mm}
}\label{Figd}
\end{figure}


\subsection{Standard \& Poor's 500 }\label{SecSandP}

\begin{figure}[t]
\centering
   \includegraphics[width = \textwidth]{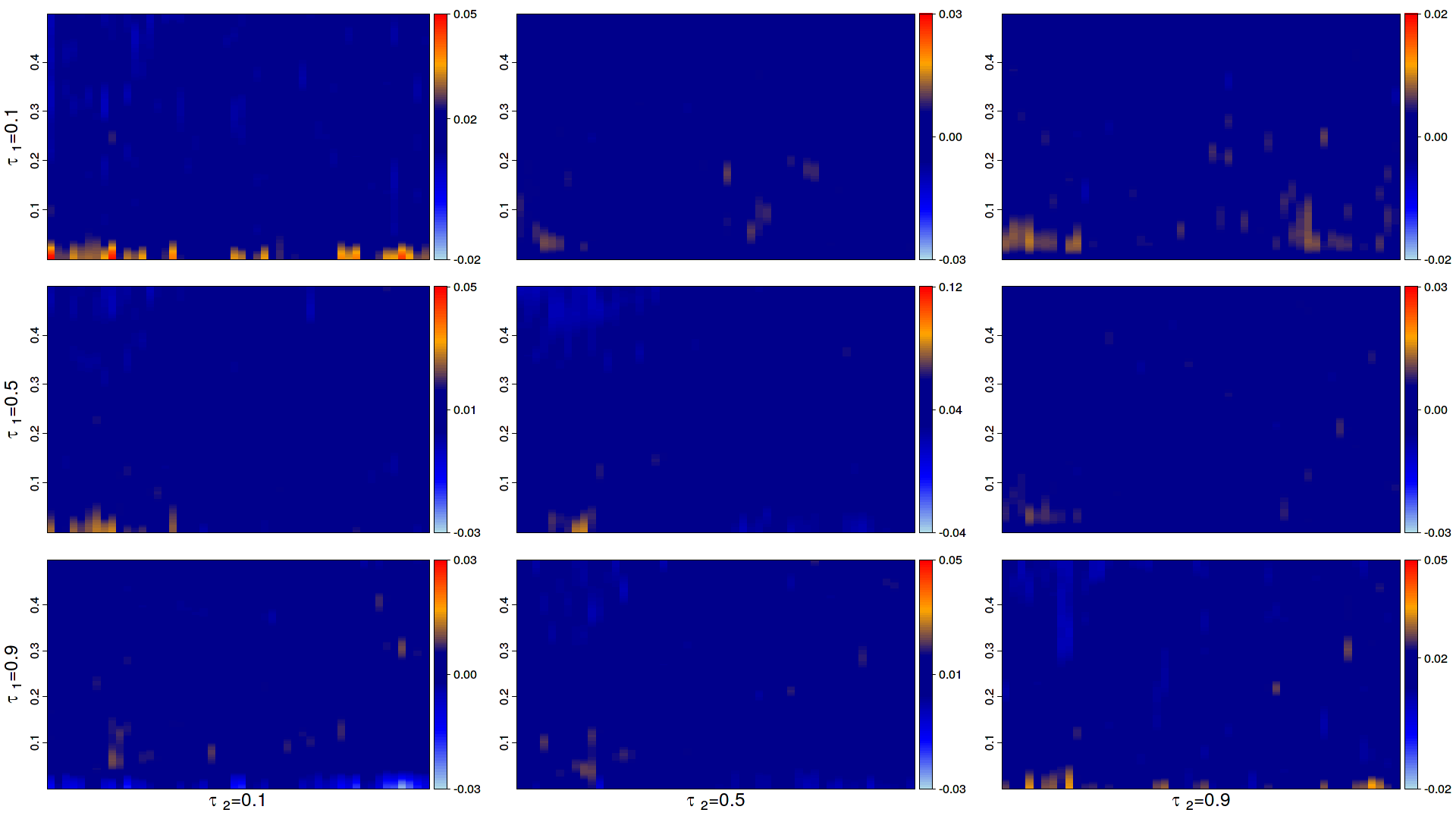}
\caption{\small Time-frequency heatmaps of the quantile lag-window estimator for the log-returns from the S\&P500 between 1962-2013 for quantile levels $0.1,0.5$ and $0.9$. The vertical axis represents  frequencies ($0<\omega/2\pi <0.5$) and the horizontal axis is time ($1\leq t\leq 12992$). The plots are organized  as explained in Section \ref{Seccalibr}; the color code is provided along the right-hand side of each  figure.
\vspace{-3mm} }
\label{figSPgen}
\end{figure}

   \begin{figure}[htbp]
\centering
  \includegraphics[width = 0.49\textwidth]{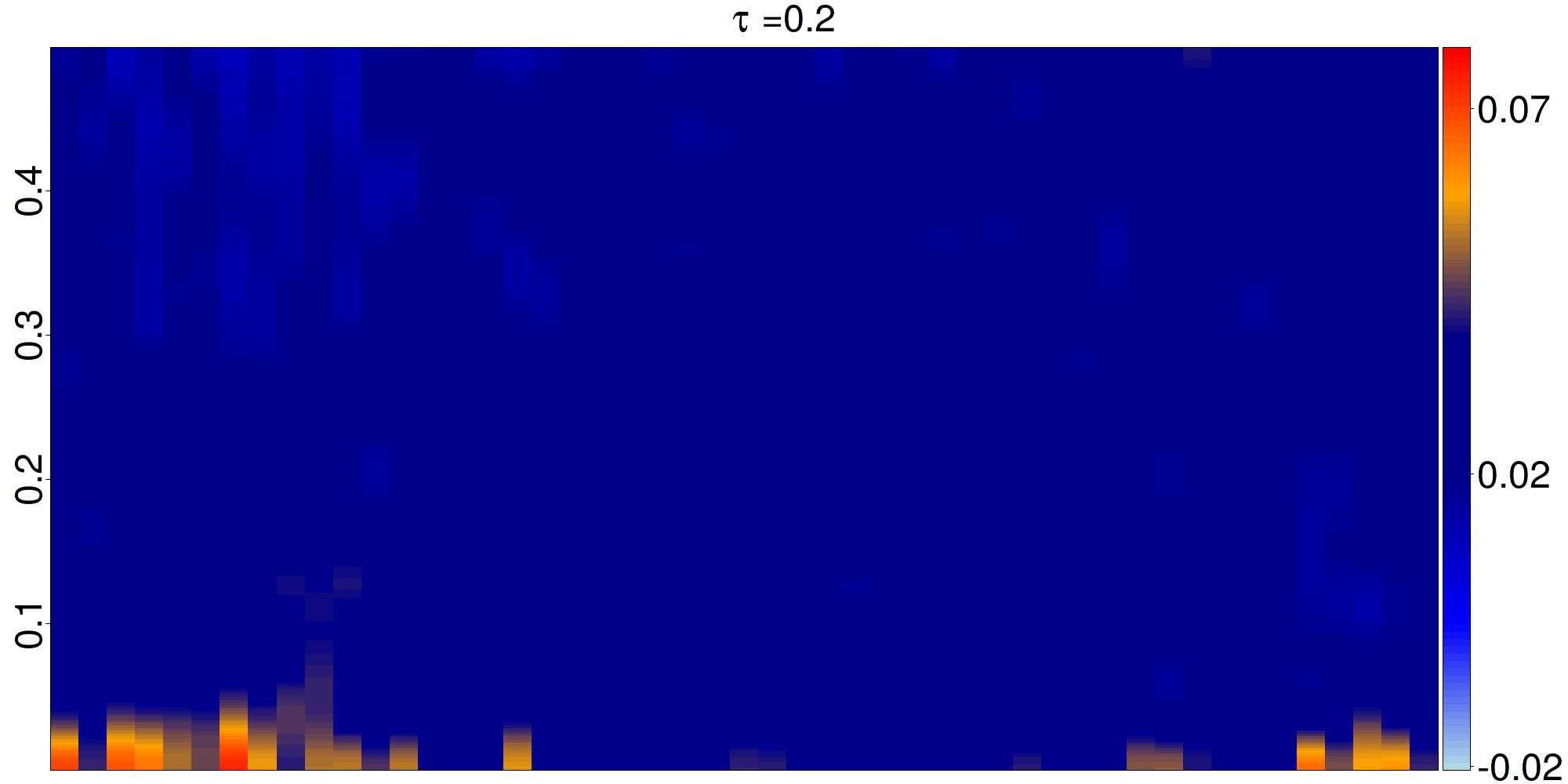}
  \includegraphics[width = 0.49\textwidth]{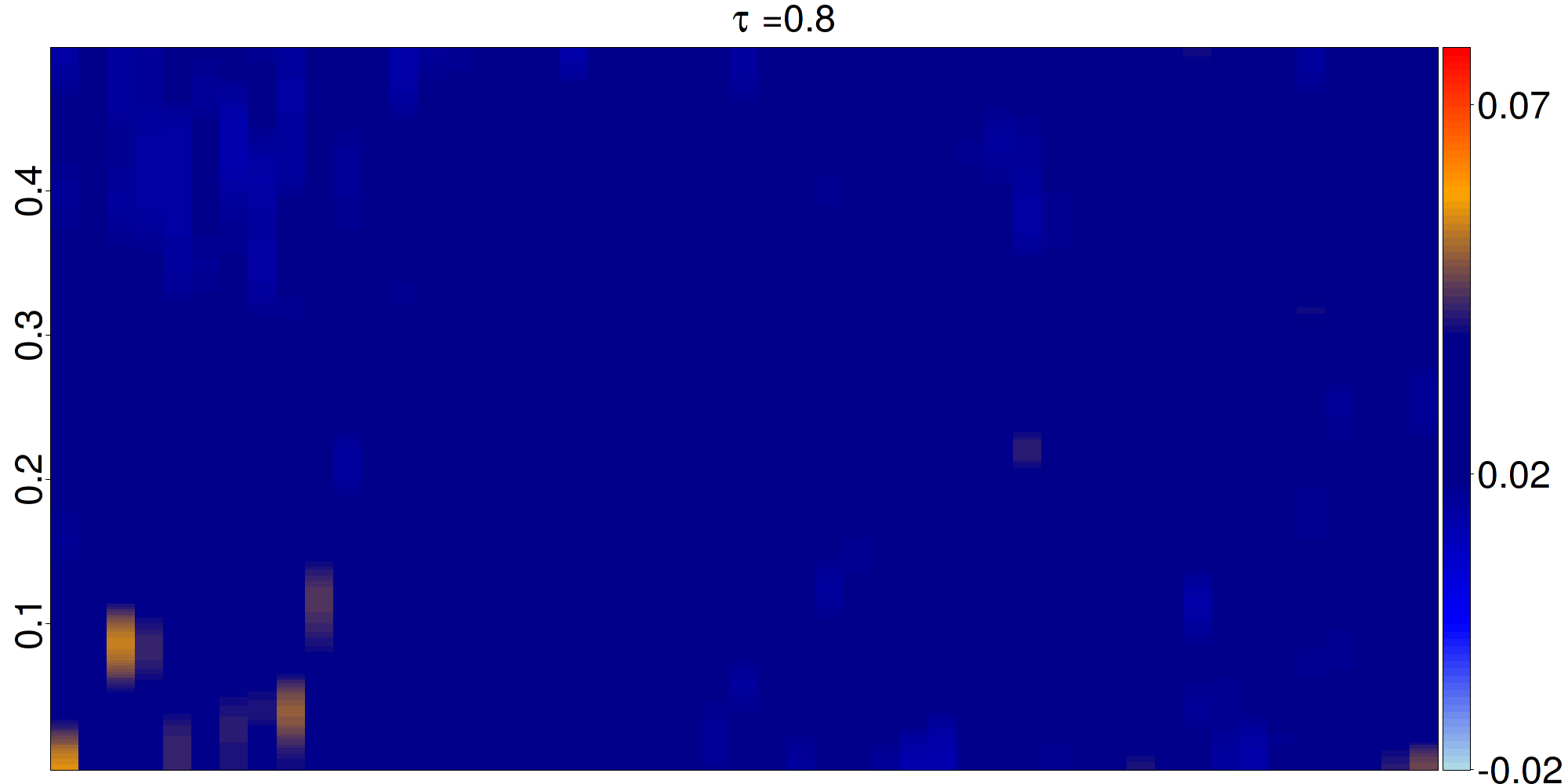}
\caption{\small Time-frequency heatmaps of the quantile lag-window estimators for $\tau_1 = \tau_2  \in \{0.2,0.8\}$ (no imaginary parts, thus).  The vertical axis represents the frequencies   ($0<\omega/2\pi <0.5$) and the horizontal axis is time ($1\leq t\leq 12992$); for each value of $t$, a periodogram is plotted against frequencies via the color code provided along the right-hand side of each  figure.
\vspace{-3mm}}
\label{SPFig4}
\end{figure}

\begin{figure}[t]
    \centering
   \includegraphics[width = 1\textwidth]{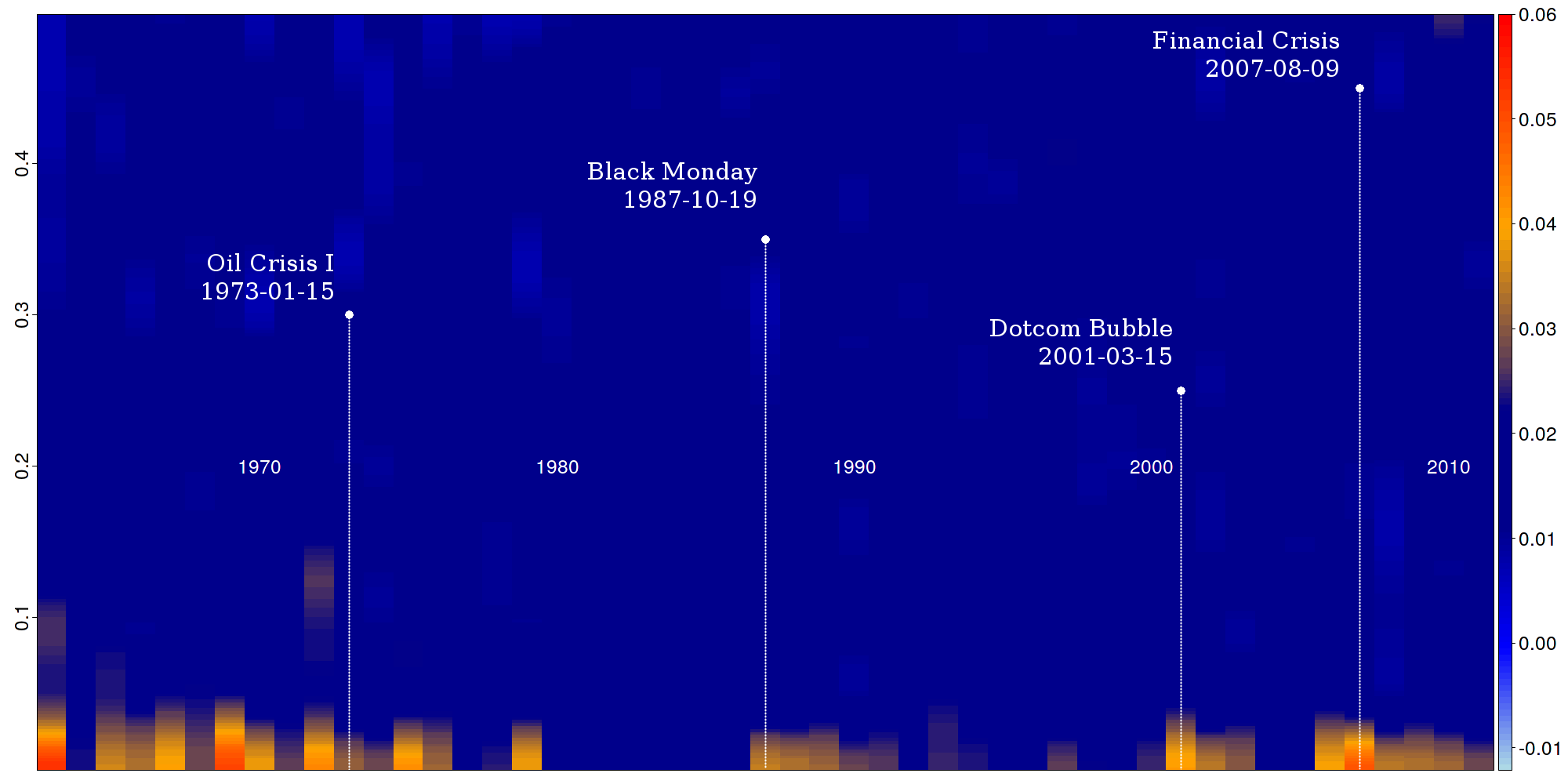}
    \caption{{\small The  $\tau_1 = \tau_2 = 0.1$ local lag-window estimator of Figure~\ref{figSPgen}; vertical dashed lines indicate historical financial crises, namely the Oil Crisis of 1973, the Black Monday (19.10.1987) which took place during the Savings and Loan Crisis in the USA, the bursting of the dot-com bubble in 2001 (followed by the early 2000s recession), and the 2007-2012 financial crisis. }}
\label{fig2}
    \end{figure}\vspace{-1mm}

We now turn back to the S\&P500 index series already considered in the introduction, with~$T=12992$ daily observations from 1962 through~2013 (differences of the logarithms of daily opening and closing prices for about 52 years). We applied the same estimation method as above, 
  with the same window function as described in Section~\ref{SecSimu},  
 with bandwidth~$B_n = 25$,
 window length~$n=512$,  and the sets \mbox{${\cal T}_0 = \{256 + 256j|0\leq j \leq 49\}$} and~$\Omega = \{2\pi j/256|j=0,...,255\}.$ 
 Calibration   was performed as explained in Section~\ref{Seccalibr}. The resulting heatmaps  are shown in the heatmaps of Figure~\ref{figSPgen}.

Whereas the central heatmaps ($\tau_1=\tau_2=0.5$)   are pretty flat (uniform dark blue) with the exception of some deviations from white noise behavior limited to the early seventies, the more extreme ones ($\tau_1=\tau_2=0.1$ and~$0.9$) suggest an alternance of high low-frequency spectral densities  (yellow and red) and perfectly  {\textquoteleft flat\textquoteright} (dark blue) periods.  A closer analysis   reveals that  those periods of strong dependence in the tails typically correspond to well-identified crises and booms (see below for details). Another interesting observation is the marked asymmetry between the time-varying spectra associated with the left ($\tau = 0.1$) and right ($\tau = 0.9$) tails, which can be interpreted in terms of prospect theory (see \cite{kahneman1979}). That asymmetry is  confirmed by comparing the estimated spectra for~$\tau = 0.2$ and~$\tau = 0.8$ (see Figure~\ref{SPFig4}); again, it cannot be detected by covariance-based methods. 
Inspection of local stationary copula-based spectra thus suggests that the 
 S\&P500 series is not as close to white noise as claimed. However, it takes a combination of quantile-related and local stationarity tools to bring some evidence for that fact.

 We now take  a closer look at deviations from the white noise behavior (dark blue) which  are particularly  visible in the diagrams associated with tail quantile levels. Concentrating on the $\tau_1=\tau_2 =0.1$ case, closer inspection of the diagram reveals a relation between low-frequency spectral peaks and financial crisis events: in Figure \ref{fig2}, vertical white lines are identifying the Oil Crisis of 1973, the Black Monday (19.10.1987) which took place during the Savings and Loan Crisis in the USA, the bursting of the dot-com bubble in 2001 (followed by the early 2000s recession), and the  2007-2012 financial crisis. Those episodes clearly are matching most of the low-frequency peaks quite well, indicating that crises lead to, or consist of,  strong changes in the dependence structure of low returns. 

This apparent impact of crises on copula spectra is confirmed when focusing the analysis on the corresponding periods. In Figure~\ref{SingleCrisis}, we provide plots of the local lag-window estimators for~$\tau_1 = \tau_2 = 0.1$ before and after two of those four crises, the~{2001 bursting of the dot-com bubble} and the 2007 financial crisis. More precisely, for each of them, we calculated local lag-window estimations using only observations before the critical  date,  and compared them to estimations using only observations taken after it. {None of the pre-crisis curves indicates  any significant deviation from  white noise, whereas both post-crisis ones do.} The interpretation is that crises, locally  but quite suddenly, produce  changes in the dependencies between low returns. Those changes happen \textsl{after} the crisis onset, and   thus do not help  predict ing them; as shown by Figure~\ref{fig2}, they fade away  more slowly than they appeared. As for the atypical spectra in the late sixties, they are probably due to the fact  that the market, at that time, was much smaller, and less efficient, than nowadays.  In addition, some of the peaks of  low-returns spectral densities  at low frequencies are not systematically associated with any well-identified crisis. This indicates that, apart from crises, other events can also influence the dependence structure of low returns. Peaks in quantile spectral densities at low frequencies, indeed,  also can be caused by  time-varying variances. This fact was first observed by \cite{li2014},  who suggested that this  phenomenon could   explain some features of  the    S\&P500 spectra. A closer look, however,   reveals that it cannot account for  all dependence   in that dataset: see   Section \ref{sec:addsp500} of  the online appendix.

\begin{figure}[t]
 \begin{minipage}[b]{.48\textwidth}
  \includegraphics[width = 70mm]{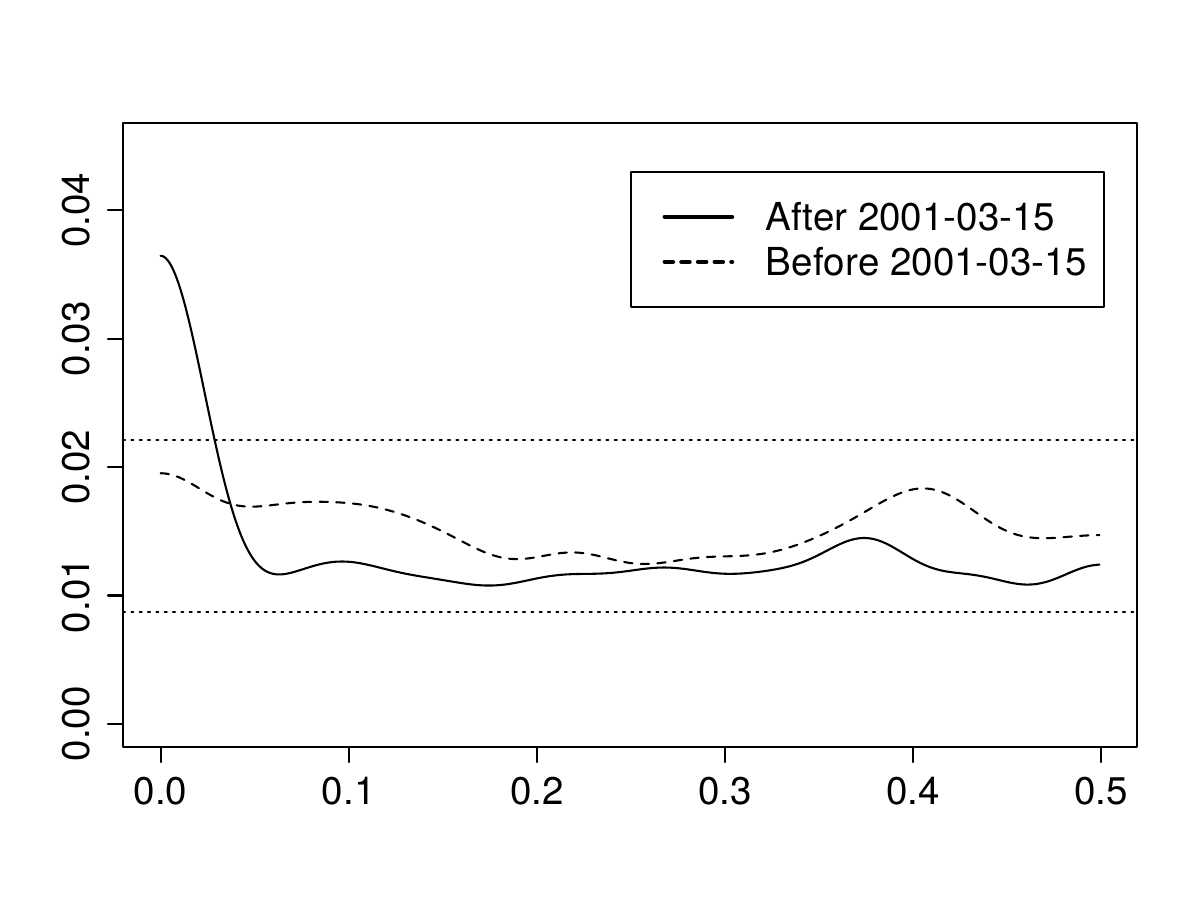}\vspace{-3mm}
  \end{minipage}
   \quad
  \begin{minipage}[b]{.48\textwidth}
  \includegraphics[width = 70mm]{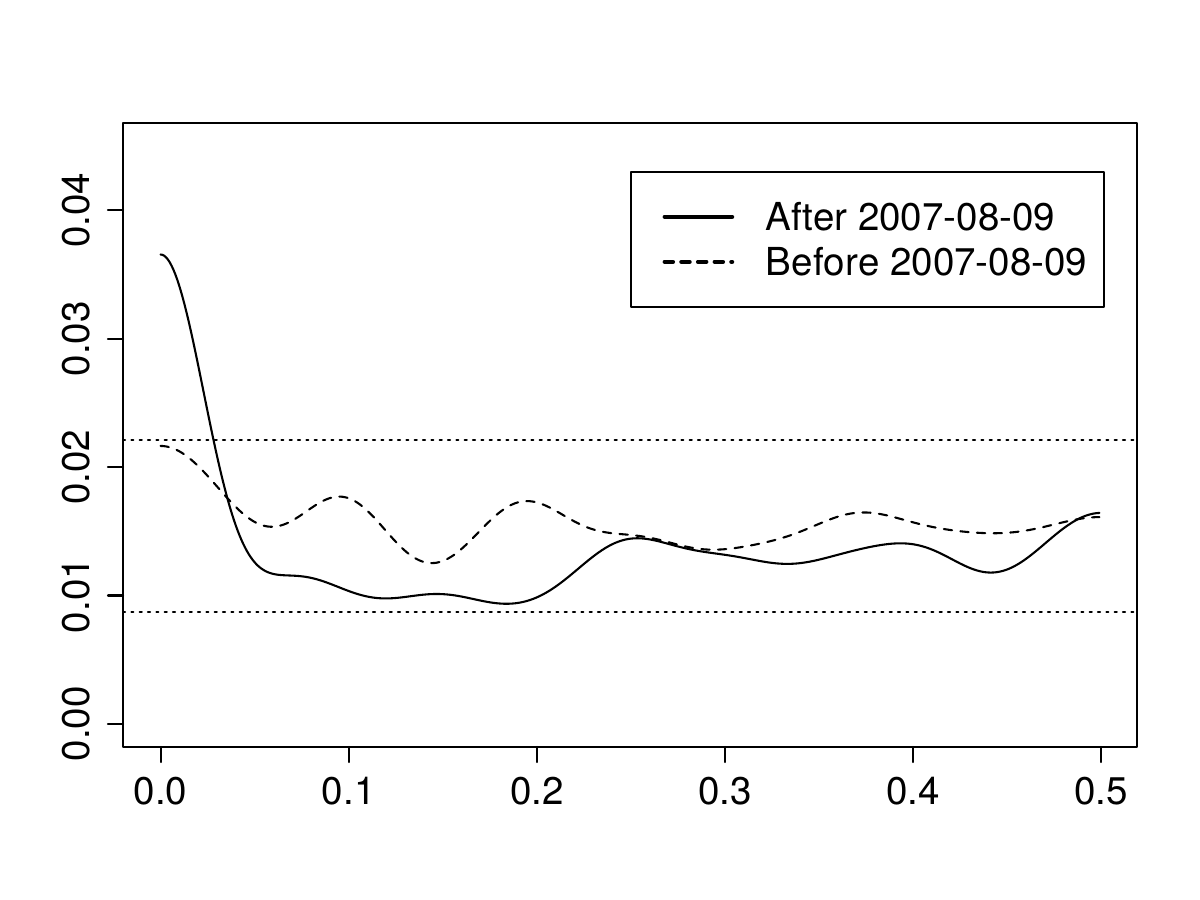}\vspace{-5mm}
  \end{minipage}
  
    \caption{\small Single local lag-window estimators calculated before (dashed) and after (solid) the bursting of the dot-com bubble in 2001 (left) and the beginning of the financial crisis in 2007 (right);   the dotted horizontal  lines represent the values of $q_{\min}$ and $q_{\max}$ from Section~$\ref{Seccalibr} (iii);$ smoothing and bandwidth choices as in Figure~\ref{figSPgen}.}
    \label{SingleCrisis}
\end{figure}

\subsection{Daily temperatures in Hohenpeissenberg} \label{SecHPB}

\begin{figure}[t]
\centering
   \includegraphics[width = 0.45\textwidth]{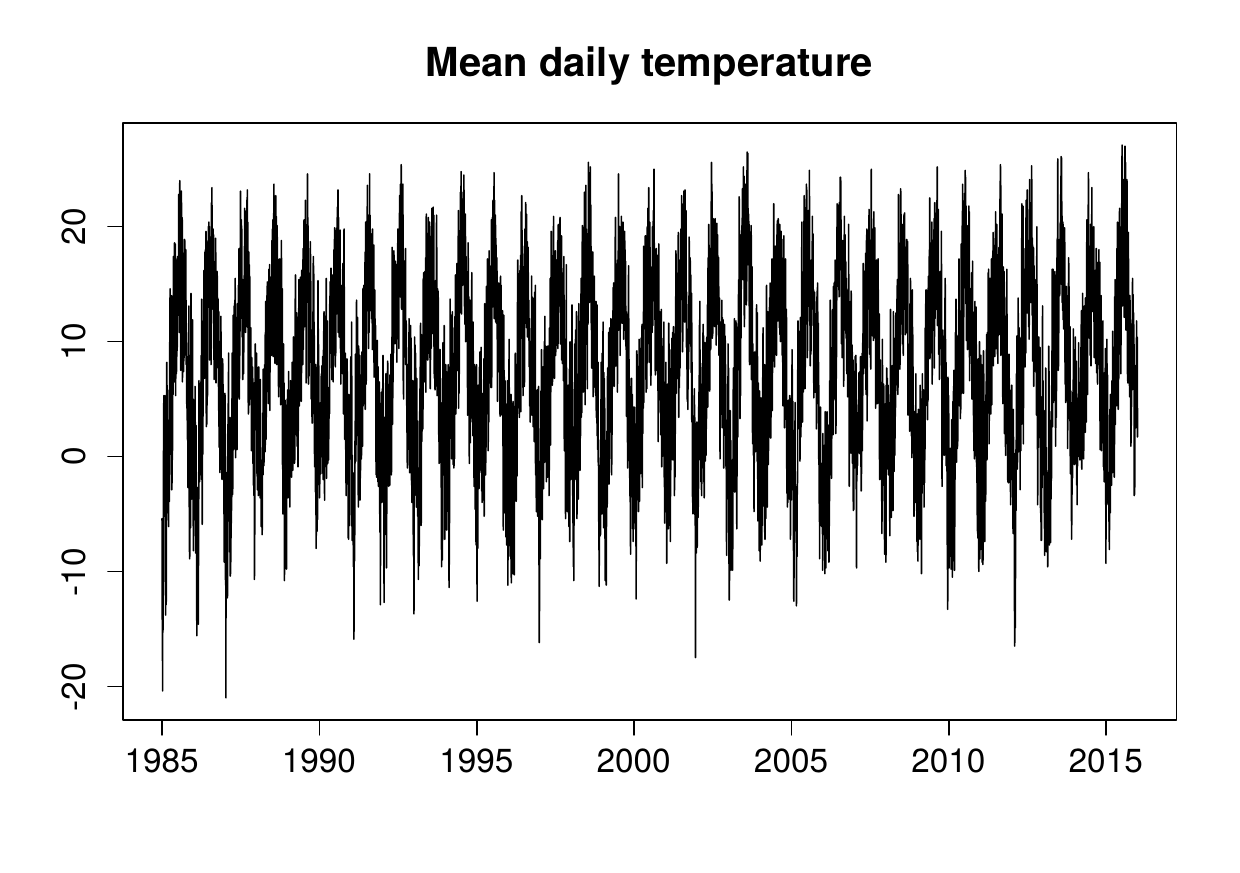}
   \includegraphics[width = 0.45\textwidth]{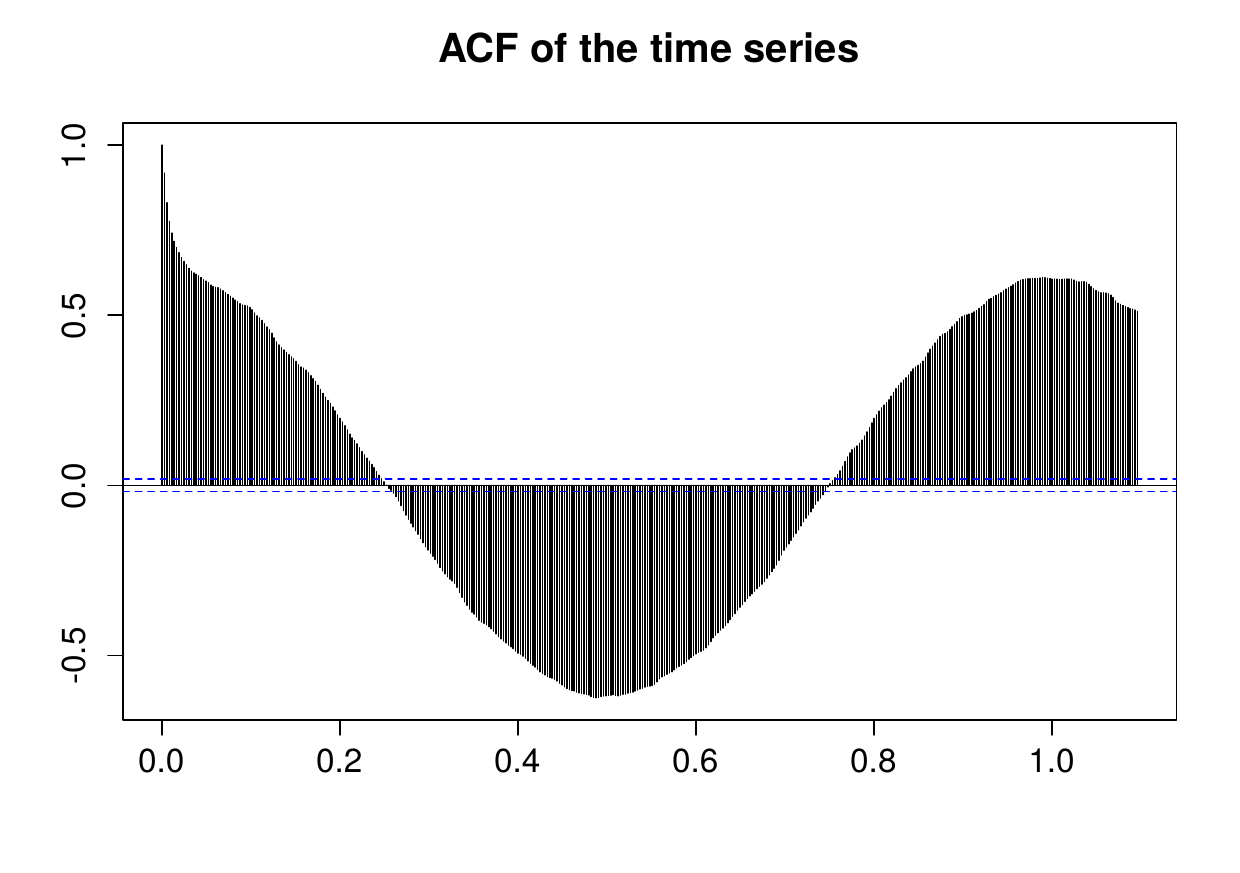}
   \includegraphics[width = 0.45\textwidth]{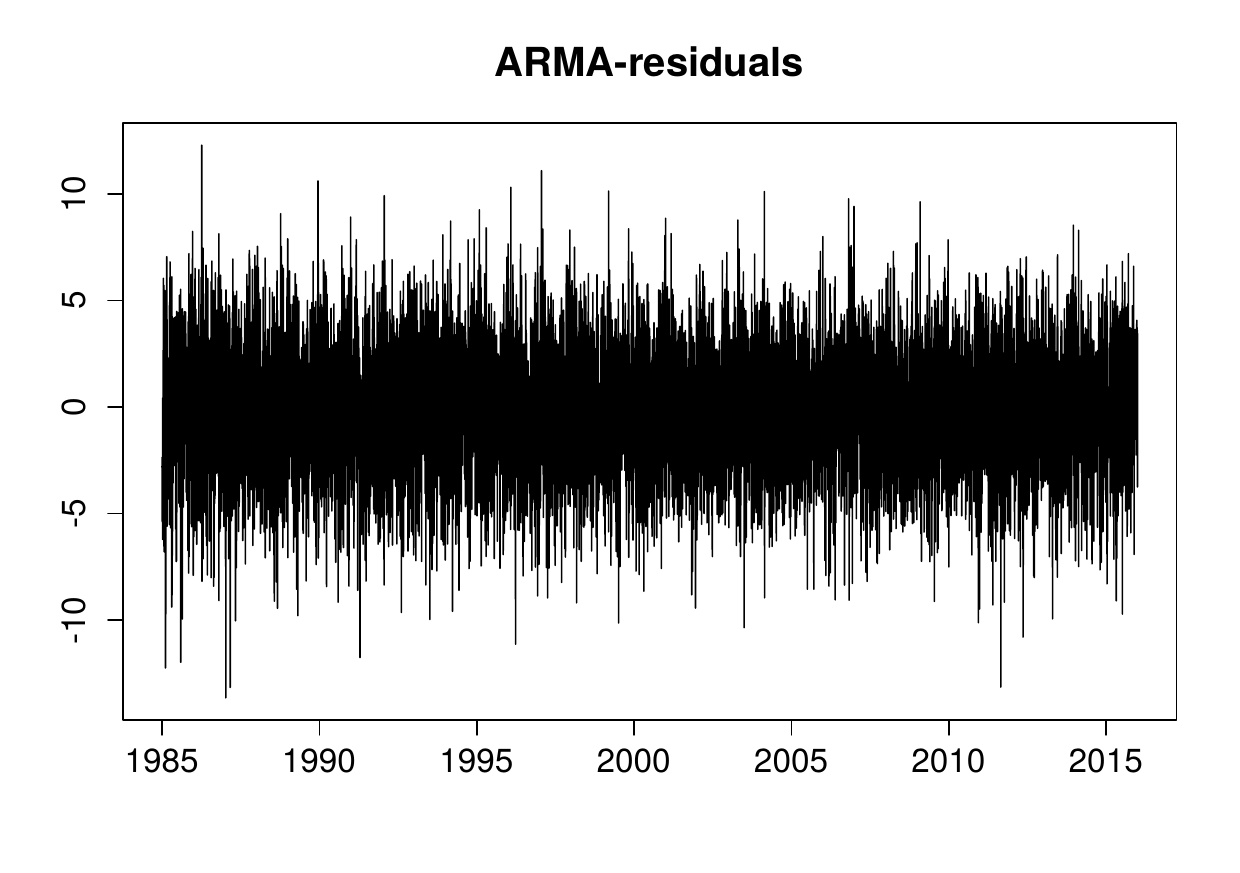}
   \includegraphics[width = 0.45\textwidth]{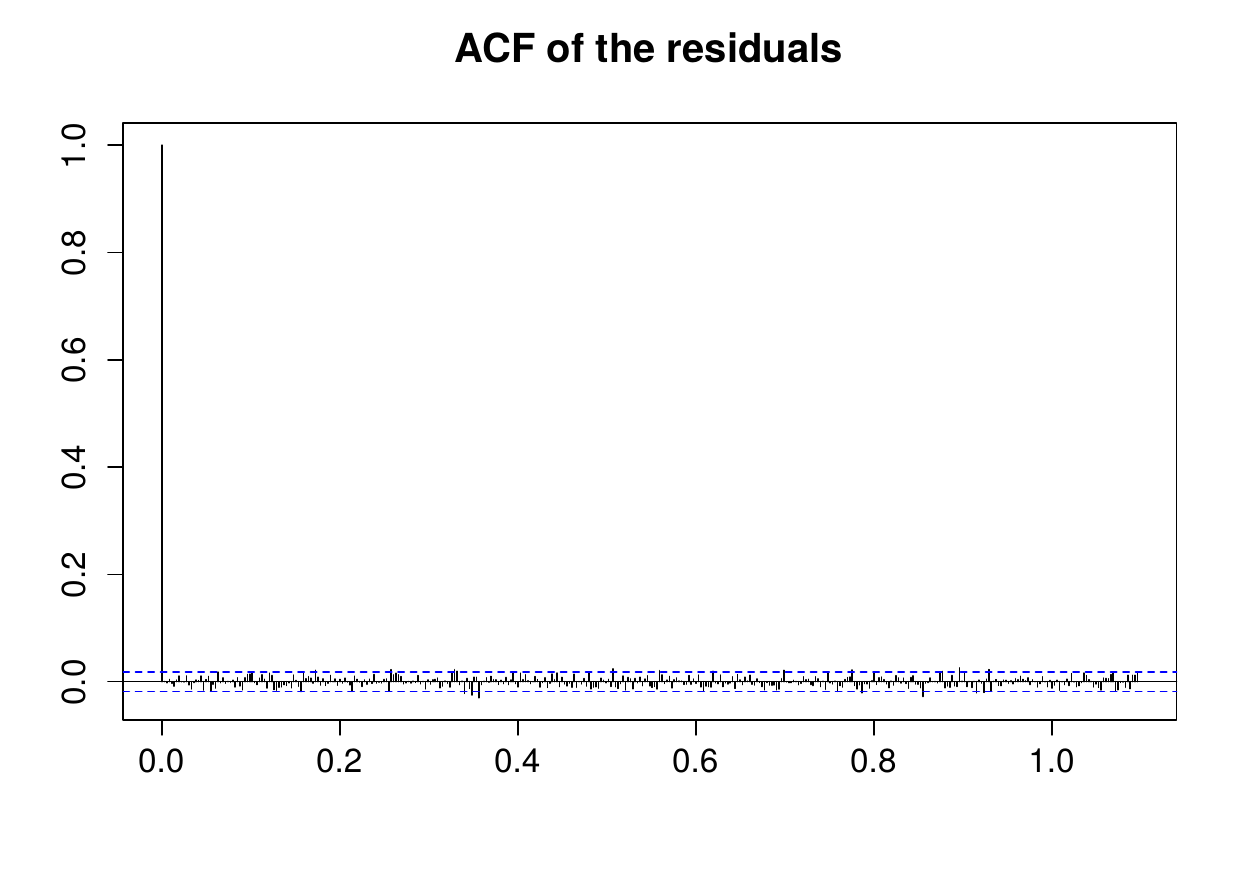}
\caption{\small First row displays $T=11315$ observation of daily mean temperature between 1985 and~2015 at the Hohenpeissenberg Meteorological Observatory and its autocorrelation function;    the second row shows the residuals of the ARMA$(3,1)$ fit and its autocorrelation function.  }
\label{FigHPB}
\end{figure}
As mentioned in Section \ref{Seccalibr}, the new methodology  admits a hypothesis testing interpretation---the null hypothesis being that of strong white noise. In our second dataset, we analyze the residuals of an ARMA$(p,q)$   fit to a seasonally adjusted time-series of air temperatures recorded at the meteorological station in Hohenpeissenberg (Germany). More precisely,   $T=11315$ observations of daily mean temperatures were recorded  between 1985 and 2015; they are displayed in the upper part of Figure \ref{FigHPB}. To remove the clearly visible seasonality , we first fit a trigonometric regression model of the form
\[ y = c + \alpha x + \sum_{k=1}^4 \beta_k \sin(2\pi k/365) + \gamma_k \cos(2\pi k/365), \]
where the linear part is used to remove any possible trend. An ARMA(p,q) model with $p=3$ and $q = 1$  (determined via $AIC$ and an 
inspection of residual autocorrelations, see \cite{campbell2011weather} for a similar approach) then was estimated from the residuals. The  residuals resulting from that second fit and their autocorrelations are shown in the lower part of 
 Figure \ref{FigHPB}. From a $L_2$ perspective, this  successfully captures the bivariate behavior of the dataset. It its therefore not surprising that 
 the (global) classical spectral density  of those residuals does not show any significant structure (see Figure \ref{FigL2SD}).
\begin{figure}[t]
\centering
   \includegraphics[width = 0.45\textwidth]{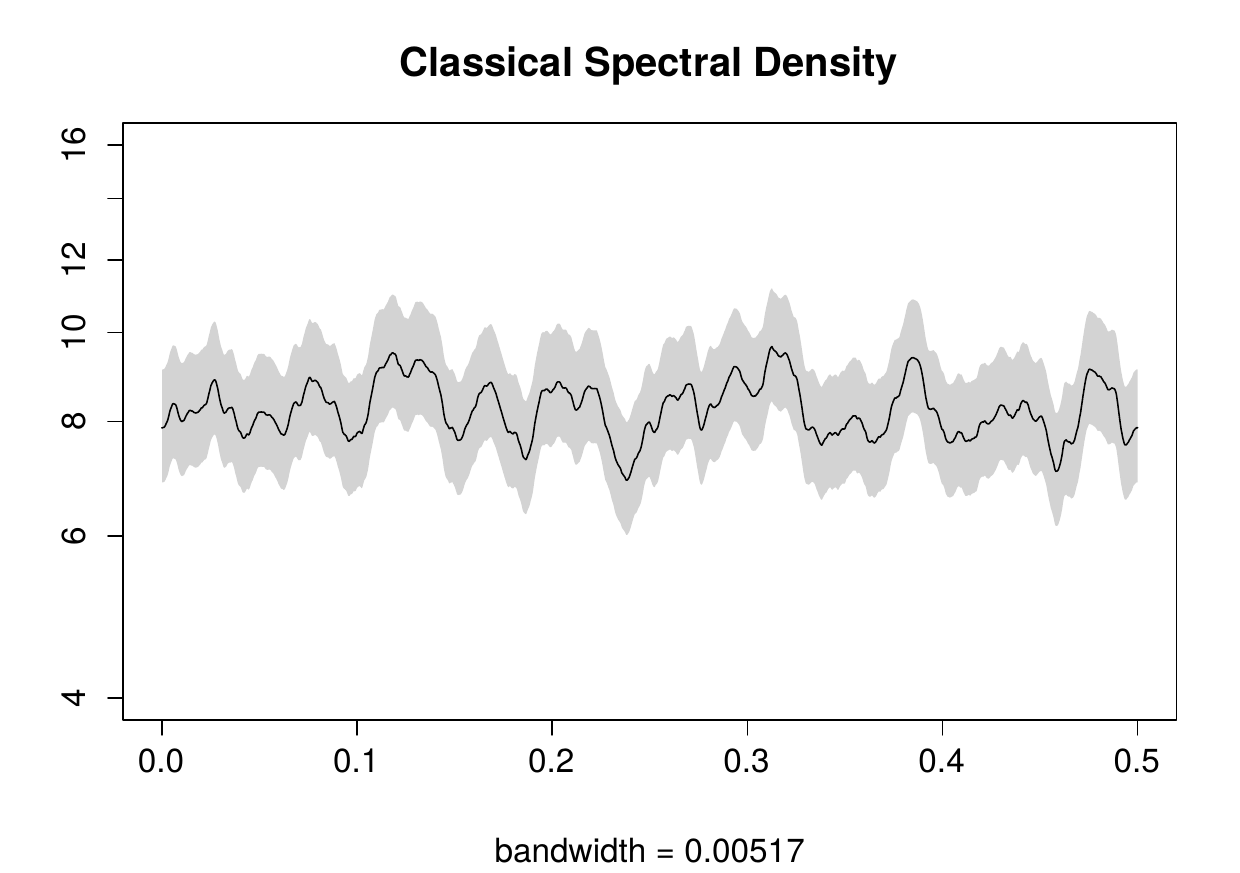}
\caption{\small Classical spectral density of the ARMA-residuals with pointwise calculated $0.95$ confidence interval .}  
\label{FigL2SD}
\end{figure}

The quantile spectral analysis of the same dataset leads to a much different conclusion. Estimating the quantile spectral densities of the same residuals with $B_n = 10,$  $n=2048,$ \linebreak  $\mathcal{T}_0 = \{1024+32j|j=0,\dots,290\}$, and $\Omega = \{2\pi j/1024|j=0,\dots,1023\}$, we obtain the  heat maps  shown in Figure~\ref{FigHPBheats}. The central heat map ($\tau_1=\tau_2 = 0.5$ deviates) indicates clear deviations from strong white noise. Starting around 1995, a significant peak around frequency zero appears. The peak reaches its maximum in 2003 and declines slowly afterwards but does not vanish. It is interesting to note that this peak is most prominent during the 2003 heat wave in Europe, which could indicate a connection with long-term  climatic fluctuations. Other significant effects, although not as dramatic, are also visible in the heat maps involving more extreme quantiles $\tau_1,\tau_2 \in \{0.1,0.9\}$. One interesting observation is the strong asymmetry between the spectra associated with low and high temperatures (quantiles).

These results suggest that an ARMA model is far from fully capturing the distributional features of the data---though it does capture its $L^2$ dynamics. The analysis performed here reveals clear deviations 
  from white noise, which again  cannot be detected by  classical spectral analysis. It  also clearly shows an evolution through time of the  dependence structure of daily temperatures.  Such findings are not entirely new, and ARMA-GARCH models have been proposed for similar datasets: see \cite{campbell2011weather}. It should be emphasized, however, that the residual spectra   we observe in this dataset do not correspond to typical  GARCH spectra, which suggests that it might be worthwhile to investigate the validity of such parametric models in greater detail.  

%
%
%

\begin{figure}[t]
\centering
   \includegraphics[width = \textwidth]{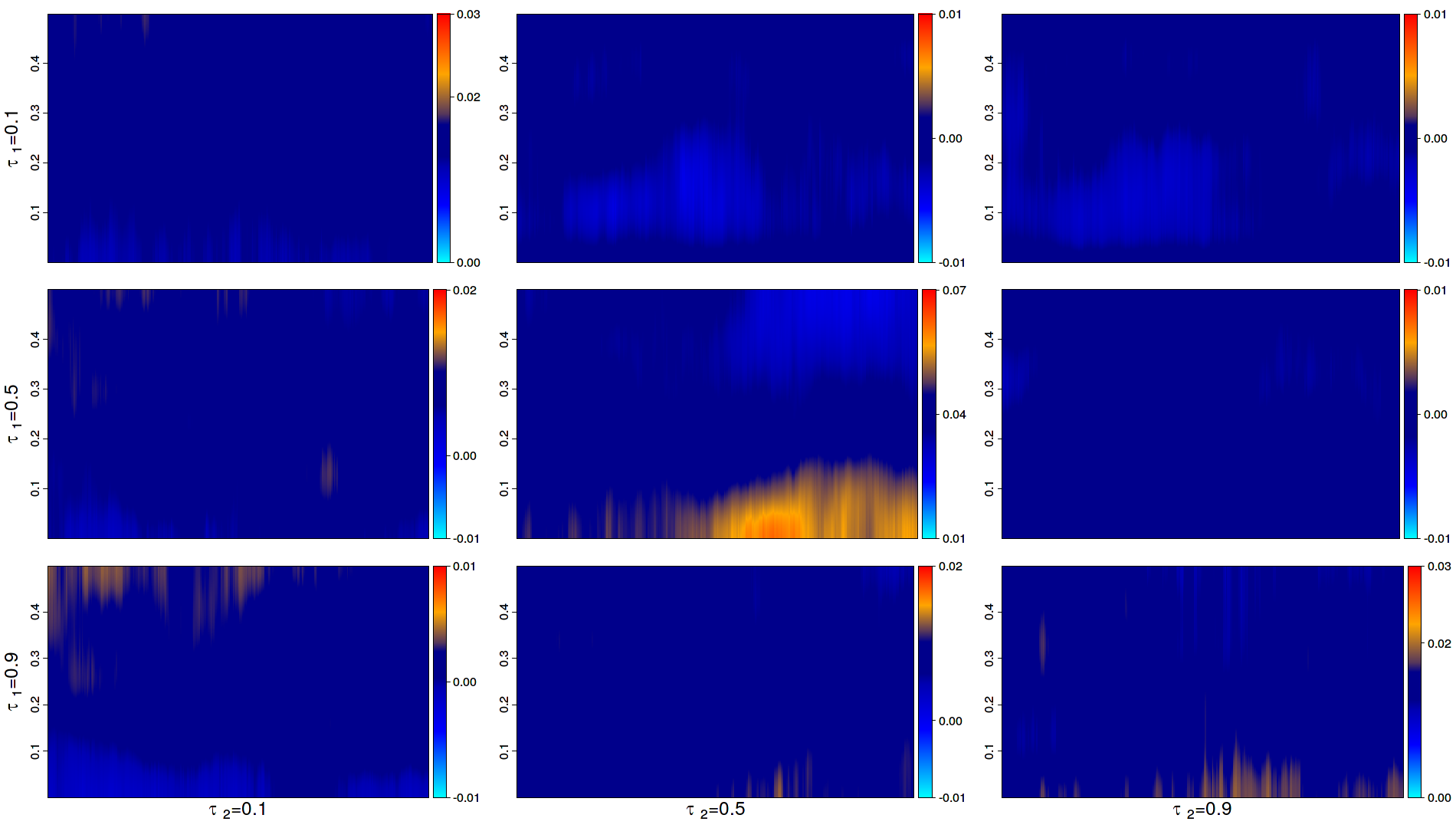}
\caption{\small Time-frequency heatmaps of the quantile lag-window estimator for the ARMA-residuals of the daily mean temperature between 1985 and 2015 at the Hohenpeissenberg Meteorological Observatory  for quantile levels $0.1,0.5$ and $0.9$. The vertical axis represents  frequencies  ($\omega/2\pi$, ranging from 0 to 0.5),  and the horizontal axis is time ($1\leq t\leq 11315$). The plots are organized  as explained in Section \ref{Seccalibr}; the color code is provided along the right-hand side of each  figure.
\vspace{-3mm} }
\label{FigHPBheats}
\end{figure}


\section{Theoretical results} 
\label{theory}
\def\theequation{5.\arabic{equation}}
\setcounter{equation}{0}

\subsection{Examples of locally strictly stationary models} \label{sec:strictstatmodels} In contrast with the many definitions of local stationarity considered in the literature, which are based on evolving covariance structures, and classical spectra, or time-varying parametric models, local strict stationarity is a purely distributional concept. In this section, we show how, under additional constraints, those other concepts eventually fall under the umbrella of our definition. 

\paragraph{5.1.1 tvMA$(\infty)$ processes.}
\cite{dahlpol2006} define a  tvMA$(\infty)$ process as admitting a representation of the form 
\begin{equation}
X_{t,T} = \mu(t/T) + \sum_{j=-\infty}^{\infty} a_{t,T}(j) \xi_{t-j}, \label{def:MA}
\end{equation}
where $\xi_t$ is i.i.d.\ white noise. This definition cover a wide range of popular linear time-varying models, such as the  tvARMA$(p,q)$ 
ones. 

Consider the following assumptions.
\begin{description}
 \item[(MA1)] 
 There exist functions $ a(\cdot,j)$ and  $\mu(\cdot): (0,1) \rightarrow \mathbb{R}$ with\vspace{-2mm}
 \end{description}
\begin{align*} \label{div}
&\sup_{t,T} | a_{t,T}(j) - a(\frac{t}{T},j)| \leq \frac{K}{T l(j)},\quad \sup_{u \in (0,1)} \left| \frac{\partial a(u,j)}{\partial u} \right| \leq \frac{K}{l(j)}, \text{ and } \sup_{u \in (0,1)} \left| \frac{\partial \mu(u)}{\partial u} \right| \leq K 
\end{align*}
\begin{enumerate}
\item[] where $K$ is a finite constant (not depending on $j$) and $\sum_j  1/l(j) < \infty.$ Furthermore, 
$\sup_{u \in (0,1)} \sum_{j=-\infty}^{\infty} |a(u,j)| < \infty$ and $\inf_{u \in (0,1)} |a(u,0)| > \rho > 0. $
\end{enumerate}
\begin{description}
\item[(MA2)] 
The random variables $\xi_t$ have bounded density function $f_\xi$ and finite  expectation, and, for some constant $C_f>0$, $f_\xi$ is such that  
$
\sup_{x \in \mathbb{R}} |xf_\xi(x)| \leq C_f.
$
\end{description}

We then have the following result. 
\begin{lemma} \label{lem:ma}
 If Assumptions (MA1) and (MA2) hold, the tvMA$(\infty)$ process defined in~$(\ref{def:MA})$ is locally strictly stationary in the sense of Definition \ref{def:strstat}, with stationary approximation  
$$X^\vartheta_{t} =  \mu(\vartheta) + \sum_{j=-\infty}^{\infty} a(\vartheta,j) \zeta_{t-j}\vspace{-3mm}
$$
where the $\zeta_t$'s are i.i.d.~copies of the $Z_t$'s. 
\end{lemma}

\paragraph{5.1.2  tvARCH$(\infty)$processes.}
\cite{dahlrao2006} define  a tvARCH$(\infty)$ process  by\vspace{-1mm}
\begin{equation} \label{def:tvARCH}
 X_{t,T} = \sigma_{t,T}Z_{t},\qquad \text{where} \qquad \sigma_{t,T}^2 = 
a_0(t/T) + \sum_{j=1}^{\infty}a_j(t/T)X^2_{t-j,T},\vspace{-1mm}
\end{equation}
where the $Z_t$'s are i.i.d.\ random variables with $\E(Z_t) = 0, \Var(Z_t) = 1,$ and density $f.$  They  show  that $X^2_{t,T}$, if not   $X_{t,T}$ itself, has an almost surely well-defined and unique expression  in the set of all causal solutions of~(\ref{def:tvARCH}) if the following assumption holds.
\begin{description}
 \item [(ARCH1)]  
 The coefficients $a_j$ in (\ref{def:tvARCH}) are non-negative and $\inf_{u \in (0,1)} a_0(u) > \rho$ for some constant~$\rho>0.$  There exist  constants $Q < \infty$, $M < \infty$, and \vspace{-1mm}$ 0<\nu<1$, and a positive sequence $l(j)$, $j\in\mathbb{N}$, such that $\sum_{j=1}^\infty  {j}/{l(j)} < \infty,$ and
\[\sup_{u \in (0,1)} a_j(u) < {Q}/{l(j)},\quad Q \sum_{j=1}^\infty {1}/{l(j)} \leq (1-\nu) \text{ and } |a_j(u) - a_j(v)| < M  {|u-v|}/l(j).\]\vspace{-1mm}
\end{description}\vspace{-4mm}

 In general, equation $(\ref{def:tvARCH})$ has no unique solution, as  the sign of a solution $X_{t,T}$ associated with the almost surely well-defined (under assumption $(ARCH1)$) $X^2_{t,T}$  can be randomly either positive or negative. 
%
To avoid this non-unicity problem, we require~$\sigma_{t,T}$ in~(\ref{def:tvARCH})  to be positive. More precisely, we impose the following condition. 
\begin{description}
 \item [(ARCH2)]
For some constant $C_f<\infty$, 
 $\sup_{x \in \mathbb{R}} |xf(x)| < C_f $, and $\sigma_{t,T} = \sqrt{\sigma_{t,T}^2}.$
\end{description}
We then have the following result. 
\begin{lemma} \label{lem:arch}
  If Assumptions (ARCH1) and (ARCH2) hold,   a process $X_{t,T}$ satisfying equation $(\ref{def:tvARCH})$ is locally strictly stationary in the sense of Definition \ref{def:strstat}, with stationary approximation  \vspace{-1mm}
 $$X^\vartheta_{t} = \sigma^\vartheta_{t}\zeta_{t}\qquad\text{with}\qquad (\sigma^\vartheta_{t})^2 = 
a_0(\vartheta) + \sum_{j=1}^{\infty}a_j(\vartheta)(X^\vartheta_{t-j})^2 \vspace{-3mm}$$
where the $\zeta_t$'s are i.i.d.~copies of the $Z_t$'s. \end{lemma}

\paragraph{5.1.3   tvGARCH$(p,q)$ processes.}
\cite{SubbaRaoGARCH} similary defines  a tvGARCH$(p,q)$ process as
\begin{equation} \label{def:tvGARCH}
 X_{t,T} = \sigma_{t,T}Z_{t},\quad \text{with} \quad \sigma_{t,T}^2 = 
a_0(t/T) + \sum_{j=1}^{p}a_j(t/T)X^2_{t-j,T} + \sum_{i=1}^{q} b_j(t/T)\sigma^2_{t-j,T},
\end{equation}
where $Z_t$ are i.i.d.\ random variables with $\E(Z_t) = 0$,  $\Var(Z_t) = 1$, and density $f.$
Parallel to (ARCH1) and (ARCH2), consider the following assumptions. 

\begin{description}
\item[(GARCH1)] 
  The coefficient functions $a_j(u)$, $j=0,\ldots , p$ and $b_j(u)$, $j=1,\ldots , q$, are positive, Lipschitz-continuous,  and satisfy, for some $0 < \mu < 1$ and $\rho >0$, \vspace{-2mm}
  \end{description}
\begin{equation} \label{as:KoeffGARCH}
\sup_{u \in (0,1)} \Big[\sum_{j=1}^{p}a_j(u) + \sum_{i=1}^{q} b_j(u) \Big]< 1 - \mu,
\quad\text{and}\quad \inf_{u \in (0,1)} a_0(u) > \rho.
\end{equation}
\begin{description}
\item[(GARCH2)] 
For some constant $C_f<\infty$, $
\sup_{x \in \mathbb{R}} |xf(x)| \leq C_f$, and $ \sigma_{t,T} = \sqrt{\sigma_{t,T}^2}$.
\end{description}

The following then holds true.
\begin{lemma} \label{lem:garch}
  If Assumptions (GARCH1) and (GARCH2) hold,  a process $X_{t,T}$ satisfying equation $(\ref{def:tvGARCH})$ is locally strictly stationary in the sense of Definition \ref{def:strstat}, with stationary approximation \vspace{-3mm}
$$X^\vartheta_t = \sigma^\vartheta_{t}\zeta_{t}, \quad\text{ where }\quad (\sigma^\vartheta_{t})^2 = 
a_0(\vartheta) + \sum_{j=1}^{p}a_j(\vartheta)(X^\vartheta_{t-j})^2 + \sum_{i=1}^{q} b_j(\vartheta)(\sigma^\vartheta_{t-j})^2 \vspace{-3mm}$$
 and   the $\zeta_t$'s are i.i.d.~copies of the $Z_t$'s. 
\end{lemma}
The proofs of Lemmas 5.1-5.3 can be found in Section~A.4.1 of the online appendix. 


\subsection{Asymptotic theory}\label{sec:asyth}

 Let $(\Omega, \mathscr{A}, \p)$ denote  a probability space, and let $\mathscr{B}$, and  $\mathscr{C}$ be subfields of $\mathscr{A}.$ Define $$
\beta(\mathscr{B},\mathscr{C}) := \E \sup_{C \in \mathscr{C}} |\p(C)-\p(C|\mathscr{B})|
$$
and, for an array $\{Z_{t,T}:1\leq t \in \mathbb{Z}\}$, let \vspace{-3mm}
\[
\beta(k) := \sup_T \sup_{t \in \mathbb{Z} } \beta(\sigma(Z_{s,T}, s \leq t),\sigma(Z_{s,T}, t + k \leq s )), \vspace{-2mm}
\]
where $\sigma(Z)$ is the $\sigma-$field generated by $Z.$  Recall that a process is called {\it $\beta$-mixing} or {\it absolutely regular} if $\beta(k)\rightarrow 0$ as $k \rightarrow \infty.$

Before  proceeding with 
  the asymptotic properties of~$\hat{\mathfrak{f}}_{t_0,T}(\omega,\tau_1,\tau_2)$, we are collecting here some technical assumptions needed in the~sequel. 

\begin{enumerate}
\item[(K)] The lag-window function $K$  in $(\ref{esti})$ satisfies $\|K\|_\infty \leq 1$, $K(0)=1$ and has support~$[-1,1];$ its extension to $\R$ is $d$ times continuously differentiable with $d \geq 2$. Additionally, $K$ has {\textquoteleft characteristic exponent\textquoteright}  $r>0$ [see \cite{parzen1957} or  \cite{priestley1981}], that is,~$r$ is the largest integer such that
 $
C_K(r) := \lim_{u \to 0} \big( {1 - K(u)}\big)/{|u|^r}
$ 
exists, is finite and non-zero.
\end{enumerate}

\begin{enumerate}
\item[(A1)] The triangular array $\{X_{t,T}\}$ is $\beta$-mixing with 
 mixing coefficients $\beta^{[X]}(k) = O(k^{-\delta})$ for some $\delta>1$. The same holds for $\{X_t^\unu\}$.
\item[(A2)]For any $\eps > 0$ define
$
\rho_n(\eps) := \big(\frac{\eps + n^{{1}/{(1+\delta)}}\eps^2}{n}\log(n)\big)^{{1}/{2}} \vee (n^{-{\delta}/{(1+\delta)}}\log(n)),
$
with $\delta$ as in Assumption (A1). Assume that $\rho_n(T^{-2/5}) = o((nB_n)^{-1/2})$, \mbox{$T^{-2/5}  = o(B_n^{1/2}/n^{1/2})$} and  
\[
\frac{B_n^2}{n} + \frac{n^2}{T^2B_n} + \frac{n^{1/2}B_n}{T} = o\Big(\frac{B_n^{1/2}}{n^{1/2}}\Big).
\]
\item[(A3)] 
\begin{enumerate}
\item[(i)] For some $\gamma>2$ and $T^{-2(\gamma-1)/5\gamma} = o(B_n^{1/2}n^{-1/2}),$\vspace{-2mm}
\[
\sup_t \sup_{x,y} \Big| F_{t,t+k;T}(x,y) - F_{t;T}(x)F_{t+k;T}(y)\Big| = O(|k|^{-\gamma}).
\]\vspace{-11mm}
 
\item[(ii)] $\sum_{k \in \Z} |k|^r \sup_{u,\eta_1,\eta_2} |\gamma^u_k(\eta_1,\eta_2)| < \infty\vspace{1mm}$, where the supremum is over a neighborhood of $(\unu,q^\unu(\tau_1),q^\unu(\tau_2))$.
\item[(iii)] For $2 \leq p \leq 8$, define\vspace{-4mm}
\hspace{-1cm}
\begin{align*}
\hspace{0cm}&\kappa_p(s_1,...,s_{p-1}) := \sup_{T}\sup_{t \in \Ntheta} \sup_{x_1,...,x_p} |\cum(\I{X_{t,T}\leq x_1},\I{X_{t+s_1,T}\leq x_2},...,\I{X_{t+s_{p-1},T}\leq x_p})|\\
\hspace{0cm}&\kappa_p^\unu(s_1,...,s_{p-1}) := \sup_{x_1,...,x_p} |\cum(\I{X_{0}^\unu \leq x_1},\I{X_{s_1}^\unu\leq x_2},...,\I{X_{s_{p-1}}^\unu\leq x_p})|;\vspace{-2mm}
\end{align*}\vspace{-9mm}

\noindent assume moreover that the quantities $\kappa_p(s_1,...,s_{p-1})$ and $\kappa_p^\unu(s_1,...,s_{p-1})$ are absolutely summable over $s_1,...,s_{p-1} \in \Z$.
\end{enumerate}
\item[(A4)] 
\begin{enumerate}
\item[(i)] The joint distribution functions $F_{t_1,...,t_j;T}$ of $(X_{t_1;T},...,X_{t_j;T}), j=2,...,4$ are twice continuously differentiable, and all partial derivatives of order one and two are bounded, uniformly in $t_1,...,t_j,T$ and the arguments. The distribution function~$F_{t_1;T}$ is twice continuously differentiable and its derivatives are bounded uniformly in $t_1$ and $T$. 
\item[(ii)] Let  
$
\mathfrak{d}^{(r)}_\omega\mathfrak{f}^u(\omega,x,y) := \frac{1}{2\pi} \sum_{k \in \Z} |k|^r \gamma^u_k(x,y) e^{- \mathrm{i}k \omega},
$ where $r$ is taken  from Assumption~(K). The function $(u,x,y) \mapsto \mathfrak{d}^{(r)}_\omega\mathfrak{f}^u(\omega,x,y)$ is continuous in a neighborhood of $(u,x,y) = (\unu,q^\unu(\tau_1),q^\unu(\tau_2))$.
\item[(iii)] The function $u \mapsto \mathfrak{f}^u(\omega,G^u(q^\unu(\tau_1)),G^u(q^\unu(\tau_2)))$ is twice continuously differentiable in a neighborhood of $u = \unu$.
\item[(iv)] The functions $G^u_k$ and $ G^u$ in the definition of local strict stationarity are, for some~$d \geq 2,$ $d$ times continuously differentiable with respect to $u.$ The function~$G^\unu$ has a density, which is uniformly bounded away from zero on an open set that contains  $q^\unu(\tau_1)$ and  $q^\unu(\tau_2)$.
\end{enumerate}
\end{enumerate} 

\begin{rem} 
{\rm Assumptions (K) and (A2) place mild restrictions on the lag-window generator~$K$ and  
   the bandwidth parameter, respectively. One can show that they are satisfied by the  bandwidth parameters  leading to optimal asymptotic MSE rates for the mean squared error, see the discussion in Remark \ref{rem:bw} for more details. Assumption~(A3) is  verified if $\delta$ in~(A1) is large enough: in fact,   it is sufficient to replace the $\beta$-mixing coefficients in~(A1) by $\alpha$-mixing coefficients (see Lemma~\ref{lem:mix} in the online Appendix for additional details on bounding cumulants through $\alpha$-mixing coefficients). 
 Assumption~(A4) places conditions on the smoothness properties of the underlying processes which rule out   processes  with jump-like non-stationarity.} 
\end{rem}

\begin{rem}{\rm
Observe that 
\begin{align*}
\mathfrak{f}^u\big(\omega,G^u(q^\unu(\tau_1)),G^u(q^\unu(\tau_2)\big)
= \frac{1}{2\pi}\sum_{k \in \Z} e^{-ik\omega} \big(G^u_k(q^\unu(\tau_1),q^\unu(\tau_2)) - G^u(q^\unu(\tau_1))G^u(q^\unu(\tau_2))\big).
\end{align*}

This means that the differentiability of $u \mapsto \mathfrak{f}^u(\omega,G^u(q^\unu(\tau_1)) ,G^u(q^\unu(\tau_2)))$ depends on the local smoothness with respect to time of joint distributions.}
\end{rem}

Our main result states that, after proper centering and scaling,  $\hat{\mathfrak{f}}_{t_0,T}(\omega,\tau_1,\tau_2)$ is asymptotically  (complex) normal.

\begin{thm}\label{thm:asymain}
Let Assumptions (K) and (A1)-(A4) hold. Then, for any sequence $\omega_n$ of Fourier frequencies such that $|\omega_n - \omega| = O(1/n)$ for some $\omega \in (0,\pi)$, and for $t_0 = \lfloor T\unu \rfloor$,
\begin{equation} 
\sqrt{n/B_n} \Bigg(
\begin{array}{c}
\Re \hat{\mathfrak{f}}_{t_0,T}(\omega_n, \tau_1, \tau_2) - \Re \mathfrak{f}^\unu(\omega,\tau_1,\tau_2) - \Re b(\omega, \tau_1, \tau_2) \\
\Im \tilde{\mathfrak{f}}_{t_0,T}(\omega_n, \tau_1, \tau_2) - \Im \mathfrak{f}^\unu(\omega,\tau_1,\tau_2) - \Im b(\omega, \tau_1, \tau_2)
\end{array}
 \Bigg) \weak \mathcal{N}\Big(0,\Sigma^2(\omega,\tau_1,\tau_2)\Big)
\end{equation}
where
\[
\Sigma^2(\omega,\tau_1,\tau_2) := \pi \mathfrak{f}^{\unu}(\omega, \tau_1, \tau_1)\mathfrak{f}^{\unu}(\omega, \tau_2, \tau_2) \int K^2(u) du \Big( \begin{array}{cc} 1 & 0  \\ 0  & 1  \end{array} \Big)
\] 
\begin{flalign*}
&\text{and } & b(\omega, \tau_1, \tau_2) &:= - C_K(r)B_n^{-r} \mathfrak{d}^{(r)}_\omega\mathfrak{f}^\unu(\omega,\tau_1,\tau_2)  \\ 
& & &\hspace{10mm} + \frac{n^2}{2T^2}\frac{\partial^2}{\partial u^2} \mathfrak{f}^u(\omega,G^{u}(q^\unu(\tau_1)),G^{u}(q^\unu(\tau_2)))\Big|_{u=\unu} + o(B_n^{-r} + \frac{n^2}{T^2}).
\end{flalign*}
\end{thm}

\begin{rem} \label{rem:bw}{\rm
Theorem \ref{thm:asymain} implies consistency of the estimator, which, however, holds under weaker assumptions. 
The same theorem  also can be used to obtain the local window length~$n$ and the bandwidth parameter $B_n$ that minimize the asymptotic mean squared error of $\hat{\mathfrak{f}}_{t_0,T}(\omega_n, \tau_1, \tau_2)$. To illustrate the idea, assume that we want to optimize the asymptotic mean squared error (MSE) of $\Re \hat{\mathfrak{f}}_{t_0,T}(\omega_n, \tau_1, \tau_2)$. Considering   $r = d$, let~$\sigma^2 := \Sigma^2_{11}(\omega,\tau_1,\tau_2),$\vspace{-1mm}
\[
b_u := \frac{1}{2}\frac{\partial^2}{\partial u^2} \mathfrak{f}^u(\omega,G^{u}(q^\unu(\tau_1)),G^{u}(q^\unu(\tau_2)))\Big|_{u=\unu} \quad \text{and} \quad b_\omega := - C_K(r) \mathfrak{d}^{(r)}_\omega\mathfrak{f}^\unu(\omega,\tau_1,\tau_2).\vspace{-1mm}
\]  
In this notation the asymptotic MSE of $\Re \hat{\mathfrak{f}}_{t_0,T}(\omega_n, \tau_1, \tau_2)$ is 
  $\frac{B_n}{n}\sigma^2 + b_u \frac{n^2}{T^2} + b_\omega B_n^{-r}$. Assuming that $b_u \neq 0$ and $ b_\omega\neq 0$, straightforward calculations entail that this MSE is minimized for \vspace{-2mm}
\begin{align*}
  n  = T^{\frac{2+4r}{2+5r}} \big(\sigma^2 b_\omega^{-1/r}b_u^{2+1/r}(2r+4)\big(\tfrac{r}{2}\big)^{\frac{-r}{r+1}} \big)^{\frac{r}{2+5r}}, ~
B_n  = T^{\frac{2}{2+5r}} \big(\sigma^{-2} b_u^{-1/2}b_\omega^{5/2} (2r+4)  \big)^{\frac{2}{2+5r}}\big(\tfrac{2}{r}\big)^{\frac{-3}{2+5r}}.\vspace{-2mm}
\end{align*}
As one would expect, larger values of $r,$ corresponding to smoother local spectral densities (as functions of frequency), lead to more smoothing and faster convergence rates. For $r=2$, the asymptotic MSE of $\hat{\mathfrak{f}}_{t_0,T}(\omega_n, \tau_1, \tau_2)$ turns out to be of the order $T^{-2/3}$. One can show that, if the constant  $\delta$ in Assumption (A1)  is large enough,  the above choices of the bandwidth parameters satisfy condition (A2) if $r\geq 2$. The above formulas provide  rough guidelines about the choice of smoothing parameters. However,    the expression for the bias contains unknown parameters,  such as derivatives of the local copula spectral density kernel, which are difficult to estimate in practice.  }
\end{rem}

\section{Conclusions}

In this paper, we have defined local copula spectra using a new notion of local strict stationarity; we have  constructed a lag-window type estimator and proved  its asymptotic normality. In a stationary context,  it has  been shown that copula-based spectra  provide a  description of serial dependence structures which is 
substantially  more informative and flexible  than classical covariance-based spectra. The benefits of this new spectral methodology are thus extended to slowly evolving dependence structures. Those benefits are  highlighted in a simulation study and by analyzing two datasets,   the daily log-returns of the classical S\&P500  series and a meteorological  series recorded in Hohenpeissenberg. That analysis indeed reveals a number of  interesting features 
 that cannot be detected by 
 a more  traditional $L^2$-based approach.

Several  important questions are calling for further research, though. 
 Our method   requires the choice of a smoothing parameter---an issue which is common to all methods based on local stationarity concepts.  It seems important to have some data-driven procedure  providing  general guidelines on what a `good' choice of smoothing parameters is. An interesting approach to this  problem   has been   suggested recently by \cite{craomb2002}, and an important  direction for future research is
   the  extension of this method  to the present setting. It also is important to develop methods for uniformly  (in frequency and local time) valid 
   statistical inference on quantile spectra. This is challenging, and to the best of our knowledge  such methodology for the simpler case of classical $L^2$ spectra has only recently been developed by  \cite{liuwu2010} in the stationary case. Finally, it is natural to assume that the dependence structure of a time series contains both smooth changes  and sudden jumps. In the present paper, we have dealt with smooth changes only, and an extension accommodating possible jumps would be most welcome.  For example, in an $L^2$ setting,  such smoothness assumptions could be avoided using the piecewise
  locally stationary concept of \cite{zhou2013} or by considering  the evolutionary wavelet spectra as described in \cite{van2008locally}. Extending the distributional approach described in this paper  to piecewise
    locally stationary processes or  wavelet-based spectra (even in a strictly stationary case) is an interesting and challenging direction for
    future research. 

\bigskip
\medskip

\noindent
{\bf Acknowledgements.} We   gratefully acknowledge the many suggestions and constructive comments by two  Referees, an Associate Editor and an Editor, that helped improving the original version of this  paper.

\nocite{timerev1988}
\nocite{roueff}

\bigskip

%
%

\setlength{\bibsep}{1pt}
\setlength{\bibsep}{0pt}
\begin{small} 
\bibliographystyle{apalike}
\bibliography{LocStatBib_2016-07-11}
\end{small}

\newpage

\section*{Online Appendix} \label{app:online}
\def\theequation{A.\arabic{equation}}
\def\thesection{A.\arabic{section}}
\setcounter{equation}{0}
\setcounter{section}{0}

In this online appendix, we collect (Section~A.1) some additional information on the spectral concept considered here, (Section~A.2) some additional simulation results, (Section~A.3) some further analysis of the  S\&P500 data, and (Section~A.1)  (Sections~A.4-A.6)  the proofs of the main results, along with  some technical details. 

\section{A connections with  the Wigner-Ville spectra} \label{app:addrem}

A further theoretical justification for the time-varying copula spectral density kernels considered in this paper is their relation to  the so-called Wigner-Ville spectrum. The  Wigner-Ville spectrum (in its classical L$^2$ version)    is based on the so-called Wigner distribution of a process of the form  $\{X_{t,T}\}$ and has its origins in quantum mechanics. It was   used later on in the signal processing community. Its properties for time-varying spectral analysis have been investigated in \cite{martin}.

 For the  series of  indicators we are dealing with here, the Wigner-Ville spectrum takes the form \vspace{-3mm}
\begin{equation} \label{wigvil}
\mathfrak{W}_{\tnull,T}(\omega,\tau_1,\tau_2 ) := \hspace{-5pt} \sum_{s = - \infty}^{\infty} \hspace{-5pt} \Cov\Big(\I{X_{\lfloor \tnull + s/2\rfloor , T} \leq F^{-1}_{\lfloor \tnull + s/2\rfloor ; T}(\tau_1)}, \I{X_{\lfloor \tnull - s/2\rfloor , T}\leq F^{-1}_{\lfloor \tnull - s/2\rfloor ; T}(\tau_2)}\Big)  \frac{ e^{-i \omega s}}{2\pi}
\vspace{-3mm}\end{equation}
  (see  \cite{martin}). 

The following proposition  establishes  a strong relation between our  time-varying copula spectral density ke\-rnels~$\mathfrak{f}^\unu(\omega,\tau_1,\tau_2)
$, as defined in $(\ref{def:tvf})$, and the Wigner-Ville spectrum~$\mathfrak{W}_{\tnull,T}(\omega,\tau_1,\tau_2 )$.

\begin{prop} \label{lem:unique}
	Let $\{X_{t,T}\}$ be locally strictly stationary, with approximating pro\-cesses $\{X^\unu_{t}\}$, and assume that Assumption (A1) holds.
	If moreover the $\gamma_h^\unu(\tau_1,\tau_2)$'s are absolutely summable for any $\unu$ and $( \tau_1,\tau_2 )\in (0,1)^2$, then, for any fixed $\unu$ and $(\tau_1,\tau_2) \in (0,1)^2$, along any sequence~$\tnull = \tnull(T)$ such that~$\tnull/T~\!\!\to~\!\unu$,\vspace{-3mm}
	\[
	\sup_{\omega\in(-\pi,\pi ]} \bv \mathfrak{f}^{\unu}(\omega,\tau_1,\tau_2 ) -  \mathfrak{W}_{\tnull,T}(\omega,\tau_1,\tau_2 ) \bv = o(1),
\vspace{-4mm}	\]
	where $\mathfrak{W}_{\tnull,T}$  denotes the indicator Wigner-Ville spectrum   defined in \eqref{wigvil}.
\end{prop}
\noindent\textbf{Proof.} 
From the absolute summability of $\gamma_h^\unu(\tau_1,\tau_2),$ we obtain\vspace{-3mm}
\[
\mathfrak{f}^\unu(\omega,\tau_1,\tau_2) = \frac{1}{2 \pi} \sum_{h = -T^{1/4}}^{T^{1/4}} \gamma_h^\unu(\tau_1,\tau_2)e^{-i \omega h} + o(1),\vspace{-4mm}
\]
while Assumption (A1) yields\vspace{-3mm}
\begin{flalign*}
&\mathfrak{W}_{\tnull,T}(\omega,\tau_1,\tau_2 ) \\ 
&= \frac{1}{2 \pi} \sum_{h = -T^{1/4}}^{T^{1/4}} \Big(F_{\lfloor \tnull - h/2\rfloor,\lfloor \tnull + h/2\rfloor;T}(F^{-1}_{\lfloor \tnull - h/2\rfloor;T}(\tau_1),F^{-1}_{\lfloor \tnull + h/2\rfloor;T}(\tau_2)) - \tau_1\tau_2\Big)e^{-i \omega h} + o(1). \vspace{-4mm}
\end{flalign*}
Writing the difference between the   leading terms  in $\mathfrak{f}^\unu(\omega,\tau_1,\tau_2)$ and $\mathfrak{W}_{\tnull,T}(\omega,\tau_1,\tau_2 )$  in  terms of distribution functions yields\vspace{-4mm}
\begin{align*}
&\frac{1}{2 \pi} \sum_{h = -T^{1/4}}^{T^{1/4}} |F_{\lfloor \tnull - h/2\rfloor,\lfloor \tnull + h/2\rfloor;T}(F^{-1}_{\lfloor \tnull - h/2\rfloor;T}(\tau_1),F^{-1}_{\lfloor \tnull + h/2\rfloor;T}(\tau_2)) - G^{\unu}_{h}(q^\unu(\tau_1),q^\unu(\tau_2))|\\
&   \leq \frac{1}{ \pi} \sum_{h = -T^{1/4}}^{T^{1/4}}  \frac{L}{g_{\text{min}}} \bv \frac{h}{T} + \frac{1}{T} \bv =  o(1),
\vspace{-4mm}\end{align*}
where the last inequality follows from Lemma $\ref{lem:distquant}.$ The claim follows. 
\hfill $\qed$

 For more information about the Wigner-Ville spectrum, its properties and applications, see \cite{martin}. \vspace{-2mm}

\section{Additional Simulations} \label{app:addsim}\vspace{-2mm}
\subsection{Gaussian tvAR(2) }\label{sec:p1}
In Figure~\ref{Figa}, we display,  for a classical Gaussian time-varying AR(2) process,  the same heat maps as we did in Section~\ref{SecSimu}; in particular, part (a) was obtained along the same lines as described there. 

\begin{figure}[H]
	\begin{minipage}[b]{.48\textwidth}
		\includegraphics[width = 72mm, height = 41mm]{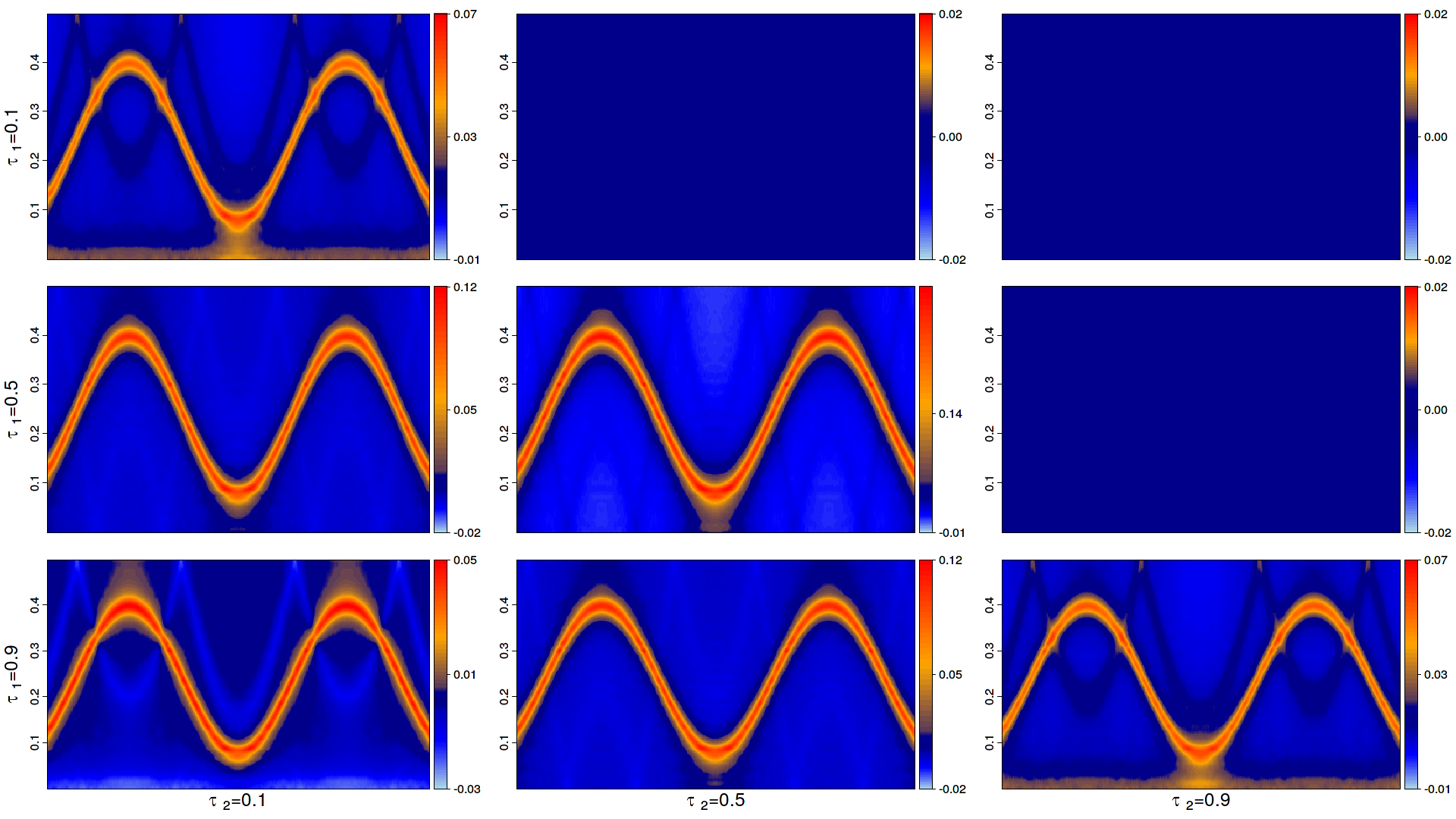}
		\subcaption{\footnotesize Actual  copula spectral densities (simulated)}
	\end{minipage}
	\begin{minipage}[b]{.48\textwidth}
		\includegraphics[width = 72mm, height = 41mm]{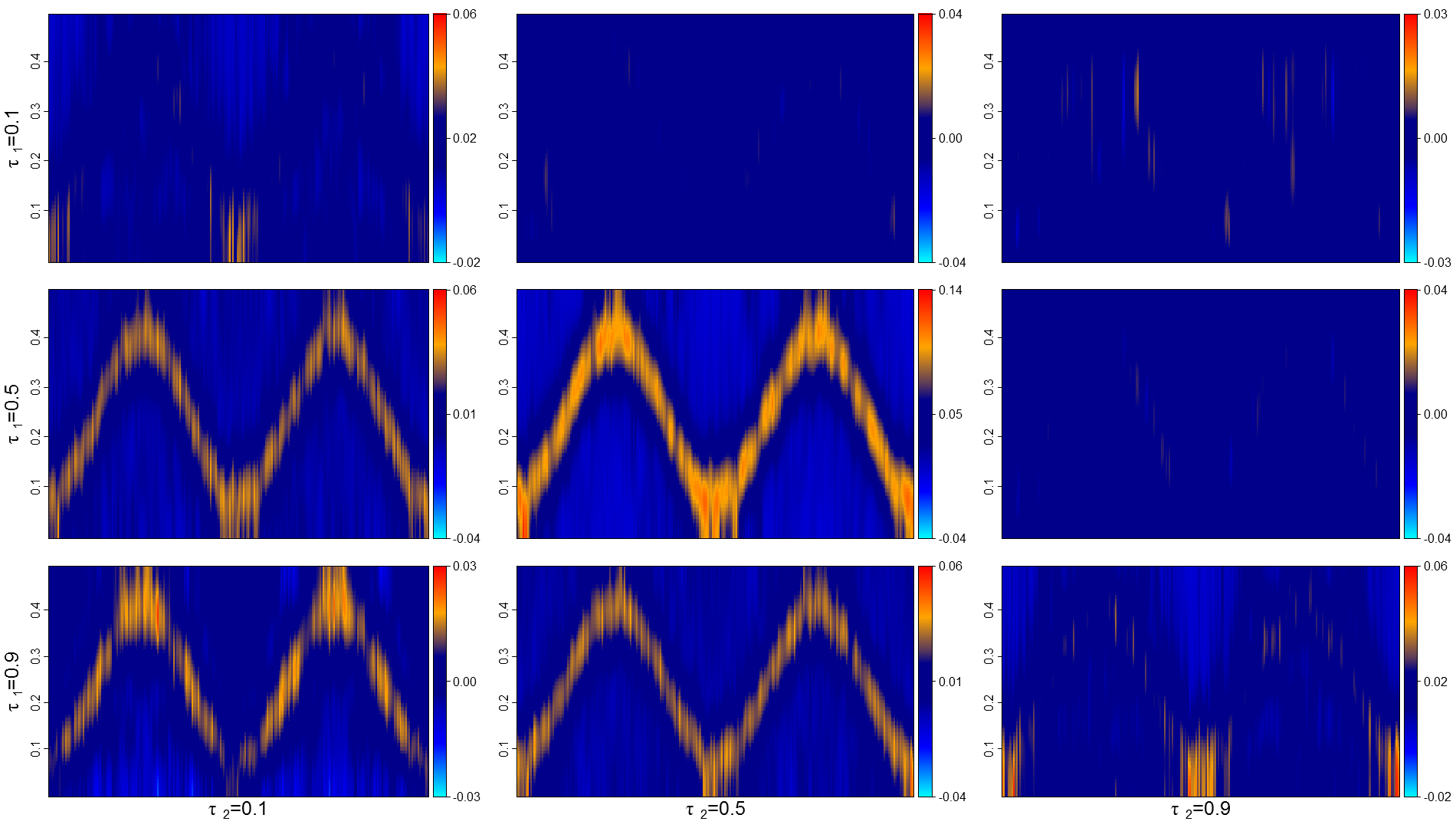}
		\subcaption{\footnotesize Estimated copula spectral densities, $n = 128$}
	\end{minipage}
	\begin{minipage}[b]{.48\textwidth}
		\includegraphics[width = 72mm, height = 41mm]{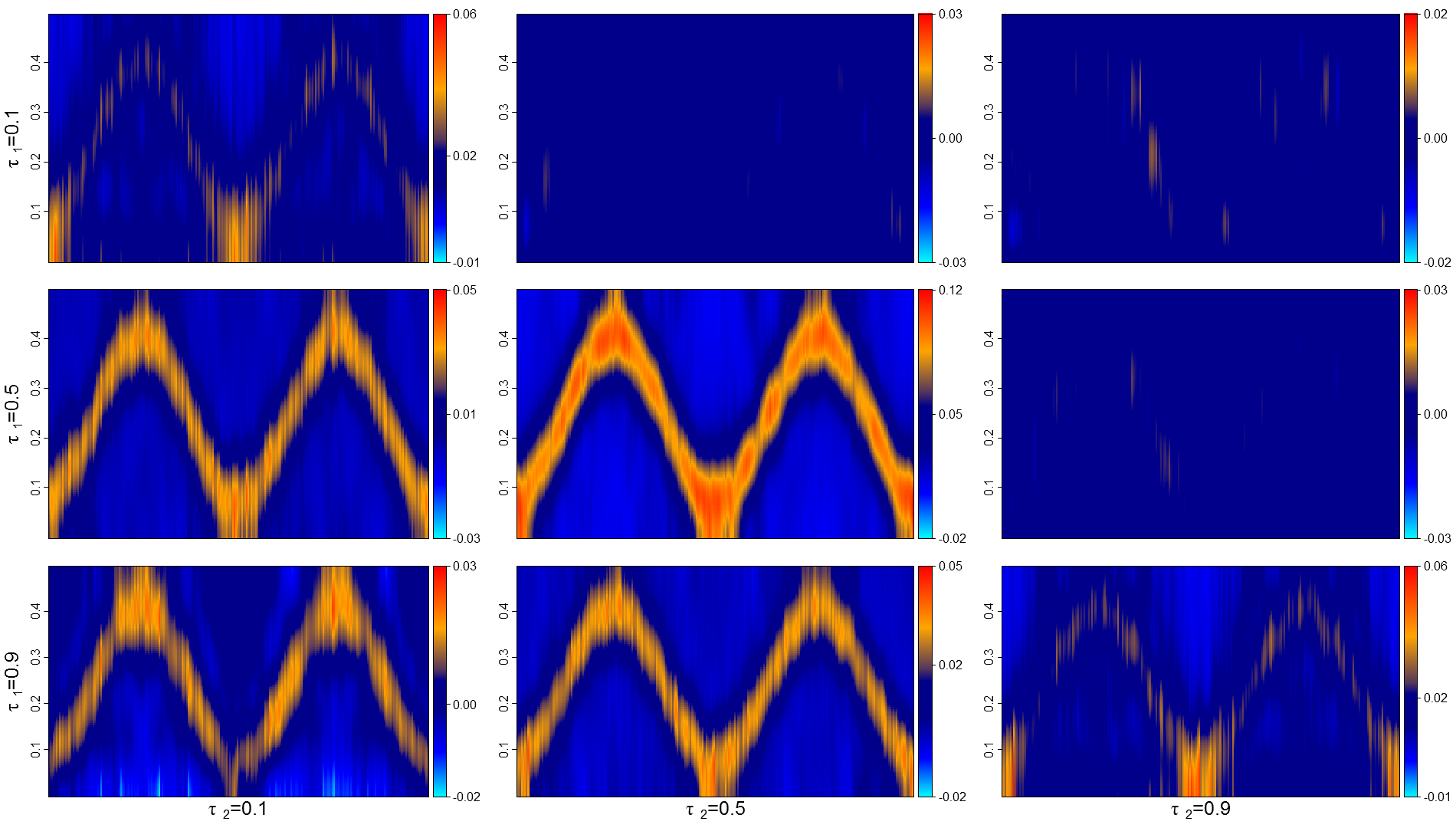}
		\subcaption{\footnotesize Estimated copula spectral densities, $n = 256$}
	\end{minipage}
	\begin{minipage}[b]{.48\textwidth}
		\includegraphics[width = 72mm, height = 41mm]{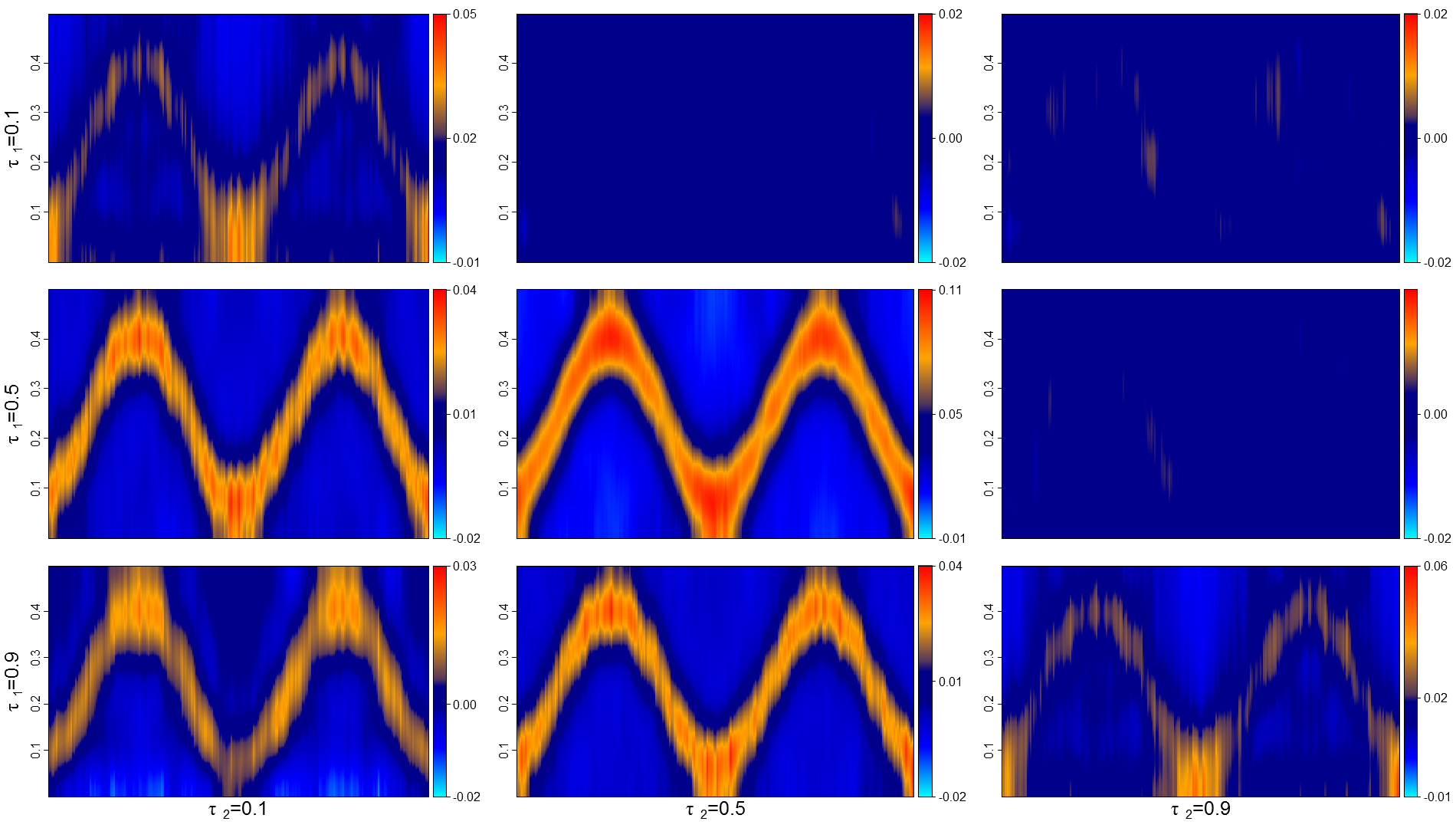}
		\subcaption{\footnotesize Estimated copula spectral densities, $n = 512$}
	\end{minipage}
	\begin{minipage}[b]{.48\textwidth}
		\includegraphics[width = 72mm, height = 41mm]{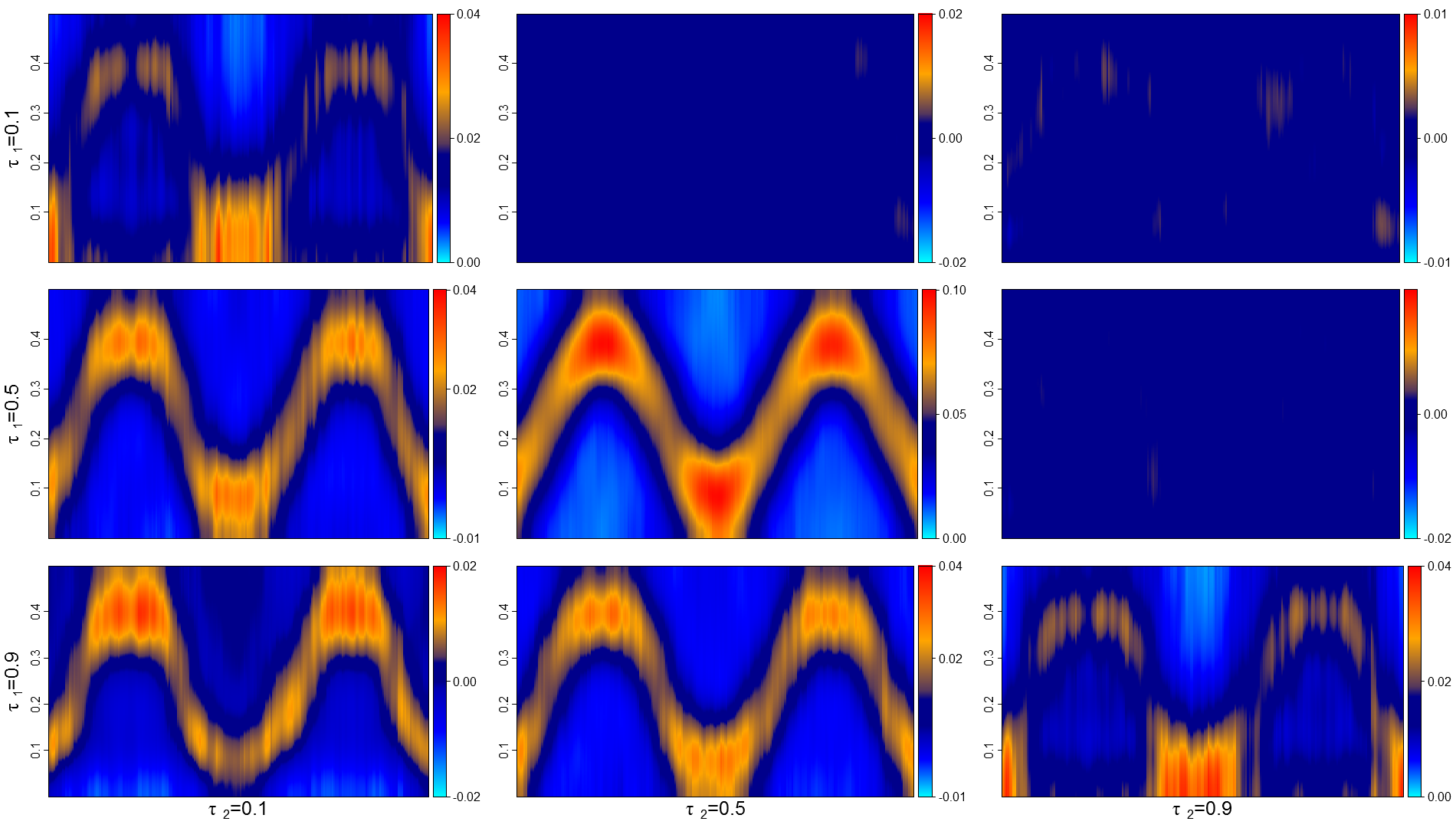}
		\subcaption{\footnotesize Estimated copula spectral densities, $n = 1024$}
	\end{minipage}\hspace{0.5cm}
	\begin{minipage}[b]{.48\textwidth}
		\includegraphics[width = 72mm, height = 41mm]{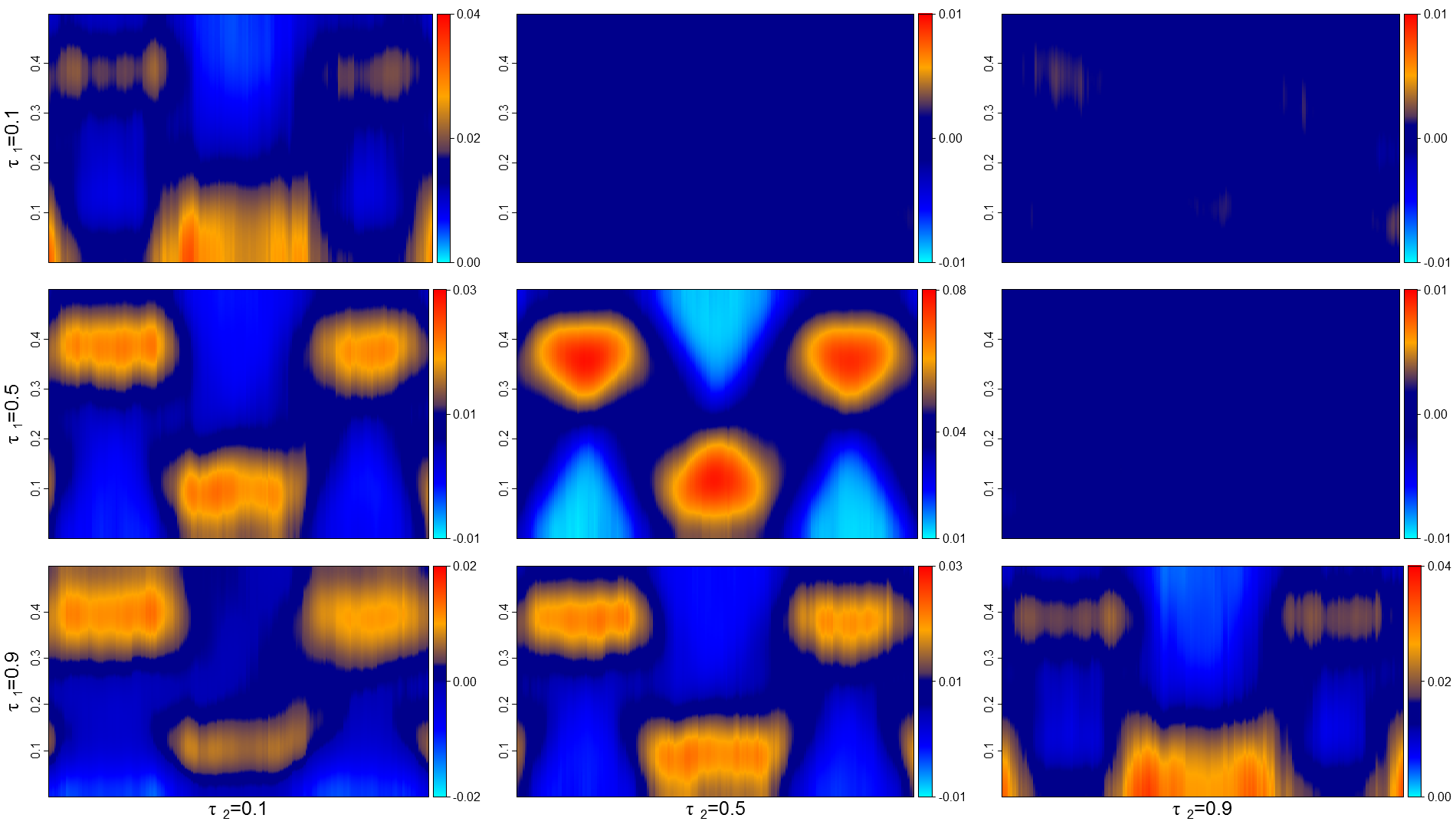}
		\subcaption{\footnotesize Estimated copula spectral densities, $n = 2048$}
	\end{minipage}
	
	\caption{\small Heatmaps of the Gaussian time-varying AR(2) process described in Section \ref{sec:p1} and the corresponding estimators, for 
		various window lenghts. 
	}\label{Figa}
\end{figure}

The model equation, taken from \cite{dahlhaus2012}, is\vspace{-1mm}
\begin{equation}
X_{t,T} = 1.8\cos(1.5-\cos(2\pi t/T))X_{t-1,T} - 0.81X_{t-2,T} + Z_t\vspace{-1mm}
\end{equation}
with i.i.d.~noise~$Z_t \sim \mathcal{N}(0,1)$. Its strictly stationary approximation, at $t_0=\vartheta T$, $0\leq \vartheta \leq 1$,~is\vspace{-2mm}
\begin{equation}
X_t^{\vartheta} = 1.8\cos(1.5-\cos(2\pi \vartheta))X_{t-1}^{\vartheta} - 0.81X_{t-2}^{\vartheta} + \zeta_t\vspace{-1mm}
\end{equation}
where $\zeta_t$ similarly is $\mathcal{N}(0,1)$ white noise. This tvAR(2) process exhibits a time-varying periodicity which is clearly visible in the heat diagrams associated with the real parts of  its time-varying copula-based spectral densities,   displayed in the lower triangular part  of   Figure~\ref{Figa}(b). The uniformly dark blue imaginary parts in the upper triangular part  are  a consequence  of the fact that those  imaginary parts actually are zero, since Gaussian processes are  time-reversible
[see Proposition 2.1 in \cite{dhkv2014}].    
As expected, no additional information can be gained from observing different quantiles (all   heatmaps in the lower-triangular parts of (a) are the same), since the (bivariate) distributions of the process are Gaussian and the change over time only affects the correlation of these conditional distributions. Because of this the (time-varying) bivariate distribution functions (and with them all quantiles) depend only on the (time-varying) correlations of the random variables, which are also fully  captured  by $L_2$ methods.

\subsection{Gaussian tvARCH(1)} \label{sec:p3} Figure~\ref{Figc}  displays the same heatmaps  for a   time-varying ARCH(1) model  of the form\vspace{-1mm}
\[
X_{t,T} = \sqrt{1/2 + (0.9t/T)X^2_{t-1,T}}Z_t\vspace{-1mm}
\]
with i.i.d.~noise~$Z_t \sim \mathcal{N}(0,1)$ and its strictly stationary approximation  at time~$t_0=\vartheta T$, $0\leq \vartheta \leq 1$ \vspace{-1mm}
\[
X_t^{\vartheta} = \sqrt{1/2 + 0.9\vartheta(X_{t-1}^\vartheta)^2}\zeta_t\vspace{-1mm}
\]
where $\zeta_t$ similarly is $\mathcal{N}(0,1)$ white noise.  In these stationary approximations,  the influence of $X^{\vartheta}_{t-1}$ on the variance of $X^{\vartheta}_t$ gradually  increases over time. This, quite understandably, gets reflected in the diagrams associated with  extreme quantiles, but  {is not visible}  in the ``median~ones''.\vspace{-1mm}

\begin{figure}[htbp]
	\begin{minipage}[b]{.48\textwidth}
		\includegraphics[width = 72mm, height = 41mm]{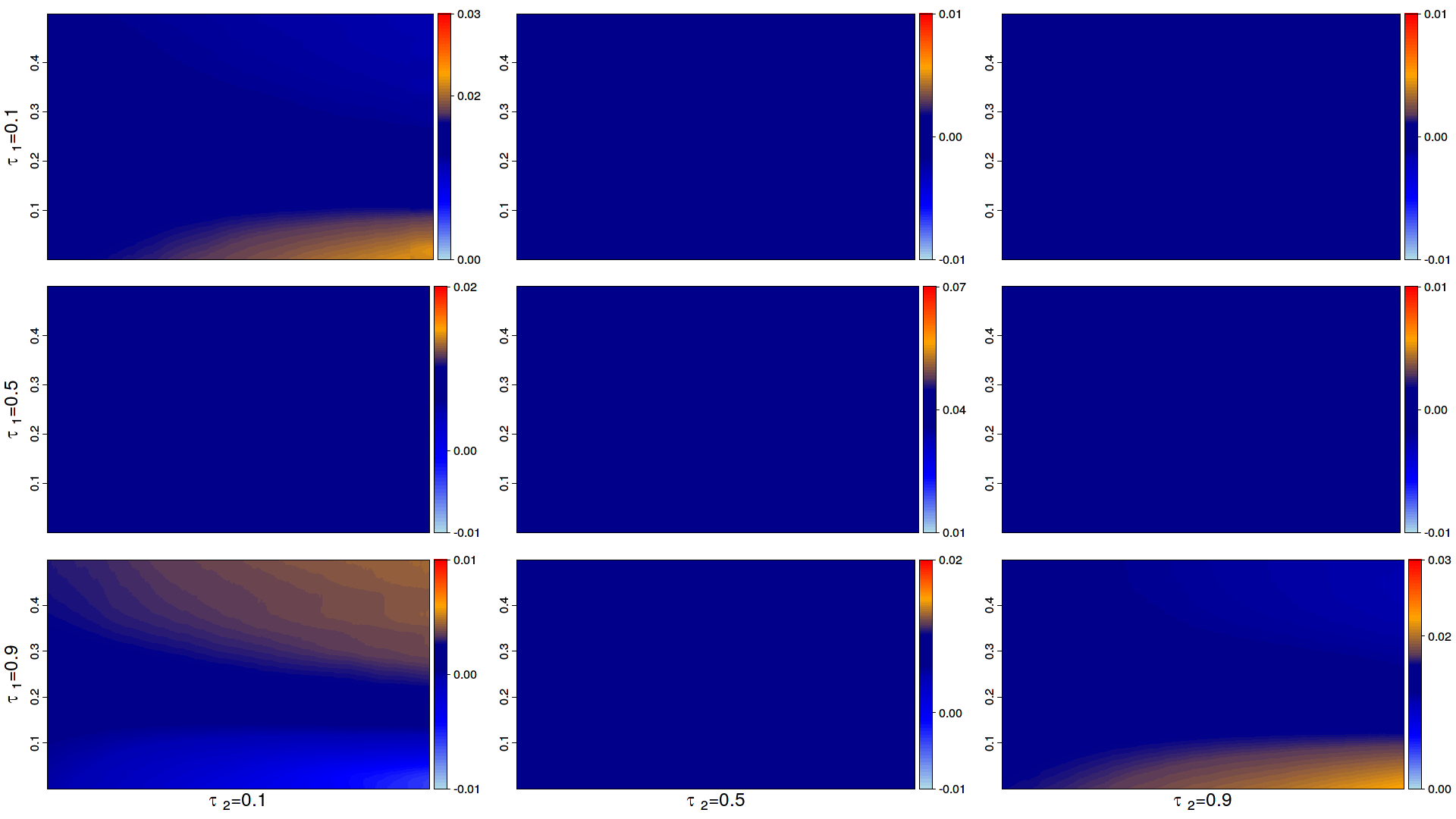}
		\subcaption{\footnotesize Actual  copula spectral densities (simulated)}
	\end{minipage}
	\begin{minipage}[b]{.48\textwidth}
		\includegraphics[width = 72mm, height = 41mm]{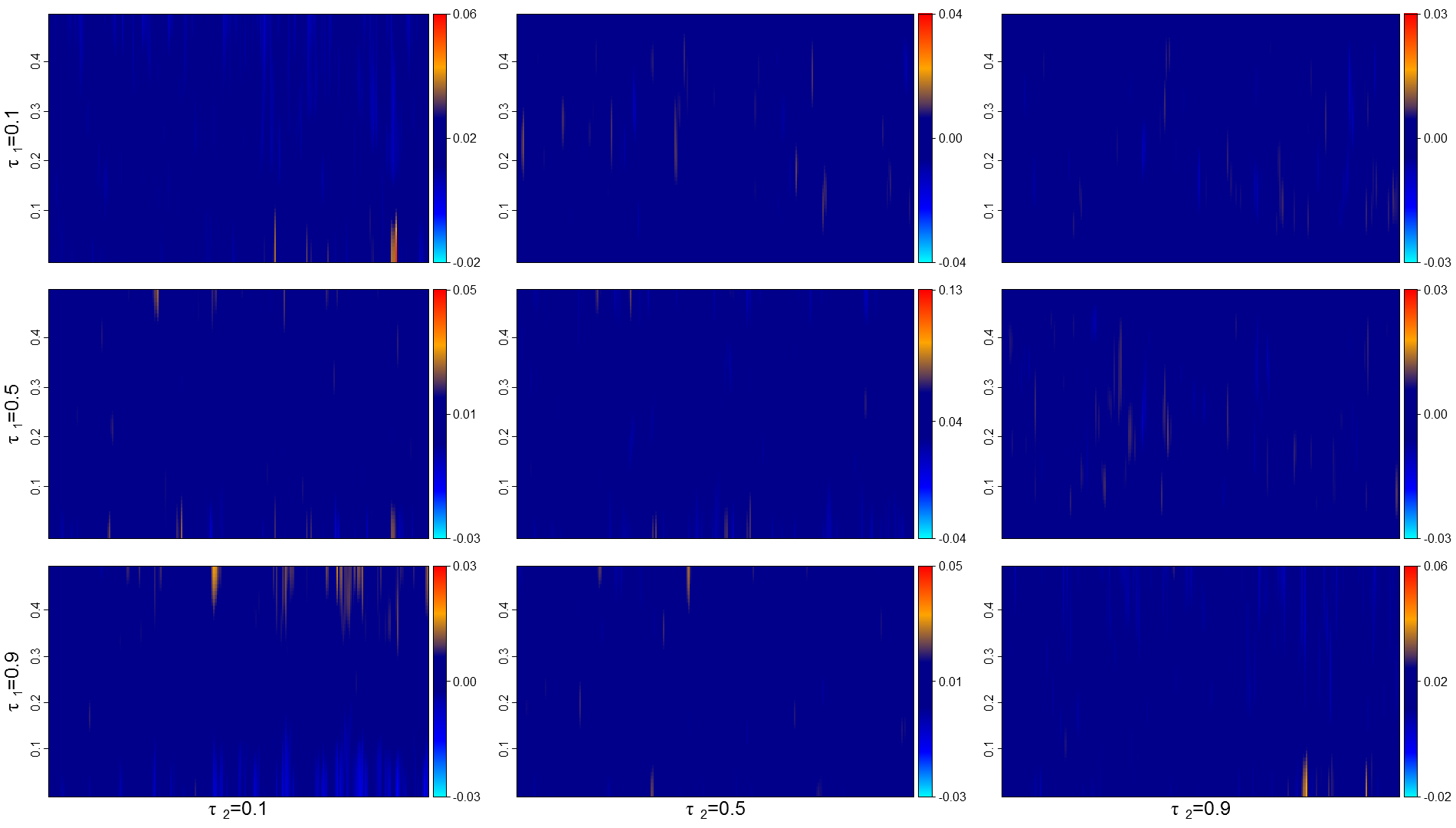}
		\subcaption{\footnotesize Estimated copula spectral densities, $n = 128$}
	\end{minipage}
	\begin{minipage}[b]{.48\textwidth}
		\includegraphics[width = 72mm, height = 41mm]{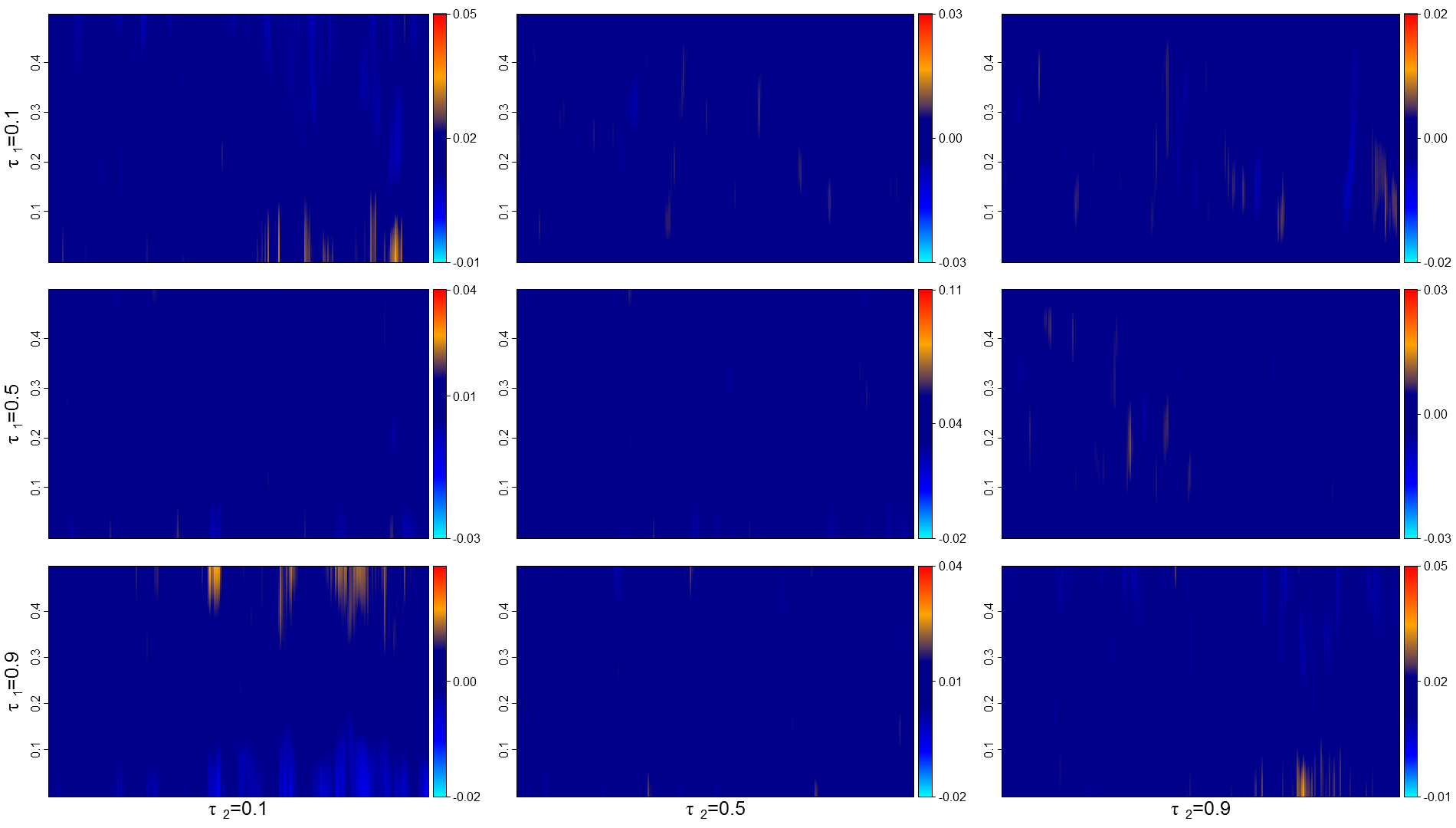}
		\subcaption{\footnotesize Estimated copula spectral densities, $n = 256$}
	\end{minipage}
	\begin{minipage}[b]{.48\textwidth}
		\includegraphics[width = 72mm, height = 41mm]{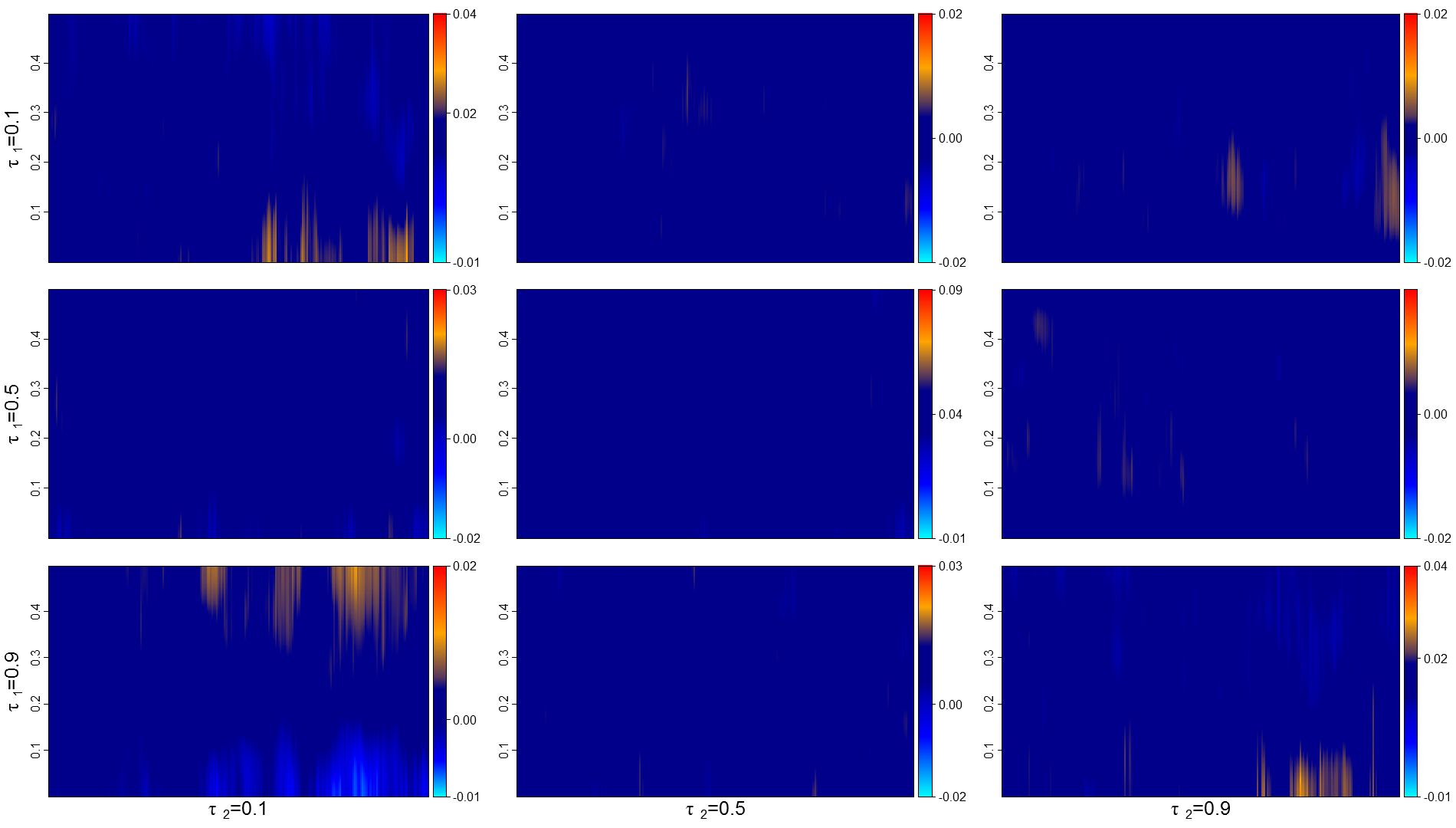}
		\subcaption{\footnotesize Estimated copula spectral densities, $n = 512$}
	\end{minipage}
	\begin{minipage}[b]{.48\textwidth}
		\includegraphics[width = 72mm, height = 41mm]{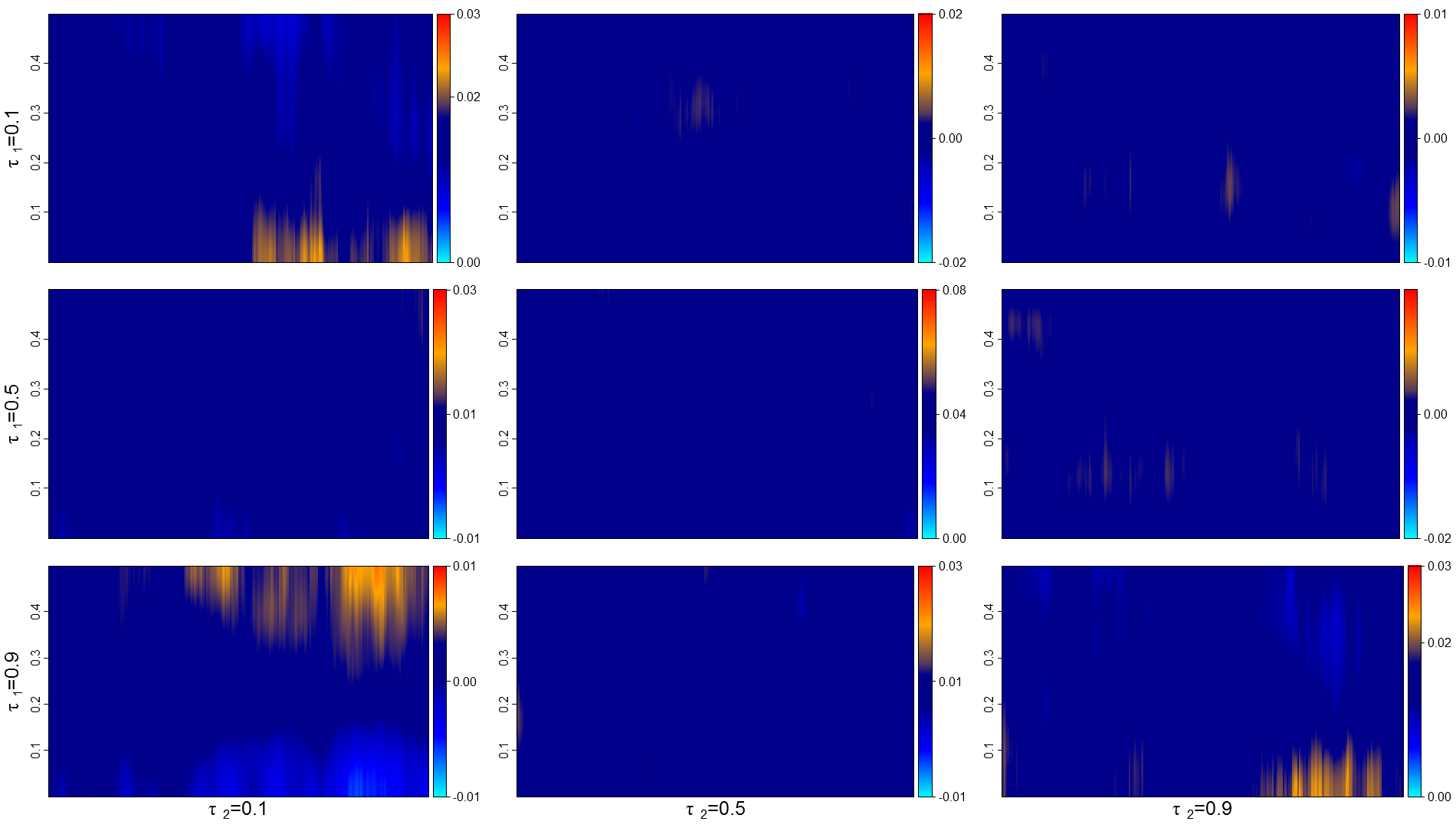}
		\subcaption{\footnotesize Estimated copula spectral densities, $n = 1024$}
	\end{minipage}
	\hspace{5mm}
	\begin{minipage}[b]{.48\textwidth}
		\includegraphics[width = 72mm, height = 41mm]{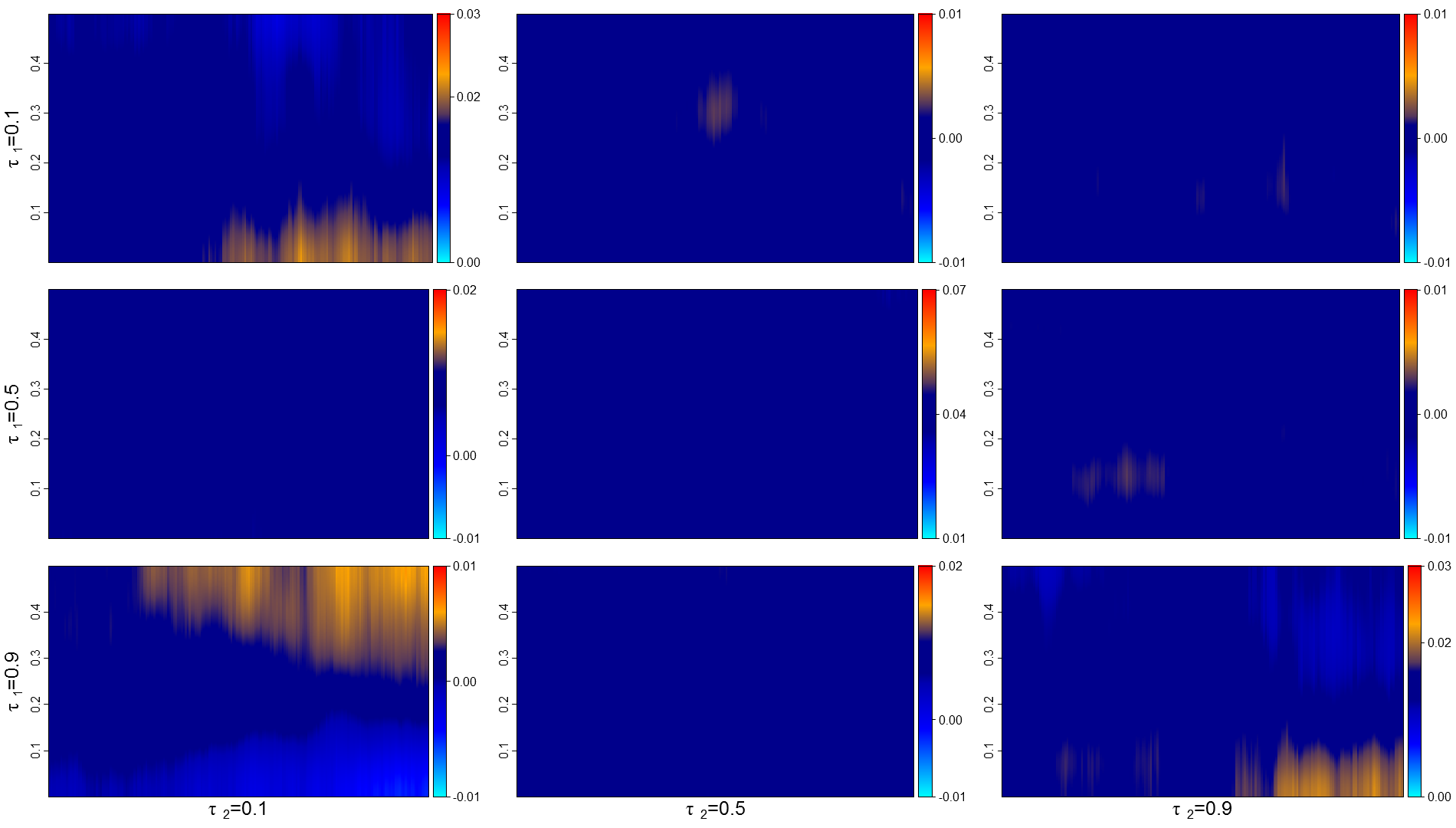}
		\subcaption{\footnotesize Estimated copula spectral densities, $n = 2048$}
	\end{minipage}
	\caption{\small Heatmaps of the Gaussian time-varying ARCH(1) process described in Section \ref{sec:p3} and the corresponding estimators, for 
		various window  lenghts.
		\vspace{-5mm}
	}\label{Figc}
\end{figure}

\subsection{Influence of the  series length} \label{app:shortts}
As suggested by a referee we want to include heatmaps of estimators calculated from shorter time series $T$. In our non-stationary setting, a smaller $T$ is essentially equivalent to a faster evolution of the features of the process under study. 
 Figure $\ref{Fig:short}$ compares estimators for the Cauchy tvAR$(2)$  and the tvQAR(1) processes (studied in Section $\ref{sec:p4}$)  for series lengths~$T=2^{12}=4096$ and $T=2^{11}=2048$ (one single realization), bandwidth $B_n = 10$ and window length $w=512$. The results indicate that estimation rapidly deteriorates with decreasing $T$. While nonstationarity nevertheless remains quite significant in the Cauchy tvAR$(2)$ case, the  signal in the real parts for the slowly varying tvQAR(1) is barely visible; time-irreversibility, on the other hand, remains well detected. 
 



\begin{figure}[H]
	\begin{minipage}[b]{.48\textwidth}
		\includegraphics[width = 72mm, height = 41mm]{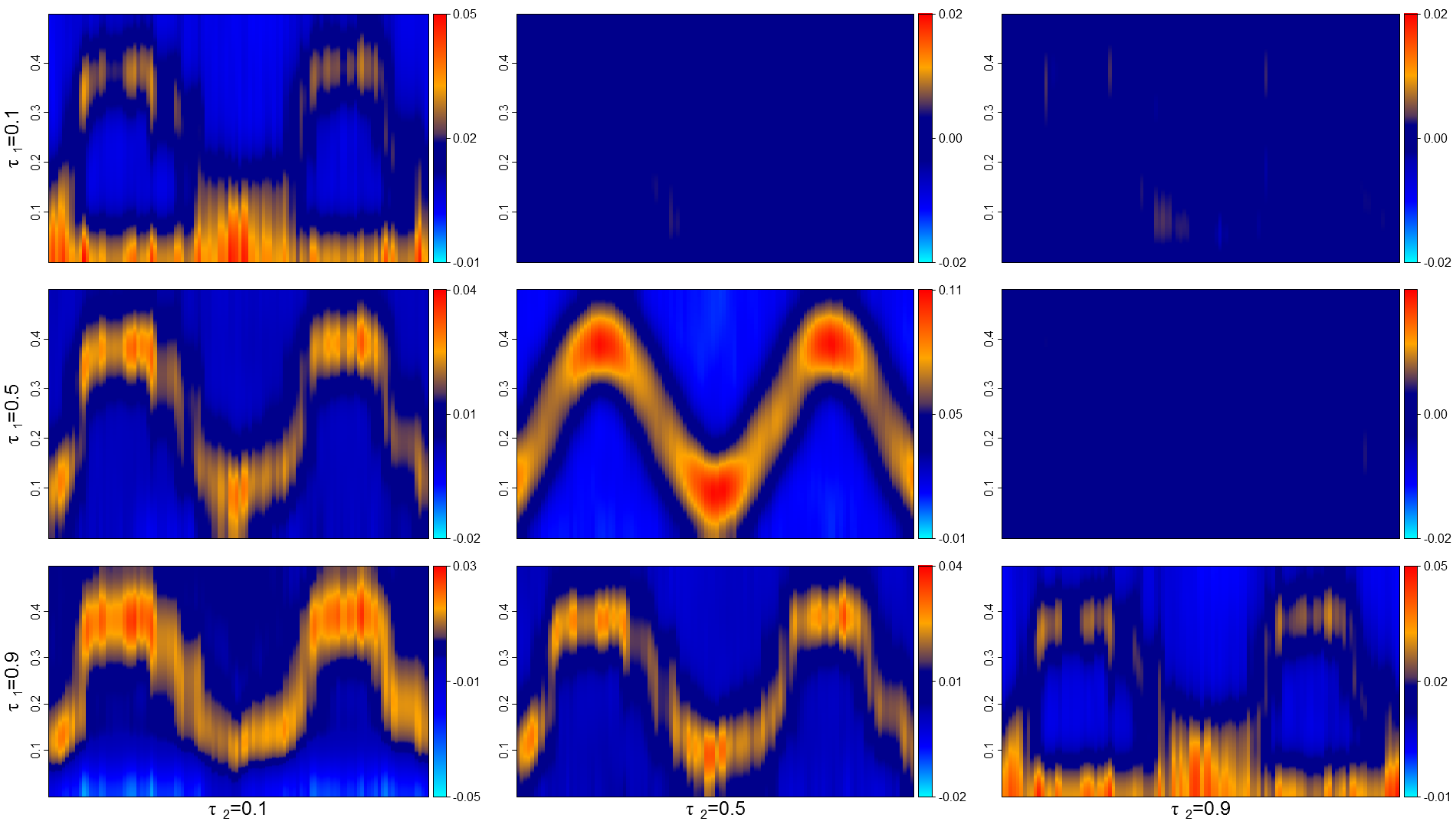}
		\subcaption{\footnotesize Time-varying Cauchy AR$(2)$ with $T=4096$}
	\end{minipage}
	\begin{minipage}[b]{.48\textwidth}
	\includegraphics[width = 72mm, height = 41mm]{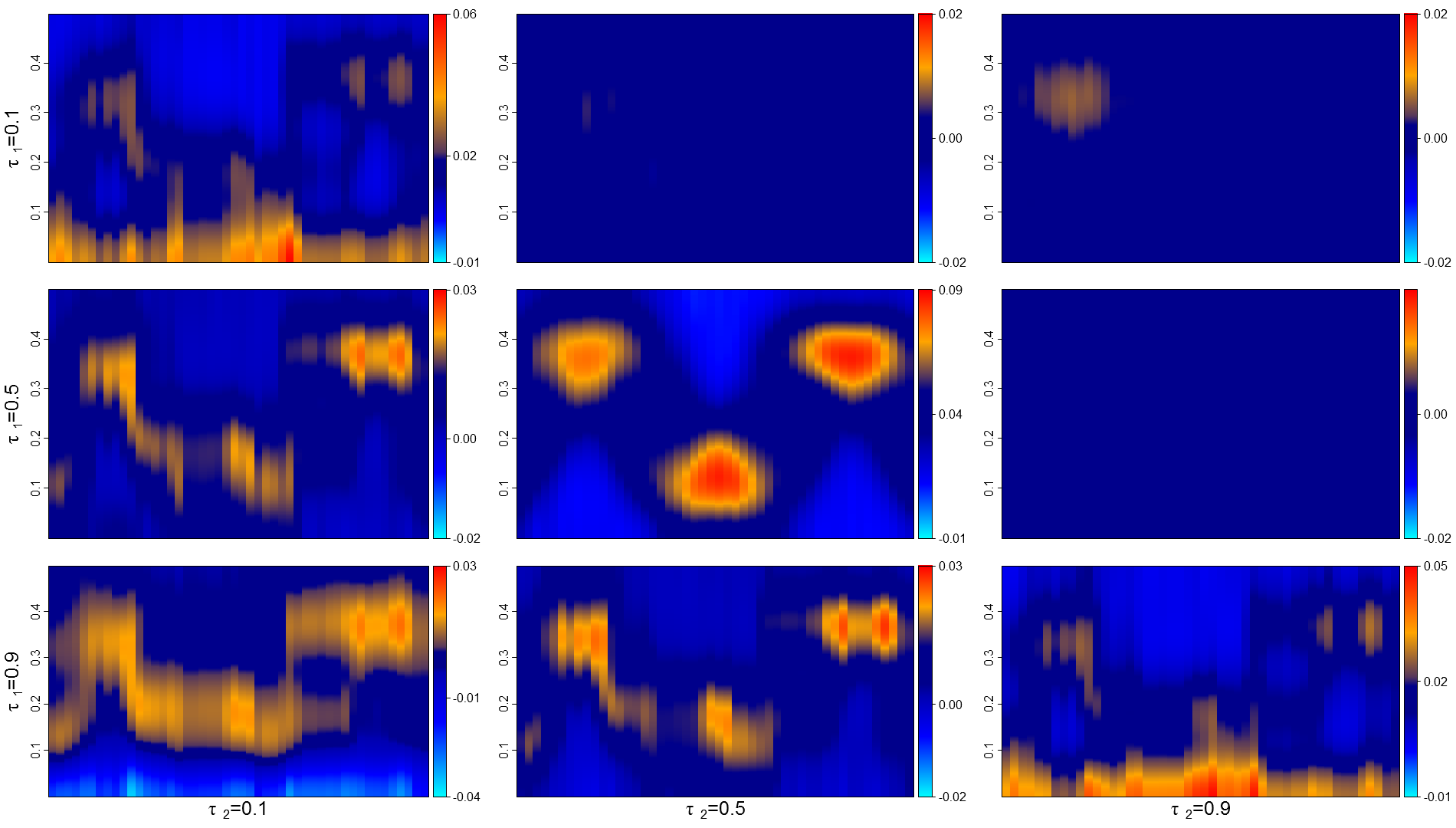}
		\subcaption{\footnotesize Time-varying Cauchy AR$(2)$ with $T=2048$}
	\end{minipage}
	\begin{minipage}[b]{.48\textwidth}
	\includegraphics[width = 72mm, height = 41mm]{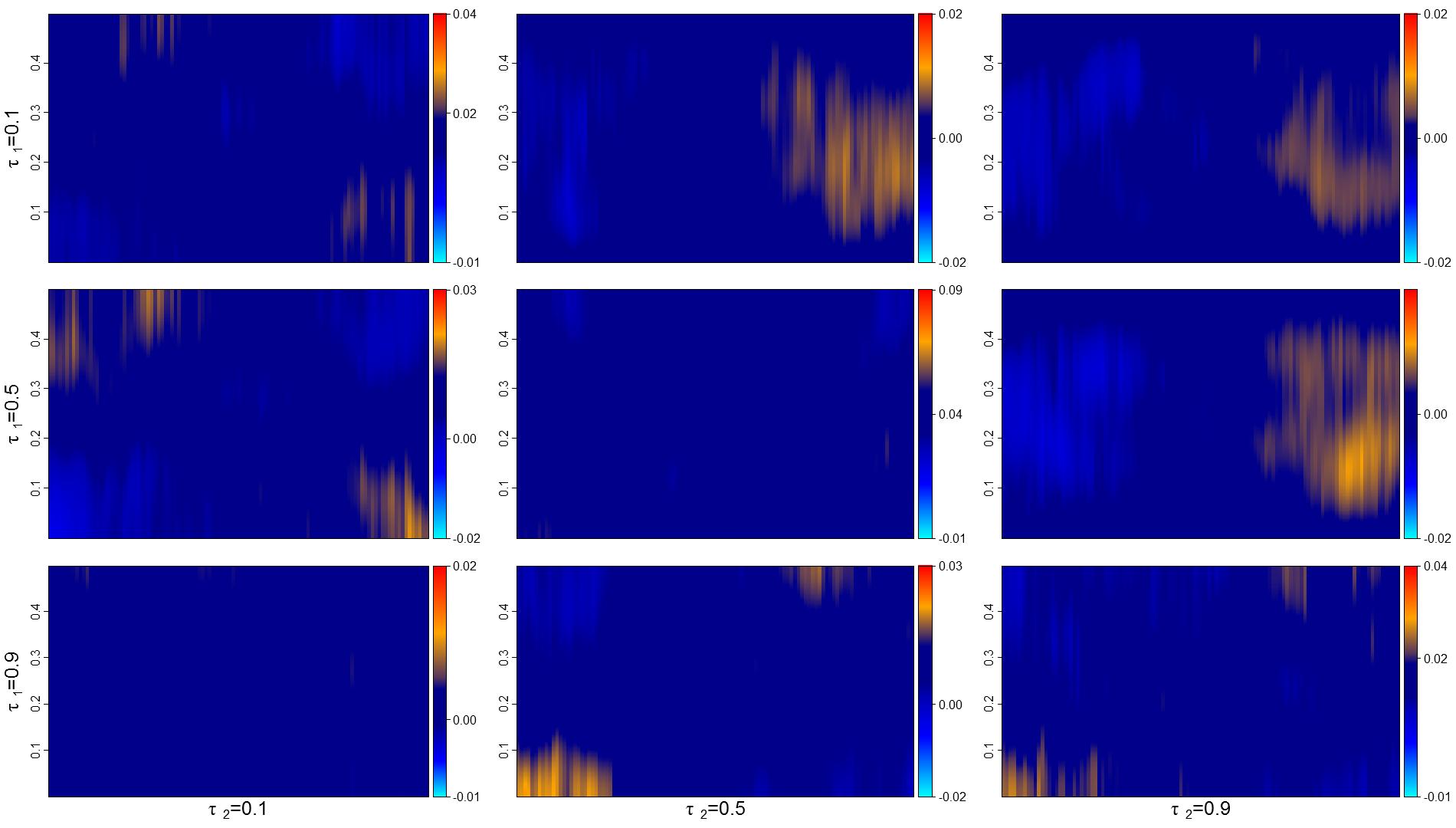}
		\subcaption{\footnotesize Time-varying QAR$(1)$ with $T=4096$}
	\end{minipage}
		\hspace{5mm}
	\begin{minipage}[b]{.48\textwidth}
	\includegraphics[width = 72mm, height = 41mm]{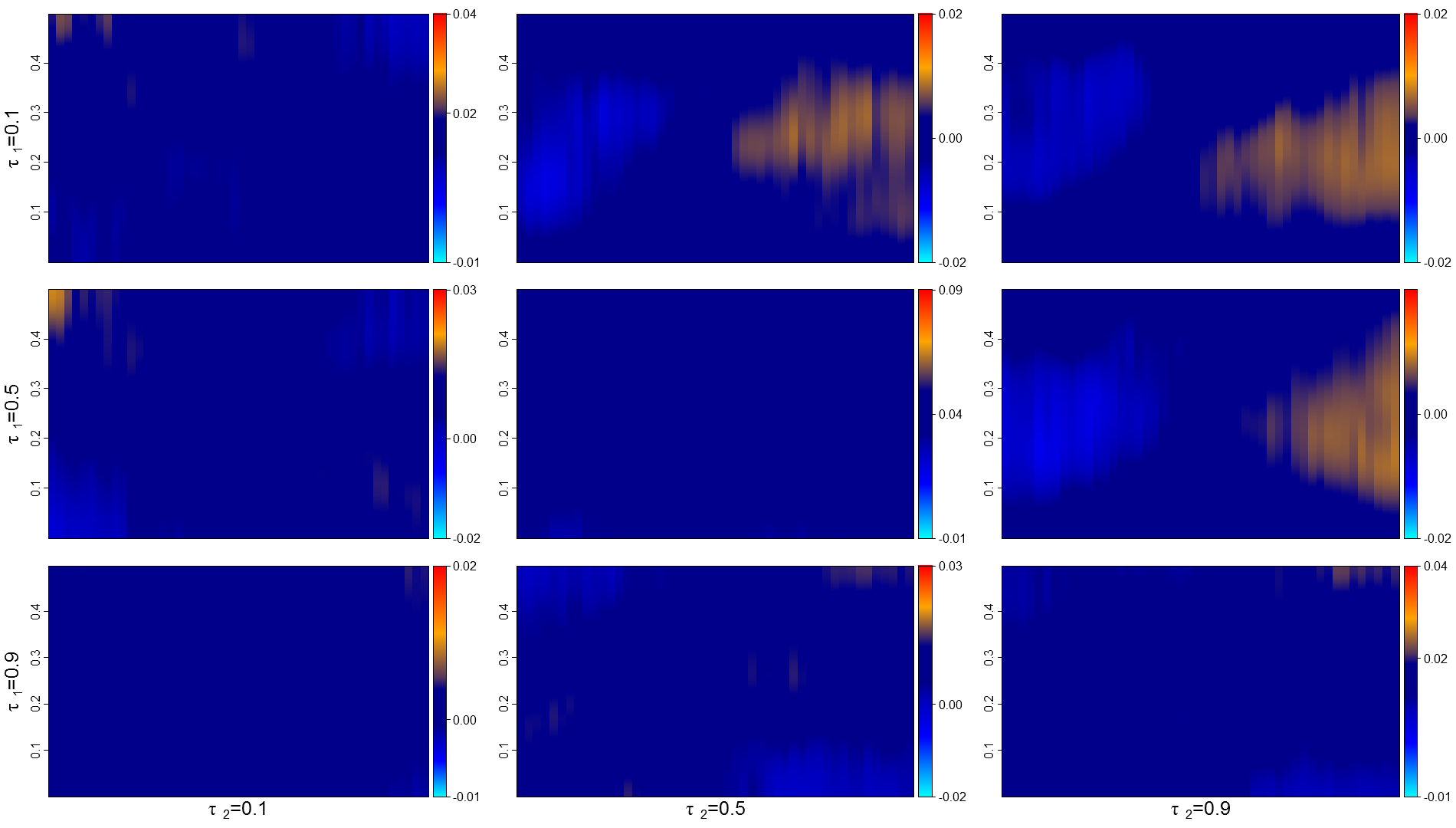}
		\subcaption{\footnotesize Time-varying QAR$(2)$ with $T=2048$}
	\end{minipage}
	\caption{\small Heatmaps of estimated spectra for the time-varying Cauchy AR$(2)$ and QAR$(1)$ processes, $T=2^{12}=4096$ and $T=2^{11}=2048$. 
		\vspace{-5mm}
	}\label{Fig:short}
\end{figure}

\section{ Time-varying variances and low frequency peaks in quantile spectra}

As mentioned in Section $\ref{SecSandP}$ and pointted out by \cite{li2014},  low frequency peaks in quantile spectral densities   also can be caused by  time-varying variances. 
By definition, copula spectra are invariant under strictly increasing transformations of the marginals. Invariance of the corresponding estimators, however, does not hold unless the pace of marginal changes   is slow   compared with the choice of   local window  lengths. Could  fastly varying marginal variances, via some  tvARCH model,  account for the type of quantile spectral plots associated with the S\&P500 data?

\subsection{A tvARCH(0) model approach for  S\&P500 log-returns} \label{sec:addsp500}

A plot of 
 the local variances of  S\&P500 returns (estimated in local windows) against time (Figure \ref{FigV})  suggests that those variances can be considered roughly stable over periods of at most  100 observations. This is quite small compared to the window length needed to estimate a quantile spectrum, a mismatch that could produce spurious deviations from white noise behavior in the heatmaps. 
 To investigate whether such fast marginal changes   can indeed produce the type of quantile spectral plots associated with the S\&P500 data, we followed a heuristic  approach inspired by \cite{li2014}. More precisely, based on  local windows of length~$101$ (details are provided in Section~\ref{tvARCHbest}), we first computed estimators $\hat\sigma^2_t$, $t = 1,...,12992$, of the time-varying variances~$ \sigma^2_t$.   With those estimated variances, we constructed an artificial series (of the tvARCH(0) form) 
\begin{equation} \label{eq:ARCH0}
X_{t,T} = \hat \sigma_t Z_t,
\end{equation} 
where $\{Z_t\}_{t=1}^{12992}$ denote i.i.d.\ draws with replacement from the S\&P500 values  $Y_t$ standardized by their estimated standard deviation, i.e.~from $\{Y_{t}/\hat \sigma_{t}\}_{t=1}^{12992}$ where $(Y_t)_{t=1}^{12992}$ denote the  observed S\&P500 returns. To see how well the time-varying copula spectrum of the S\&P500 returns can be matched by the spectrum of the process $X_{t,T}$, we simulated $J = 1000$ independent copies of~$X_{t,T}$ and, for each realization, we  computed local estimators of the quantile spectrum, $\hat f_{t_0,T}^{j}$,  say, $j=1,...,J$. Out of those $J=1000$ realizations, we selected one that produces quantile spectra matching 
those of the S\&P500 return---see \ref{tvARCHbest} for details. For the sake of brevity, we restrict our comparison to the real parts of the quantile combinations $(\tau_1,\tau_2) \in \{(0.1,0.1),(0.9,0.9)\}$ and the imaginary parts corresponding to~$(\tau_1,\tau_2) \in \{(0.1,0.9)\}$.   The corresponding time-frequency plots are shown in Figure~\ref{fig:heatsparch}. 
\begin{figure}[htbp]
	\centering
	\vspace{.2cm}
	\includegraphics[width = .8\textwidth]{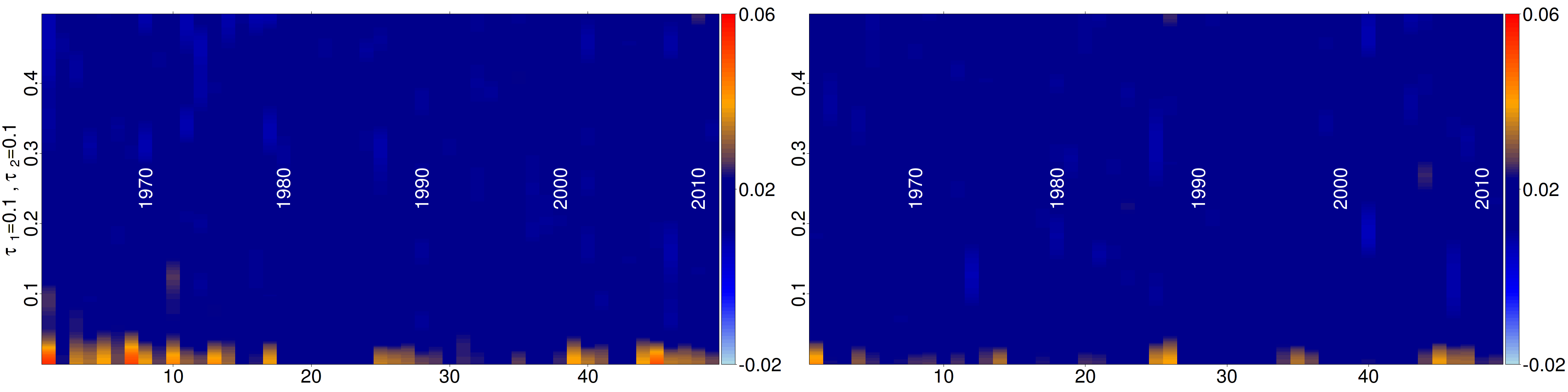}
	\vspace{.2cm}
	\includegraphics[width = .8\textwidth]{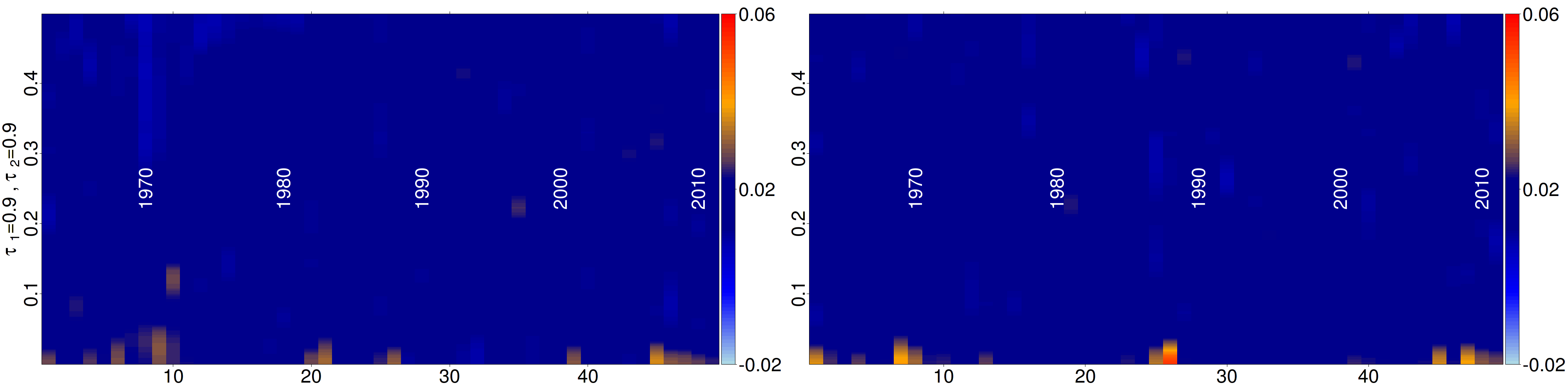}
	\vspace{.2cm}
	\includegraphics[width = .8\textwidth]{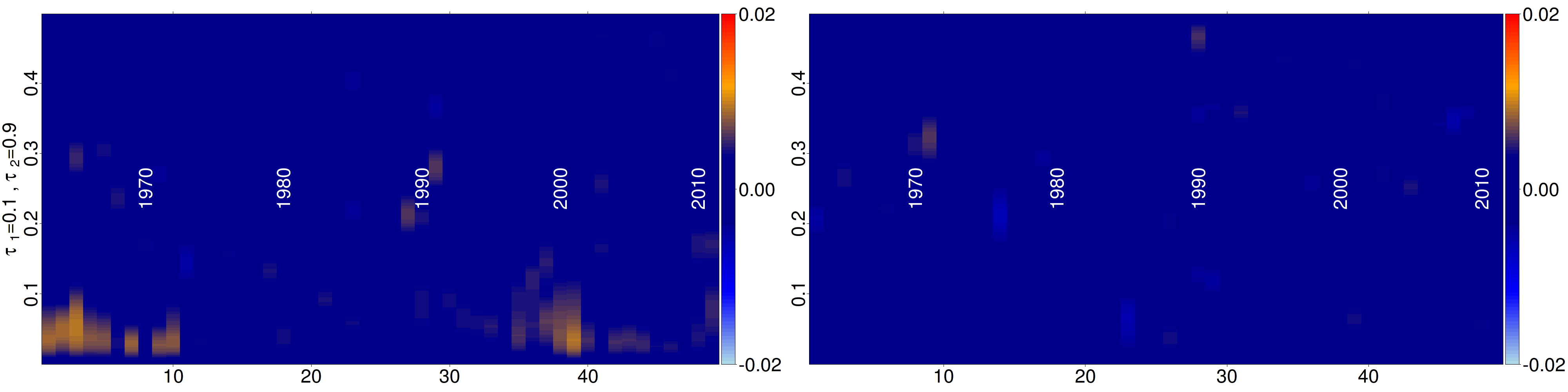}
	\caption{\small Time-frequency heatmaps for the S\&P500 log-returns (left column) and one realization of a tvARCH(0) (right column) process where the parameter is estimated as the time-varying variance of the S\&P500. First row: $\tau_1 = \tau_2 = 0.1$. Second row: $\tau_1 = \tau_2 = 0.9$. Third row: imaginary parts $(\tau_1 = 0.1, \tau_2 = 0.9)$. }
	\label{fig:heatsparch}
\end{figure}
Comparing the first two row figures corresponding to the quantile levels $(0.1,0.1)$ and $(0.9,0.9)$, we see that applying our estimators to the time-varying process $X_{t,T}$ indeed produces some peaks at low frequencies. However, those peaks, in the~$(0.1,0.1)$-spectrum, are not enough, and  not strong enough, thus failing to completely capture the strong dependence in negative  returns which is one of the typical features of financial data. The imaginary parts of the $(0.1,0.9)$ S\&P500 spectra (third row in Figure~Figure~\ref{fig:heatsparch}) shows significant deviations from white noise behavior, in particular during the periods  1965--1973 and~1995--2003. In contrast to that, the corresponding imaginary parts for the time-varying process $X_{t,T}$ are virtually indistinguishable from those of a white noise spectrum. This indicates that a tvARCH(0) model does not provide an adequate description of the joint distributions of high and low returns and additional evidence   that a tvARCH(0) model fails to capture important aspects of the S\&P500 dynamics.

\subsection{The ``best\  tvARCH$(0)$ fit"  of the S\&P500 data}\label{tvARCHbest}
Let us provide some details on the heuristic approach adopted in Section~\ref{sec:addsp500}. The time-varying variances $\sigma^2_t$ of the log-returns $Y_1,\dots,Y_T$ were estimated by
\[
\hat \sigma^2_t = \frac{1}{2n + 1} \sum_{s = 1}^T \bar{K}\big(\frac{s-t}{n}\big) \Big(Y_t - \frac{1}{2n+1} \sum_{|l-t|\leq n} Y_l\Big)^2 
\]
where $\bar{K}$ is the Parzen window multiplied with $4/3$ (so that it integrates to one), 
  $n=50$ (accomodating fast changes that are still sufficiently smooth---see Figure \ref{FigV}). Note that, to compute $\hat\sigma^2_t$ for $t = {1,\dots,12992},$ we need $13092$ observations of the S\&P500 from 1962-01-02 to 2014-01-03. Those additional observations were omitted in the rest of the paper to keep all pictures consistent. 
For our investigation, we created $J = 1000$ artificial time series as defined in $(\ref{eq:ARCH0})$ and for each of them, we calculated a collection of local lag-window estimators $\hat f_{t_0,T}^{j}(\omega,\tau_1,\tau_2),\ j=1,...,J$ by using the same windows and bandwidths as for the log-returns of the S\&P500. To select the ``best match", we concentrate on the real parts of~$\hat f_{t_0,T}^{j}(\omega,0.1,0.1)$ and~$\hat f_{t_0,T}^{j}(\omega,0.9,0.9),$ and  the imaginary part of $\hat f_{t_0,T}^{j}(\omega,0.9,0.1).$ Let $\hat f_{t_0,T}(\omega,\tau_1,\tau_2)$ stand for the local lag-window estimator of the S\&P500 log-returns, and consider  the $L^2$ distances 
\[ 
d_1(j) = \sum_{\omega \in \Omega} \sum_{t_0 \in {\cal T}_0}  [\Re\hat f_{t_0,T}^{j}(\omega,0.1,0.1))- \Re\hat f_{t_0,T}(\omega,0.1,0.1)]^2,\qquad j = 1,\dots,\ J.
\] 
Denote by $d_2(j)$ and $d_3(j)$ the same $L^2$  distances computed  for the real and  imaginary parts of~$\hat f_{t_0,T}^{j}(\omega,0.9,0.9)$ and~$\hat f_{t_0,T}^{j}(\omega,0.9,0.1)$ respectively. The ``best match"  was selected as the realization $j_{\text{min}}$   minimizing the sum $d_1(j)+d_2(j)+d_3(j).$

\begin{figure}[H]
	\centering
	\includegraphics[width = 0.6\textwidth]{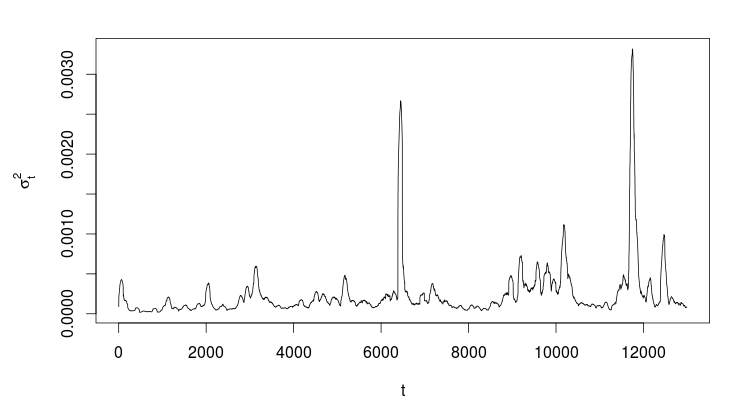}
	\caption{{\small Estimated time-varying variance of the log-returns of the S\&P500.  }}
	\label{FigV}
\end{figure}
\begin{figure}[ht]
	\centering
	\includegraphics[width = 0.48\textwidth]{Pictures/Data/SP500/SP500.png}
	\includegraphics[width = 0.48\textwidth]{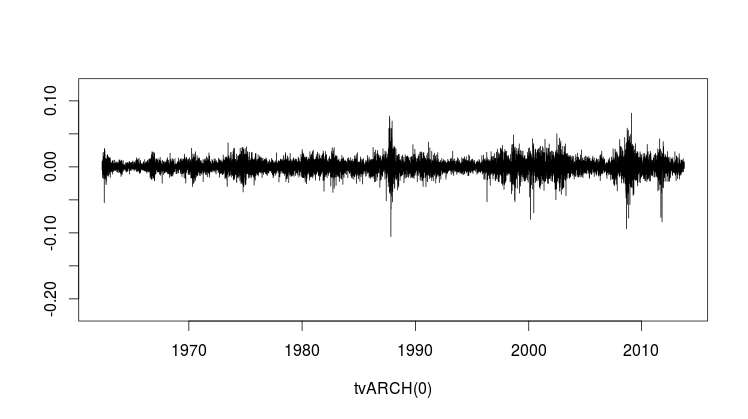}
	\caption{\small Log-returns of the S\&P500 between 1963 and 2013 and the simulated tvARCH$(0)$ process that was selected as the ``best match".}
	\label{fig:tscompared}
\end{figure}

\subsection
{Comparing  the ``global"  tvARCH(0) and  S\&P500 spectra}

If a tvARCH(0) approach to the study of the S\&P500 data is to be adopted, one may argue that, in view the invariance argument mentioned at the beginning of this section, the localized spectral analysis developed in this paper is not required, and that the  stationary  or ``global"  methods developed in  \cite{dhkv2014} are the appropriate ones.  The right check for the adequacy of a  tvARCH(0) model then  should be based on a comparison  between  (estimated) stationary quantile spectra, which avoids the trouble caused by a possible mismatch between the pace of  marginal changes and the chosen window length. 

Accordingly, in this section, we provide a comparison of the ``global"  spectra, i.e.~spectra computed  from the complete dataset,  without localization, as in   \cite{dhkv2014}, of the S\&P500 returns on one hand,  of the process $X_{t,T}$ defined in equation \eqref{eq:ARCH0} on the other hand. We simulated $J=1000$ independent replications of the process $X_{t,T}$, and for each replication we computed the ``global"  lag-window estimator based on~$
\mathcal{N}_{\tnull,T}~:=~\{1,\dots,T\}
$, the bandwidth $B_n = 25$ and the same lag-window function as in the analysis of the S\&P500 log returns. 
This yields  a collection of estimators $(\hat f^{j})_{j=1,...,J}$.  
Next, for each frequency~$\omega$ in~$ \{\frac{2 \pi j}{6496}|j = 0,\dots,6495\}$ and each couple $\tau_1,\tau_2$ in $\{0.1,0.5,0.9\},$ we computed the~$1\%$ quan\-tile~$q_{\text{min}}^\Re(\omega,\tau_1,\tau_2)$ of the $J$-tuple $(\Re\hat f^{j}(\omega,\tau_1,\tau_2))_{j=1,...,J}$ and the $99\%$ quantile $q_{\text{max}}^\Re(\omega,\tau_1,\tau_2)$  of the $J$-tuple $(\Re\hat f^{j}(\omega,\tau_1,\tau_2))_{j=1,...,J}$. The quantiles $q_{\text{max}}^\Im(\omega,\tau_1,\tau_2)$ and  $q_{\text{max}}^\Im(\omega,\tau_1,\tau_2)$ were computed similarly. 

\begin{figure}[ht]
	\centering
	\includegraphics[width = \textwidth]{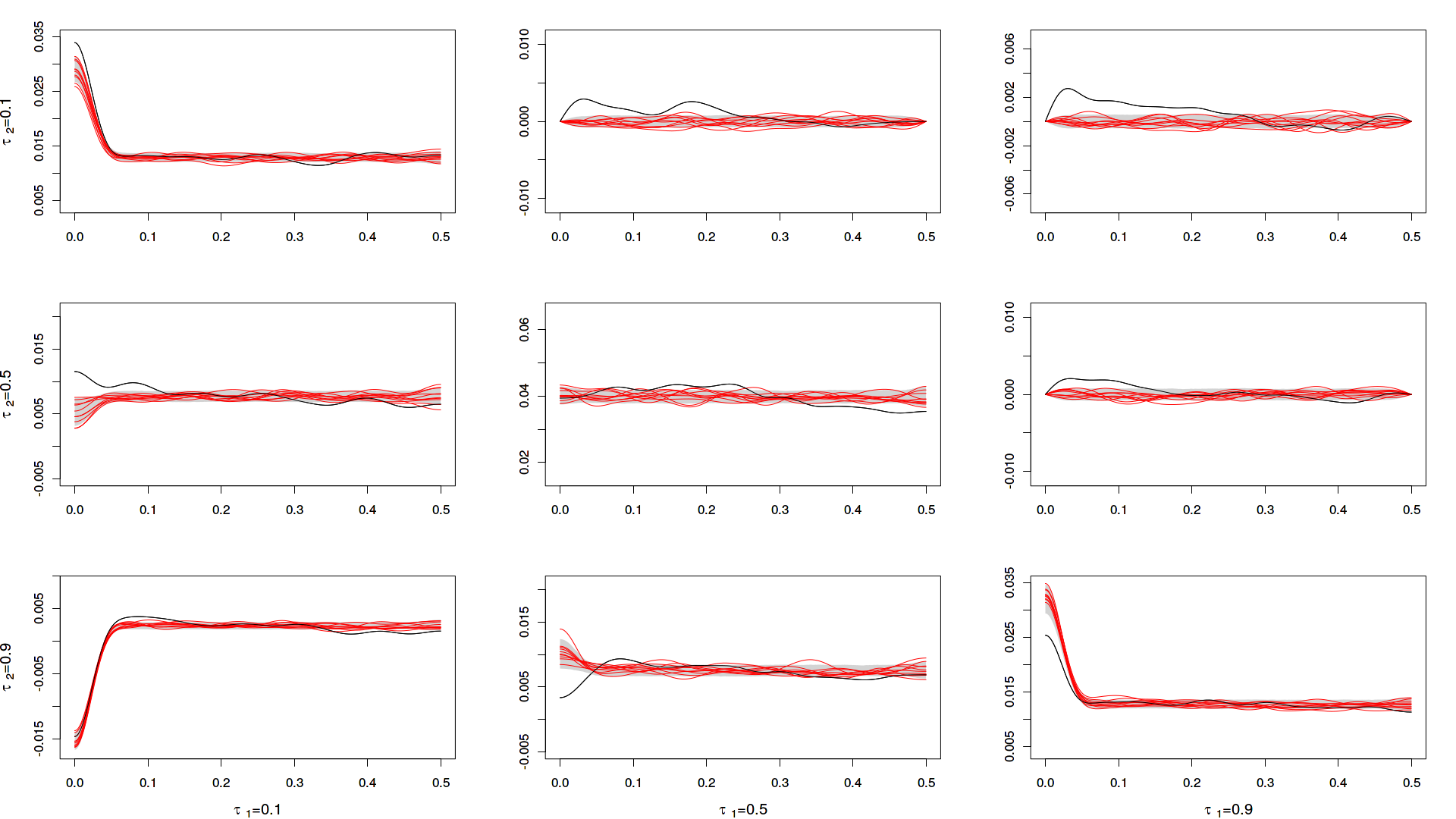}
	\caption{\small A comparison of global lag-window estimations of the quantile spectra of the S\&P500 (over the period 1962-2013; black lines) and those obtained from simulated tvARCH$(0)$ processes (red lines, with grey-shaded pointwise confidence regions).}
	\label{fig:ap500arch0_1}
\end{figure}

Quantile spectra computed from the S\&P500 dataset (in black)  are depicted in Figure \ref{fig:ap500arch0_1}, together with the estimators $\hat f^{j}(\omega,\tau_1,\tau_2),\ j=1,...,10$ (red lines) and, for each~$(\omega,\tau_1,\tau_2)$, a gray area covering the interval $[q_{\text{min}},q_{\text{max}}].$ As predicted by the analysis of \cite{li2014}, the~$\tau_1=\tau_2=0.1$ and $\tau_1=\tau_2=0.9$ quantile spectral densities of the tvARCH(0) process exhibit prominent peaks around frequency zero. Observe, however, that those peaks are typically higher than those of the S\&P500 for $\tau_1=\tau_2=0.1$ and (slightly) lower for $\tau_1=\tau_2=0.9$. Additionally, the estimators $\Re \hat{f}(\omega,\tau_1,\tau_2)$ for $(\tau_1,\tau_2) = (0.1,0.5)$, and $ (0.9,0.5)$ and $\Im \hat{f}(\omega,\tau_1,\tau_2)$ for \mbox{$(\tau_1,\tau_2) = (0.1,0.5), (0.1,0.9)$} and $(0.5,0.9)$ lie well outside the gray ``confidence" areas for a wide range of frequencies. This again indicates that the tvARCH(0) process does not provide an adequate description 
of the (global) dynamic features of the S\&P500 returns.  

\begin{figure}[ht]
	\centering
	\includegraphics[width = \textwidth]{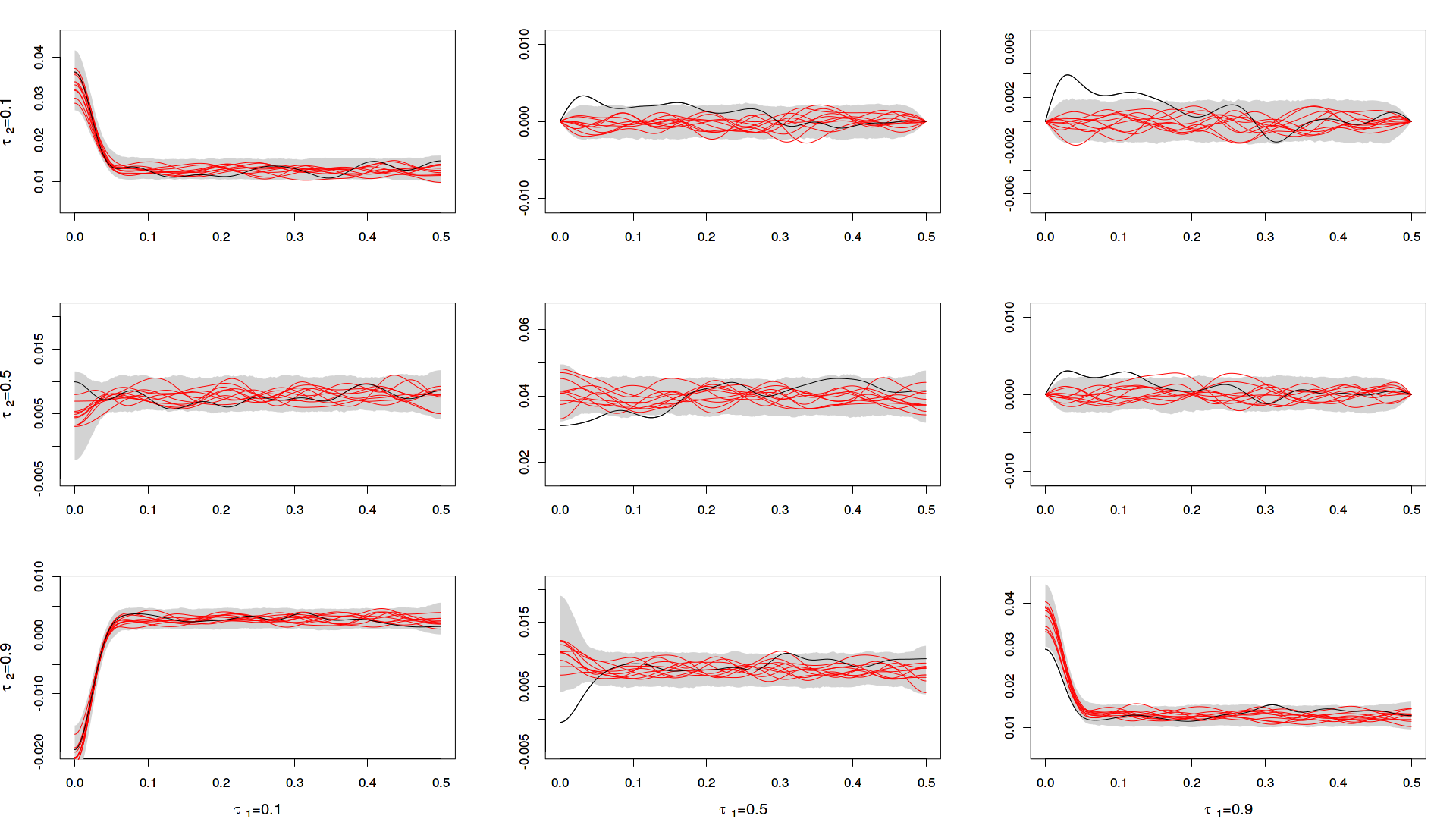}
	\caption{\small Same as Figure \ref{fig:ap500arch0_1}, with S\&P500 observations restricted to the period 2000-2013.}
	\label{fig:ap500arch0_2}
\end{figure}

It is natural to wonder whether the mismatch between this simple time-varying variance model and the S\&P500 data is due to a structural difference between the  S\&P500 dynamics back in the seventies and its more recent dynamics.  
 To address this question, we replicated the procedure described above for the S\&P500 returns in the time period 2000-2013. The corresponding results are displayed in Figure \ref{fig:ap500arch0_2}. While the estimators of the~$(0.1,0.1)$ and $(0.9,0.9)$ spectra  now (just barely) lie within the gray ``confidence" areas, we still observe highly significant deviations between the imaginary part of the ARCH(0) and S\&P500 spectra for the quantile combination $(\tau_1,\tau_2) = (0.1,0.9)$, and between their real parts for~$(\tau_1,\tau_2) = (0.5,0.9)$.

Summarizing our findings, a localized analysis of quantile spectra reveals features that cannot be explained by a model based solely on fast local changes in variance. Note that more sophisticated models such as time-varying ARCH and GARCH processes have been suggested to describe financial data, see \cite{Fryzlewicz_tvARCH} for a recent contribution. It would be very interesting to compare the quantile spectra of such time-varying stochastic volatility processes with those of the S\&P500 returns. Preliminary comparisons indicate that the structure of the imaginary parts of the S\&P500 time series cannot be explained by  time-varying GARCH models.


\section{Proofs for the main results}\label{app:main}

\subsection{Proofs of the results in Section \ref{sec:strictstatmodels}} \label{sec:proofstrictstatmodels}

We begin by some preliminary comments. Assume that $X_{t,T}$ and $X^\vartheta_t$ are defined on the same probability space.  Expressing   distributions function in terms of expectations and indicators leads to
\begin{align*}
&|F_{t_1,t_2;T}(x_1,x_2)-G_{t_2-t_1}^\vartheta((x_1,x_2))| \\
& \quad = |\E(\prod_{1\leq j \leq 2} \Ind{X_{t_j,T} \leq x_j} - \prod_{1\leq j \leq p} \Ind{(X^\vartheta_{t_j}) \leq x_j})|\\
&  \quad \leq \sum_{k=1}^2 | \E( \prod_{j=1}^{k-1} \Ind{X_{t_j,T} \leq x_j} [\Ind{X_{t_k,T} \leq x_k} - \Ind{(X^\vartheta_{t_k})\leq x_k}] \prod_{j=k+1}^{p}\Ind{(X^\vartheta_{t_j})\leq x_j})|\\
& \quad  \leq \sum_{k=1}^2 \E(|\Ind{X_{t_k,T} \leq x_j} - \Ind{(X^\vartheta_{t_k})\leq x_j}|).
\end{align*}
Therefore, in order to prove Lemma \ref{lem:ma}-\ref{lem:garch}, it is sufficent to show that 
\[
\sup_{x \in \mathbb{R}} \E(|\Ind{X_{t,T} \leq x} - \Ind{(X^\vartheta_{t})\leq x}|) \leq L\big(\big|\frac{t}{T} - \vartheta\big| + \frac{1}{T}\big).
\]

\subsubsection{Proof of Lemma $\ref{lem:ma}$}
By Assumption $(MA1)$, we have $\sup_{\vartheta \in (0,1)} \sum_{j=-\infty}^{\infty} |a(\vartheta,j)| < \infty$, which by standard arguments implies strict stationarity of the process $X^\vartheta_{t}$ (see for example  Proposition $3.1.1$ in \cite{BrockwellDavis}). Without loss of generality, we can assume that~$\mu(\vartheta) = 0$. In order to establish distributional properties, we always can specialize the noise $\zeta_t$ driving $X^\vartheta_{t}$--an arbitrary copy of the noise $Z_t$ driving $X_{t,T}$---as being $Z_t$ itself. Denoting by~$\mathcal{A}$ the~$\sigma-$field generated by~$\{\xi_s|s \neq 0\},$
\begin{align*}
&\sup_{x \in \mathbb{R}} \E(|\Ind{X_{t,T} \leq x} - \Ind{(X^\vartheta_{t})\leq x}|) \\
&\qquad \leq \sup_{x \in \mathbb{R}} \E \big[\E(| \Ind{X_{t,T} \leq x } - \Ind{X^\vartheta_{t} \leq x }| \big| \mathcal{A} ) \big]\\
&\qquad \leq \sup_{x \in \mathbb{R}} \E \big[\E(| \Ind{\xi_t \leq \frac{1}{a_{t,T}(0)}\{x - \sum_{j\neq 0} a_{t,T}(j) \xi_{t -j}\}} - \Ind{\xi_t \leq \frac{1}{a(\vartheta,0)}\{x - \sum_{j\neq 0} a(\vartheta,j) \xi_{t -j}\}}| \big| \mathcal{A} )\big] \\
& \qquad =  \sup_{x \in \mathbb{R}} \E[ | F_{\xi}(\frac{1}{a_{t,T}(0)}\{x - \sum_{j\neq 0} a_{t,T}(j) \xi_{t -j}\}) - F_{\xi}(\frac{1}{a(\vartheta,0)}\{x - \sum_{j\neq 0} a(\vartheta,j) \xi_{t -j}\}) | ]\\
& \qquad \leq \E[C_1|{a_{t,T}(0)}-{a(\vartheta,0)}|+C_2||S_{t,T} - S_t^\vartheta|],
\end{align*}
where
\[
S_{t,T} := \sum_{j \neq 0 } a_{t,T}(j) \xi_{t -j}, \quad
S_t^\vartheta := \sum_{j \neq 0} a(\vartheta,j) \xi_{t -j},
\]
and the last inequality follows from the two dimensional mean-value theorem. To be more precise, 
\begin{multline*}
\Big|F_\xi\left(\frac{x-v}{u}\right) - F_\xi\left(\frac{x-v'}{u'}\right)\Big| \\*
\leq \Big|\int_0^1 f_\xi\left(\frac{x-v_t}{u_t}\right)\frac{x-v_t}{u^2_t}dt\Big| |u-u'| + \Big|\int_0^1 f_\xi\left(\frac{x-v_t}{u_t}\right)\frac{1}{u_t}dt\Big| |v-v'|,
\end{multline*}
with $u_t = u + t(u'-u)$ and $v_t = v + t(v'-v).$ From Assumption (MA2) the integrals are bounded by constants $C_1$ and $C_2$ which are independent of $x.$
Straightforward calculations, under the assumptions made, lead to
\[
\E[ |S_{t,T} - S_t^\vartheta| ] = \text{O}(|t - \vartheta T^{-1}| +T^{-1}),
\]
which completes the proof.
\hfill $\qed$

\subsubsection{Proof of Lemma $\ref{lem:arch}$}
Observe that Assumptions (ARCH1) imply that
\begin{equation} \label{eqn:strstatARCH}
\sup_{\vartheta \in (0,1)} \E(\sum_{j=1}^\infty a_j(\vartheta)Z^2_j) < 1,
\end{equation}
which is sufficient for the existence and uniqueness of a strictly stationary solution $(X^\vartheta_{t})^2$ with finite first moment(see \cite{Giraitis2000} Theorem $2.1$). Now $(ARCH1)$ yields $\sum_j |a_j(\vartheta)| < \infty,$ which implies that $(\sigma_t^\vartheta)^2$ is locally strictly stationary (Proposition $3.1.1$ of \cite{BrockwellDavis} again). Let $\sigma_t^\vartheta = \sqrt{(\sigma_t^\vartheta)^2}$ and set $X^\vartheta_{t} = \sigma^\vartheta_{t}Z_t.$
To prove local strict stationarity, it suffices to bound
$\sup_{x \in \mathbb{R} } \E(|\Ind{X_{t,T} \leq x} - \Ind{X^\vartheta_{t}\leq x}|).$
Denoting by $\mathcal{A}_t$ the $\sigma-$algebra generated by $(Z_t,Z_{t-1},\dots),$ observe that
\begin{flalign*}
\E\big(\big|\Ind{X_{t,T} \leq x} - \Ind{X^\vartheta_{t}\leq x}\big|\big) &= \E\big(\E\big[\big|\Ind{X_{t,T} \leq x} - \Ind{X^\vartheta_{t}\leq x}\big| \Big\vert \mathcal{A}_{t-1}\big]\big) \\
= \E\big(\E\big[\big|\Ind{Z_t \leq x/\sigma_{t,T}} - \Ind{Z_t \leq x/\sigma^\vartheta_{t}}\big| \Big\vert \mathcal{A}_{t-1}\big]\big)&\leq \E\big(\big|F({x}/{\sigma_{t,T}}) - F({x}/{\sigma^\vartheta})\big|\big)\\
= \E\big(\big|\int_{\sigma_{t,T}}^{\sigma^\vartheta_{t}} {x}{y}^{-1}f({x}{y}^{-1})y^{-1}dy\big|\big) &\leq \E\big(C\big|{\sigma_{t,T}}-{\sigma^\vartheta_{t}}\big|\big),
\end{flalign*}
where the last inequality follows from $(ARCH2)$ and the fact that $\min({\sigma_{t,T}},{\sigma^\vartheta_{t}}) > \rho.$  Now, as $Z_t$ is independent of $(\sigma_{t,T},\sigma^\vartheta_{t}),$ we have
\begin{align*}
&\E(|\sigma_{t,T}^2 - (\sigma_t^\vartheta)^2|) = \E[Z_t^2 (|\sigma_{t,T}^2 - (\sigma_t^\vartheta)^2|)] = \E(|X_{t,T}^2 - (X_t^\vartheta)^2|) \leq \leq C\Big(\big|\frac{t}{T} - \vartheta\big| + \frac{1}{T}\Big)
\end{align*}
where the last inequality follows from Theorem 1 in \cite{dahlrao2006}. Noting again that $\min({\sigma_{t,T}},{\sigma^\vartheta_{t}})$ is bounded away from zero, we have, for some appropriate constant~$C$, 
$
\E(|{\sigma_{t,T}}-{\sigma^\vartheta_{t}}|) \leq C\E(|{\sigma^2_{t,T}}-{(\sigma^\vartheta_{t})^2}|) , 
$
which concludes the proof.\hfill$ \qed$

\subsubsection{Proof of Lemma $\ref{lem:garch}$}

Assumption $(GARCH1)$ implies that 
$\sup_{u \in [0,1]} \E(Z^2_t)\Big[\sum_{j=1}^{p}a_j(u) + \sum_{i=1}^{q} b_j(u)\Big] < 1$, 
which is sufficent for the existence of a strictly stationary solution $X^\vartheta_t$ (see the Remark after Corollary 2.2 in \cite{Bugerol1992}).
Similar calculations as in the tvARCH case yield
$\E(|\Ind{X_{t,T} \leq x} - \Ind{X^\vartheta_{t}\leq x}|) \leq \E(C|{\sigma^2_{t,T}}-{(\sigma^\vartheta_{t}})^2)$, 
and a bound on $\E(|{\sigma^2_{t,T}}-{(\sigma^\vartheta_{t})^2}|)$ follows as in Section $5.2.$ (top of page $1168$) of \cite{SubbaRaoGARCH}.
\hfill $\qed$

\subsection{Proof of Theorem \ref{thm:asymain}}

The proof proceeds in several steps, which we briefly outline here; details are provided in Section~\ref{sec:proofsmainsteps}. 
  First, we establish that the estimator $\hat q_{t_0,T}(\tau)$ can be replaced with the true quantile levels $\tau$, that is,
\begin{flalign}  \label{eq:main1}
&\nonumber \hat{\mathfrak{f}}_{t_0,T}(\omega_n,\tau_1,\tau_2) = \frac{1}{2\pi}\sum_{|k|\leq n-1} K_n(k) e^{-i \omega_n k} \\ 
&\qquad \times\frac{1}{n} \sum_{t \in T(k)} \Big(\I{X_{t,T} \leq  q^\unu(\tau_1)} - \tau_1 \Big)\Big( \I{X_{t+k,T} \leq q^\unu(\tau_2)} - \tau_2 \Big) + o_P((B_n/n)^{1/2}),
\end{flalign}
uniformly in $\omega_n \in \tilde{\mathcal{F}}_n(\eps),$ where $\tilde{\mathcal{F}}_n(\eps)$ denotes the set of all Fourier frequencies in the interval~$(\eps,\pi-\eps)$. Second, we prove that, uniformly again in $\omega_n \in \tilde{\mathcal{F}}_n(\eps)$,
\begin{equation} \label{eq:main2}
\hat{\mathfrak{f}}_{t_0,T}(\omega_n,\tau_1,\tau_2) = \tilde{\mathfrak{f}}_{t_0,T}(\omega_n, \tau_1, \tau_2) + O_P\Big(\frac{B_n^2}{n} + \frac{n^2}{T^2B_n} + \frac{n^{1/2}B_n}{T} \Big)
\end{equation}
where
\[
\tilde{\mathfrak{f}}_{t_0,T}(\omega_n, \tau_1, \tau_2) := \frac{1}{2\pi}\frac{1}{n}\sum_{|k|\leq n-1} K_n(k) e^{-i \omega_n k} \sum_{|t-t_0| \leq m_T-B_n} Y_{t,\tau_1} Y_{t+k,\tau_2}
\]
and $Y_{t,\tau} := \I{X_{t,T} \leq q^\unu(\tau)} - F_{t;T}(q^\unu(\tau))$. The advantage of this representation lies in the fact that the random variables $\I{X_{t,T} \leq q^\unu(\tau_1)} - F_{t;T}(q^\unu(\tau_1))$ are centered, which considerably simplifies some of the computations that follow. Next, observe that
\[
\tilde{\mathfrak{f}}_{t_0,T}(\omega_n, \tau_1, \tau_2) = \tilde{\mathfrak{f}}_{t_0,T}(\omega, \tau_1, \tau_2) + O_P(B_n/n)
\]
since $|\omega_n-\omega| = O(1/n)$. Finally, we  prove that 
\begin{equation} \label{eq:asynorm}
\sqrt{B_n/n} \left(
\begin{array}{c}
\Re \tilde{\mathfrak{f}}_{t_0,T}(\omega, \tau_1, \tau_2) - \Re \E \tilde{\mathfrak{f}}_{t_0,T}(\omega, \tau_1, \tau_2) \\
\Im \tilde{\mathfrak{f}}_{t_0,T}(\omega, \tau_1, \tau_2) - \Im \E \tilde{\mathfrak{f}}_{t_0,T}(\omega, \tau_1, \tau_2)
\end{array}
\right) \weak \mathcal{N}(0,\Sigma^2(\omega,\tau_1,\tau_2))
\end{equation}
and
\begin{align} \label{eq:bias}
\nonumber \E \tilde{\mathfrak{f}}_{t_0,T}(\omega, \tau_1, \tau_2) = & \mathfrak{f}^\unu(\omega,\tau_1,\tau_2) - C_K(r)B_n^{-r} \mathfrak{d}^{(r)}_\omega\mathfrak{f}^\unu(\omega,\tau_1,\tau_2)\\
& \qquad + \frac{n^2}{2T^2}\frac{\partial^2}{\partial u^2} \mathfrak{f}^u(\omega,G^{u}(q^\unu(\tau_1)),G^{u}(q^\unu(\tau_2)))\Big|_{u=\unu}
\\ \nonumber
& \qquad + o(B_n^{-r} + n^2/T^2) + O(1/n),
\end{align}
which completes the proof of the theorem. \hfill $\qed$


\section{Some probabilistic details} \label{app:prode}

\subsection{A Lemma on cumulants}

\begin{lemma}\label{lem:mix}
	For an arbitrary stochastic process $(X_t)_{t \in \mathbb{Z}}$, let
	\[
	\alpha(n) := \sup_{t \in \mathbb{Z}} \sup_{A \in \sigma(..X_{t-1},X_{t}), B\in\sigma(X_{t+n},X_{t+n+1},...)} |\p(A \cap B) - \p(A)\p(B)|.
	\]
	Then, for any $t_1,...,t_p \in \mathbb{Z}$ and any p-tuple Borel sets $A_1,...,A_p$ there exists a constant $K_p$ depending on $p$ only such that
	\[
	\Big|\cum(I\{X_{t_1} \in A_1\}, \ldots, I\{X_{t_p} \in A_p\})\Big| \leq K_p \alpha\Big( \big\lfloor p^{-1} \max_{i,j} |t_i - t_j| \big\rfloor \Big).
	\]
\end{lemma}

\noindent\textbf{Proof.} Recall that, by the definition of cumulants, 
\begin{align}
& |\cum(I\{X_{t_1} \in A_1\}, \ldots, I\{X_{t_p} \in A_p\})| \nonumber \\
& = \Big| \sum_{\{\nu_1, \ldots, \nu_R\}} (-1)^{R-1} (R-1)!
\p\Big( \bigcap_{i \in \nu_1} \{X_{t_i} \in  A_i\} \Big) \cdots \p\Big(\bigcap_{i \in \nu_R} \{X_{t_i} \in  A_i\} \Big) \Big|, \label{lem:OrderCumYinA:eqn:ProbsCase1}
\end{align}
where the summation runs over all partitions $\{\nu_1, \ldots, \nu_R\}$ of the set $\{1,\ldots,p\}$.
In the case~$t_1=...=t_p,$ the Lemma is obviously true. If at least two indices are distinct, choose $j$ with
$
\max_{i=1,\ldots,p-1} (t_{i+1} - t_i) = t_{j+1} - t_j > 0
$
and let~$(Y_{t_{j+1}}, \ldots, Y_{t_p})$ be a random vector that is independent of $(X_{t_1}, \ldots, X_{t_j} )$ and possesses the same joint distribution as $(X_{t_{j+1}}, \ldots, X_{t_p})$. By an elementary property of the cumulants (cf. Theorem~2.3.1~(iii) in \cite{brill}), we have
\[\cum\big(I\{X_{t_1} \in A_1\}, \ldots, I\{X_{t_j} \in A_j\}, I\{Y_{t_{j+1}} \in A_{j+1}\}, \ldots, I\{Y_{t_p} \in A_p\}\big) = 0.\]
Therefore, we can write, for the cumulant of interest,
\begin{equation*}
\begin{split}
& \Big|\cum(I\{X_{t_1} \in A_1\}, \ldots, I\{X_{t_p} \in A_p\}) \\
& \quad
- \cum(I\{X_{t_1} \in A_1\}, \ldots, I\{X_{t_j} \in A_j\}, I\{Y_{t_{j+1}} \in A_{j+1}\}, \ldots, I\{Y_{t_p} \in A_p\})\Big| \\
& = \Big| \!\!\sum_{\{\nu_1, \ldots, \nu_R\}} (-1)^{R-1} (R-1)! [P_{\nu_1} \cdots P_{\nu_R} - Q_{\nu_1} \cdots Q_{\nu_R}] \Big|,
\end{split}
\end{equation*}
where the sum again runs over all partitions $\{\nu_1, \ldots, \nu_R\}$ of $\{1,\ldots,p\}$,
\begin{equation*}
P_{\nu_r}\! := \p\Big( \bigcap_{i \in \nu_r} \{X_{t_i} \in A_i\} \Big)  \ \text{ and }\
Q_{\nu_r}\! := \p\Big( \bigcap_{\substack{i \in \nu_r\\i \leq j}} \{X_{t_i} \in A_i\} \Big)
\p\Big( \bigcap_{\substack{i \in \nu_r\\i > j}} \{X_{t_i} \in A_i\} \Big),
\end{equation*}
$r=1,\ldots,R$, with $\p(\bigcap_{i \in \emptyset} \{ X_{t_i} \in A_i\}) := 1$ by convention. By the definition of $\alpha(n)$, it follows that~$| P_{\nu_r} - Q_{\nu_r} | \leq  \alpha(t_{j+1} - t_j)$ for any partition $\nu_1,...,\nu_R$ and any $r=1,...,R$. Thus, for every partition $\nu_1,...,\nu_R$,
\[
| P_{\nu_1} \cdots P_{\nu_R} - Q_{\nu_1} \cdots Q_{\nu_R}| \leq \sum_{r=1}^R| P_{\nu_r} - Q_{\nu_r} | \leq R \alpha(t_{j+1} - t_j).
\]
All together, this yields
\[
|\cum(I\{X_{t_1} \in A_1\}, \ldots, I\{X_{t_p} \in A_p\})|
\leq \alpha(t_{j+1} - t_j)\sum_{\{\nu_1, \ldots, \nu_R\}}\hspace{-3mm}R! \ .
\]
Noting that $p (t_{j+1} - t_j) \geq \max_{i_1,i_2} |t_{i_1} - t_{i_2}|$ and observing that $\alpha$ is a monotone function, we obtain 
\[|\cum(I\{X_{t_1} \in A_1\}, \ldots, I\{X_{t_p} \in A_p\})|
\leq K_p \alpha(\max |t_i-t_j|).\vspace{-10mm}
\]
\hfill\qed

\begin{lemma} \label{lem:distquant}
	Let $F$ and $G$ denote functions on the real line, with $|G(x)-G(y)| > c |x-y|$ for $x,y \in [a,b]$ where $c$ is some positive constant. For all $p,q \in (a,b),$ with $F(p) = G(q)$ and any $\epsilon > 0,$  $||F(\cdot)-G(\cdot)||_\infty \leq \epsilon$ implies $|p-q| \leq {\epsilon}/{c} .$
\end{lemma}
\noindent\textbf{Proof.}
The claim follows from the fact that
\[ c |p-q| < |G(p) - G(q)| = |G(p) - F(p)| \leq \epsilon.\vspace{-12mm}
\]
\hfill\qed

\subsection{A blocking technique for nonstationary $\beta$-mixing processes}
\label{sec:blocking}


In her paper, \cite{Yu} constructed an independent block (IB) technique to transfer the classical tools from the i.i.d. case to the case of $\beta$-mixing stationary time series. We are using the same technique here to derive results for sums of $\beta$-mixing local stationary time series, which will be used on multiple occasions. For this purpose, let~$X_{t,\T}$ be a triangular array of $\beta$-mixing processes with mixing coefficient $\beta_\T$. For each fixed~$\T$ we will divide the process~$X_{t,\T}$ into~$2\mu_n$ alternating blocks with lengths $p_n$ and $q_n,$ respectively, and a remainder block of length $\T-2\mu_\T(p_n+q_n).$ More precisely, we divide the index set into $(2\mu_n + 1)$ parts
\begin{align*}
&\Gamma_j = \{t: t_{\text{min}} + B_n + (j-1)(p_n+q_n) + 1 \leq t \leq t_{\text{min}} + B_n + (j-1)(p_n+q_n)+ p_n \}\\
& \Delta_j = \{t: t_{\text{min}} + B_n + (j-1)(p_n+q_n) + p_n + 1 \leq t \leq t_{\text{min}} + B_n + j(p_n+q_n) \}\\
& R = \{t: t_{\text{min}} + B_n + \mu_n(p_n+q_n) + 1 \leq t \leq t_{\text{min}} + n - B_n \},
\end{align*}
and introduce the notation
\[ X(\Gamma_j) = \{X_{i,\T}, i \in \Gamma_j \}, \quad X(\Delta_j) = \{X_{i,\T}, i \in \Delta_j \}\quad \text{and} \quad X(R) = \{X_{i,\T}, i \in R \}, \]
where the dependence on $\T$ is omitted for the sake of brevity. We now have a sequence of alternating $X(\Gamma_j)$ and $X(\Delta_j)$ blocks, and a remainder block $X(R):$
\[ X  = X(\Gamma_1),X(\Delta_1),X(\Gamma_2),\dots,X(\Gamma_{\mu_\T}),X(\Delta_{\mu_\T}),X(R).\]
To exploit the concept of coupling, we consider a one-dependent block sequence
\[ Y = Y(\Gamma_1),Y(\Delta_1),Y(\Gamma_2),\dots,Y(\Gamma_{\mu_\T}),Y(\Delta_{\mu_\T}),\]
where $Y(\Gamma_j) = \{\xi_i:i \in \Gamma_j\}$ and $Y(\Delta_j) = \{Y_i:i \in \Delta_j\}$ such that the sequence is independent of~$X$ and each block of $Y$ has the same distribution as the corresponding block in $X$, that~is,
\[Y(\Gamma_i) \overset{d}{=} X(\Gamma_i) \quad \text{and} \quad Y(\Delta_i) \overset{d}{=} X(\Delta_i),\]
where $\overset{d}{=}$ stands for equality in distribution. 

The existence of such a sequence and the measurability issues that arise are addressed in \cite{Yu}. The $\Gamma$- and $\Delta$-block subsequences are denoted by $X_{\Gamma}, Y_{\Gamma}, X_{\Delta}$ and $Y_{\Delta}$, respectively: for instance,
\[X_{\Gamma} := X(\Gamma_1),X(\Gamma_2),\dots,X(\Gamma_{\mu_\T}).\]
We obtain $X_{\Gamma}$ by leaving out every other block in the original sequence, which is $\beta$-mixing, so that the dependence between the blocks in $X_{\Gamma}$ becomes weaker as the size $p_n$ of the $\Gamma$-blocks increases. The following lemma from \cite{Yu} establishes an upper bound for the difference between the $\Gamma$-block sequences from the original process and the independent block sequence.

\begin{lemma} \label{lem:yu_alt}
	For any measurable function $h$ on $\mathbb{R}^{\mu_n q_n}$ with $\Vert h\Vert _\infty \leq M,$ we have, for the blocking structure just described,
	\begin{align*}
	\Big|\E_Q[h(X(\Delta))] - \E_{\tilde Q}[h(Y(\Delta))]\Big| &\leq M(\mu_n - 1)\beta_{p_n} \quad \text{and} \\
	\Big|\E_Q[h(X(\Gamma))] - \E_{\tilde Q}[h(Y(\Gamma))]\Big| &\leq M(\mu_n - 1)\beta_{q_n}.
	\end{align*}
\end{lemma}
\begin{proof}
	We only prove the first claim, which follows as an application of Corollary~2.7 in \cite{Yu} with $Q$ being the probability distribution of the $\Delta_j$ block sequence. However, note that the $\beta$-mixing rate of~$Q$ here is less than~$\beta_{p_n},$ due to the alternating block length. 
\end{proof}


Next, we consider a special case of the same blocking technique with $a_n := q_n = p_n,$ now applied to a sum of $\beta$-mixing random variables, namely $\sum_{t=1}^\T f(X_{t,\T}),$ and link its probabilistic behavior to that of the sum of the independent blocks
$\sum_{j=1}^{\mu_\T} \sum_{i \in \Gamma_j} f(Y_{i,\T}).$ 
To avoid measurability issues the function $f$ is assumed to belong to a permissable class $F_\T$ of functions (for a definition see the appendix in \cite{Yu}). Furthermore, for the sake of simplicity, assume that $\E(f(X_{i,\T})) = 0$ for all~$f \in F_\T.$
The following Lemma is a slight adjustment of Lemma $4.2$ from \cite{Yu}.
\begin{lemma}\label{lem:6.7}
	Let $F_\T$ be a sequence of permissible function classes, which are bounded by a constant $M_n$. If a sequence $(r_\T)_{\T \in \mathbb{N}}$ is such that, for $\T$ large enough, $2r_\T \mu_\T \geq \T M_\T,$ we have
	\begin{align} \label{eqn:lem6.7}
	\p\Big(\sup_{f \in F_\T} \big|\sum_{t=1}^\T f(X_{t,\T})\big| > 4 r_\T \Big) & \leq \p\Big(\sup_{f \in F_\T} \big|\sum_{j=1}^{\mu_\T} \sum_{i \in \Gamma_j} f(Y_{i,\T})\big| > r_\T \Big)\\ \nonumber
	& \quad + \p\Big(\sup_{f \in F_\T} \big|\sum_{j=1}^{\mu_\T} \sum_{i \in \Delta_j} f(Y_{i,\T})\big| > r_\T\Big) + 2 \mu_\T \beta_{a_\T}.
	\end{align}
\end{lemma}
\noindent\textbf{Proof.} 
The probability in the left-hand side of $(\ref{eqn:lem6.7})$ splits into three parts: namely,
\begin{align*}
\p\Big(\sup_{f \in F_\T} \big|\sum_{t=1}^\T f(X_{t,\T})\big| > 4 r_\T \Big) & \leq \p\Big(\sup_{f \in F_\T} \big| \sum_{j=1}^{\mu_\T} \sum_{i \in \Gamma_j} f(X_{i,\T})\big| > r_\T\Big) \\ 
& \quad +\p\Big(\sup_{f \in F_\T} \big| \sum_{j=1}^{\mu_\T} \sum_{i \in \Delta_j} f(X_{i,\T})\big| >  r_\T\Big)  \\
& \quad +\p\Big(\sup_{f \in F_\T} \big| \sum_{i \in R} f(X_{i,\T})\big| > 2 r_\T\Big). 
\end{align*}
The sum appearing in the third part, which deals with the ramainder block, is bounded by~$M_n(2a_n) \leq M_n n/\mu_n.$ As $2 r_n \mu_n \geq n M_n,$ that probability is zero.

Turning to the first part, Lemma $\ref{lem:yu_alt}$ with $h$ the indicator function of the event
\begin{align*}
\Big\{ \sup_{f \in F_\T} \big| \sum_{j=1}^{\mu_\T} \sum_{i \in \Gamma_j} f(X_{i,\T})\big| > r_\T \Big\}
\end{align*}
yields
\[\p\Big(\sup_{f \in F_\T} \big|\sum_{j=1}^{\mu_\T} \sum_{i \in \Gamma_j} f(X_{i,\T})\big| >  r_\T\Big) \leq \p\Big(\sup_{f \in F_\T} \big| \sum_{j=1}^{\mu_\T} \sum_{i \in \Gamma_j} f(Y_{i,\T})\big| > r_\T\Big) + \mu_\T\beta_{a_\T},\]
the second term can be treated by the same arguments. The claim follows. \qed

The upper bound in Lemma \ref{lem:6.7} is based on i.i.d. blocks, which allows us to use classical techniques. In particular, we will  apply the Benett inequality to further bound the sum of $\beta$-mixing random variables. For this purpose assume that the number of functions $m_f(\T)$ contained in $F_\T$ is finite, so that
\[ \p\Big(\sup_{f \in F_\T} \big|\sum_{j=1}^{\mu_\T} \sum_{i \in \Gamma_j} f(Y_{i,\T}) \big| \Big) \leq m_f(\T) \sup_{f \in F_\T} \p\Big(\big|\sum_{j=1}^{\mu_\T} \sum_{i \in \Gamma_j} f(Y_{i,\T})\big| > r_\T \Big).  \]
Furthermore, let us assume that the variance $\Var(\sum_{j=1}^{\mu_\T} \sum_{i \in \Gamma_j} f(Y_{i,\T}))$ of the blocks is bounded by some finite $V_\T$, so that the Benett inequality yields
\begin{align}
\p\Big(\big|\sum_{j=1}^{\mu_\T} \sum_{i \in \Gamma_j} f(Y_{i,\T})\big| > r_\T \Big) \leq \exp\Big(-\frac{\mu_\T V_\T}{a_\T^2 M_\T^2} h\Big(\frac{r_\T a_\T M_\T}{2 \mu_\T V_\T}\Big) \Big),
\end{align}
where $h(x) = (1+x)log(1+x).$ Calculations similar to those in the proof of Lemma 6.7 in \cite{dhkv2014} we can bound the probability by
\[ \exp\Big(-\frac{\log 2}{2}\Big( \frac{r_\T^2}{4 \mu_\T V_\T} \wedge \frac{r_\T}{2 a_\T M_\T} \Big)\Big). \] We just have proven the following Lemma
\begin{lemma}  \label{rem_block}
	Let $X_{t,\T}$ be a triangular array of $\beta$-mixing random variables and $F_\T$ a sequence of finite function classes with cardinality $\#F_\T$ that fulfills
	\[ (i) \#F_\T \leq m_f(\T), \quad (ii) \sup_{f \in F_\T} |f(X_{t,\T})| \leq M_\T \quad \text{and} \quad (iii)~\E(f(X)) = 0 \]
	Assume a blocking structure with block length $a_n := p_n = q_n$ which divides the index set into~$2\mu_n +1$ parts, where $\T/2 - a_\T \leq \mu_\T a_\T \leq \T/2, \ a_\T \rightarrow \infty$ and $\mu_\T \rightarrow \infty$, satisfying
	\begin{enumerate}[(a)]
		\item $ \mu_\T \beta_{a_\T} \xrightarrow{\T \rightarrow \infty} 0,$ 
		\item $2r_\T \mu_\T \geq \T M_\T$ and 
		\item $\Var(\sum_{i \in \Gamma_j} f(X_{i,\T})) \vee \Var(\sum_{i \in \Delta_j} f(X_{i,\T})) \leq V_\T$ for all $1 \leq j \leq \mu_\T.$
	\end{enumerate}
	If these conditions are met ,we obtain
	\[\p\Big(\sup_{f \in F_\T} \big|\sum_{t=1}^\T f(X_{t,\T})\big| > 4 r_\T \Big) \leq  2 m_f(\T) \exp\Big(-\frac{\log 2}{2}\Big( \frac{r_\T^2}{4 \mu_\T V_\T} \wedge \frac{r_\T}{2 a_\T M_\T} \Big)\Big) + o(1).\]
\end{lemma}

\subsection{Auxiliary technical results}

\begin{lemma} \label{lem:empcdf}
	Assume that $M_T \to \infty,\ T/M_T \to 0 \text{ and } t_0/T \to \unu \in (0,1)$. Under Assumptions (A1)-(A4), for any bounded $S \subset \R,$
	\[
	\Big(\sqrt{2M_T} \Big(\frac{1}{2M_T}\sum_{|t-t_0| \leq M_T} (I\{X_{t,T}\leq x\} - F_{t,T}(x))  \Big) \Big)_{x \in \R} \weak \mathbb{B}\quad \mbox{in}\quad \ell^\infty(S)
	\]
	where $\mathbb{B}$ denotes a centered Gaussian process with covariances
	\[
	\E[\mathbb{B}(s)\mathbb{B}(t)] = \sum_{k\in \Z} \big(G^\unu_{k}(x,y) - G^\unu(x)G^\unu(y) \big).
	\]
\end{lemma}

\noindent\textbf{Proof.} In order to prove  weak convergence, we need to establish asymptotic equicontinuity and convergence of finite-dimensional distributions (see Theorem 2.1 in \cite{Kosorok2008}). Convergence of finite-dimensional distributions follows as an application of Lemma \ref{lem:yu_alt} the arguments are quite standard and omitted for the sake of brevity. To prove asymptotic equicontinuity, we apply Lemma 7.1 from \cite{kvdh2014}. More precisely, consider the process $H_n(x) :=  \frac{1}{\sqrt{2M_T}}\sum_{|t-t_0| \leq M_T} (I\{X_{t,T}\leq x\} - F_{t,T}(x)),$ where $n$ denotes the cardinality of the set $\{t \in \{1,\dots,T\}:|t-t_0| \leq M_T\}$. Then,
\[
H_n(x) - H_n(y) = \sum_{|t-t_0| \leq M_T} W_{t,T}(x,y)
\]
where
\[
W_{t,T}(x,y) := \frac{1}{\sqrt{2M_T}}\Big(I\{X_{t,T}\leq x\} - F_{t,T}(x) - (I\{X_{t,T}\leq y\} - F_{t,T}(y))) \Big).
\]
Since $\E W_{t,T}(x,y) = 0$ for all $x,y$, by the definition of cumulants, we have
\[
\E|H_n(x) - H_n(y)|^4 = 3 \Big(\cum_2\Big(\sum_{|t-t_0| \leq M_T} W_{t,T}(x,y)\Big)\Big)^2 + \cum_4\Big(\sum_{|t-t_0| \leq M_T} W_{t,T}(x,y)\Big)
\]
where $\cum_k(y) := \cum(y,...,y)$. Assumption (A3)(iii) implies that
\[
\cum_4\Big(\sum_{|t-t_0| \leq M_T} W_{t,T}(x,y)\Big) = O(1/M_T)
\]
while, under Assumption (A3)(i), there exist constants $C$ and $\tilde C$ such that
\[
\Big|\cum_2\Big(\sum_{|t-t_0| \leq M_T} W_{t,T}(x,y)\Big)\Big| \leq |x-y| + C\sum_{s \geq 1} \min(|x-y|,s|^\gamma) \leq \tilde C|x-y|^{1-\gamma^{-1}}
\]
where the last equality follows by \eqref{eq:sumk}. Thus, there exists a constant $C>0$ such that, for~$|x-y| \geq 1/M_T^{1/2}$, we have
\[
\E|H_n(x) - H_n(y)|^4 \leq C |x-y|^{2-2\gamma^{-1}}.
\]
Now, fix $\delta > 0$ and apply Lemma 7.1 from \cite{dhkv2014} with 
$$\Psi(x) = x^4,\quad  d(x,y) := |x-y|^{(\gamma-1)/(2 \gamma)}, \quad \bar \eta := (2/n)^{(\gamma-1)/(2 \gamma)},\quad \mathbb{G}_x := H_n(x),\quad\!\text{and}\quad\! T := S.$$ In particular, the packing number of the bounded set $S$ with respect to the metric $d$ can be bounded by $D(\eps,d) \leq C \eps^{-2\gamma/(\gamma-1)}$ for some constant $C$ independent of $\eps$. This yields
\begin{equation} \label{eq:help111}
\sup_{x,y \in S, d(x,y) \leq \delta} |H_n(x)-H_n(y)| \leq S_1 + 2\sup_{x,y \in S, d(x,y) \leq \bar \eta} |H_n(x)-H_n(y)| 
\end{equation}
where the quantity $S_1$ satisfies
\begin{equation} \label{eq:help112}
\|\E S_1^4\|^{1/4} \leq K \Big[\int_{\bar \eta/2}^\eta \eps^{-\gamma/2(\gamma-1)} d\eps + (\delta + 2\bar\eta )\eta^{-\gamma/(\gamma-1)}\Big].
\end{equation}
Note that $\gamma > 2$ implies $\gamma/2(\gamma-1) < 1,$ so that $\eps^{-\gamma/2(\gamma-1)}$ is integrable on $[0,1]$. In particular, setting $\eta := \delta^{(\gamma-1)/(2\gamma)}$ implies $\delta\eta^{-\gamma/(\gamma-1)} = \delta^{1/2},$ hence
\[
\lim_{\delta \downarrow 0} \limsup_{n \to \infty} \|\E S_1^4\|^{1/4} = 0. 
\] 
Finally, note that similar arguments as in the proof of \eqref{eq:bincrfk} entail
\begin{equation} \label{eq:help113}
\sup_{x,y \in S, d(x,y) \leq \bar \eta} |H_n(x)-H_n(y)| = o_P(1).
\end{equation}
Jointly, \eqref{eq:help111}-\eqref{eq:help113} imply that, for any $\alpha > 0,$
\[
\lim_{\delta \downarrow 0} \limsup_{n \to \infty} P\Big( \sup_{x,y \in S, d(x,y) \leq \delta} |H_n(x)-H_n(y)| \geq \alpha \Big) = 0.
\] 
Since the metric $d$ makes the index set $S$ totally bounded, condition (ii) in Theorem~2.1 in \cite{Kosorok2008} follows. This, together with the weak convergence of the finite-dimensional distributions, completes the proof. \hfill $\qed$

\begin{lemma}	\label{lem:OrderDFT}
	Let $\ell_n \in \Z$ be a sequence such that $\omega_{\ell_n} := 2\pi \ell_n / n \rightarrow \omega \not\equiv 0 \mod 2\pi$. Let $K$ be a function satisfying assumption (K) and define~$K_n(k) := K(k/B_n)$, for~\mbox{$k \in \Z$}, where $B_n = o(n)$. Denote by $\tilde{\mathcal{F}}_n(\eps)$ the set of Fourier frequencies which are contained in $(\eps,\pi-\eps)$. Assume that condition (A4)(iv) holds. Then
	\[
	\sup_{\omega \in \tilde{\mathcal{F}}_n(\eps)}\sup_{t \in \Ntheta} \sup_{\tau} \Big| \sum_{|k| \leq B_n} K_n(k) {\rm e}^{-{\rm i} \omega k} \big(F_{t+k;T}(q^\unu(\tau)) - \tau \big) \Big| = O\Big(\frac{n}{T B_n^{d-1}} + B_n/T\Big)
	\]
	and
	\[
	\sup_{\omega \in \tilde{\mathcal{F}}_n(\eps)} \Big| \sum_{|k| \leq B_n} K_n(k) {\rm e}^{-{\rm i} \omega k}\Big| = O\Big(\frac{1}{B_n^{d-1}}\Big).
	\]
\end{lemma}

\noindent\textbf{Proof.}
We only establish the first statement since the second one can be proved by similar arguments. Let $h_{t,T}(u) := K\big( u \frac{n}{B_n} \big) [G^{\frac{t}{T} + u \frac{n}{T}}(q^\unu(\tau)) - \tau]$, $u \in [-1/2, 1/2]$ for $T$ large enough that $|u n / B_n| \leq 1$ and $\frac{k}{T} + u \frac{n}{T} \in [0,1]$. Note that, under the assumptions made, this function has support $[-B_n/n,B_n/n]$ and is $d$ times continuously differentiable as a function on $(-1/2,1/2)$. Due to local stationarity the following approximation holds:
\begin{equation*}
\begin{split}
&\sum_{|k| \leq B_n} K\Big( \frac{k}{B_n} \Big) {\rm e}^{-{\rm i} \omega k} [F_{t+k;T}(q^\unu(\tau)) - \tau] \\ 
& \qquad = \sum_{k=-n/2+1}^{n/2} K\Big( \frac{k}{n} \frac{n}{B_n} \Big)  [G^{\frac{t}{T} + \frac{k}{n} \frac{n}{T}}(q^\unu(\tau)) - \tau]{\rm e}^{-{\rm i} \omega k} + O(B_n/T) \\
& \qquad = \sum_{k=-n/2+1}^{n/2} h_{t,T}(k/n){\rm e}^{-{\rm i} \omega k} + O(B_n/T).
\end{split}
\end{equation*}
By Leibniz's rule, we have
\begin{equation*}
\begin{split}
h^{(d)}_{t,T}(u) & = \sum_{j=0}^{d-1} \binom{d}{j} \Big(\frac{n}{B_n}\Big)^j \Big(\frac{n}{T}\Big)^{d-j}  K^{(j)} \Big( u \frac{n}{B_n} \Big) \frac{\partial^{d-j}}{\partial v^{d-j}} G^{\frac{t}{T} + v}(q^\unu(\tau))\Big|_{v=u \frac{n}{T}},\\
& \quad + \Big(\frac{n}{B_n}\Big)^d  K^{(d)}\Big( u \frac{n}{B_n} \Big) \Big(G^{\frac{t}{T} + u \frac{n}{T}}(q^\unu(\tau)) - \tau\Big)
\end{split}
\end{equation*}
so that, under the assumptions made, for some constant $C_d$ depending only on $K$, $d$, and the mapping $u \mapsto G^{u}(q_{\tau}),$
\[
\sup_{t,T,u} |h^{(d)}_{t,T}(u)| \leq C_d (n / B_n)^d \frac{n}{T}.
\]

Note that, under the assumptions of the lemma, the function $u \mapsto h_{t,T}(u)$ is twice continuously differentiable on $(-1/2,1/2)$. Thus, it admits the Fourier series representation 
\[
h_{t,T}(u) = \sum_{j \in \Z} c_{j,t,T} {\rm e}^{{\rm i} 2\pi j u},
\quad \text{where }
c_{j,t,T} := \int_{-1/2}^{1/2} h_{t,T}(u) {\rm e}^{-{\rm i} 2\pi j u} {\rm d} u.
\]
Now consider a Fourier frequency $\omega_\ell = 2 \pi \ell/n \in \tilde{\mathcal{F}}_n(\eps)$. By the usual argument (see  \cite{briggshenson1995}, page 182), we have the discrete Poisson summation formula 
\begin{align*}
\sum_{k=-n/2+1}^{n/2} h_{t,T}(k/n){\rm e}^{-{\rm i} \omega_{\ell} k} &= \sum_{j \in \Z} c_{j,t,T} \sum_{k=-n/2+1}^{n/2} {\rm e}^{{\rm i} 2\pi k (j - \ell)/n}\\
&= n \Big(c_{\ell, t, T} + \sum_{k=1}^{\infty} (c_{\ell + k n, t, T} + c_{\ell - k n, t, T})\Big).
\end{align*}
For the leading term, note that $h^{(r)}_{t,T}(u) = 0$ for $|u| > B_n/n$, so that, integrating by parts yields 
\begin{equation}
\label{eqn:OrderCell}
c_{\ell,t,T} = (-1)^{d+1} \frac{1}{({\rm i} 2\pi \ell)^d} \int_{-B_n/n}^{B_n/n} h^{(d)}_{t,T}(u) {\rm e}^{-{\rm i} 2\pi \ell u} {\rm d} u.
\end{equation}
It follows that $|c_{\ell, t, T}| \leq 2 C_p  (2\pi \ell)^{-d} \frac{n}{T} (n/B_n)^{d-1} \lesssim \frac{1}{T B_n^{d-1}}$, as $\ell \asymp n$. Furthermore, by Assumption (A4)(iv) (recall that $d \geq 2$ and $\ell/n \to c \in (0,1) \text{ mod } 1$), 
\begin{equation*}
\begin{split}
\Big| \sum_{k=1}^{\infty} (c_{\ell + k n, t, T} + c_{\ell - k n, t, T}) \Big|
& \lesssim \frac{n}{T} \Big(\frac{n}{B_n}\Big)^{d-1} \sum_{k=1}^{\infty} \Big(\frac{1}{(\ell + k n)^d} + \frac{1}{(\ell - k n)^d} \Big) \\
& = \frac{1}{T B_n^{d-1}} \sum_{k=1}^{\infty} \Big(\frac{1}{(\ell/n + k)^d} + \frac{1}{(\ell/n - k)^d} \Big) \lesssim \frac{1}{T B_n^{d-1}}.
\end{split}
\end{equation*}
Note that all the bounds above hold uniformly in $\ell \asymp n$. This completes the proof of the lemma.\hfill\qed

\section{Details for the proof of Theorem \ref{thm:asymain}} \label{sec:proofsmainsteps}

\subsection{Proof of \eqref{eq:main1}} 
Define 
\[
\hat F_{t_0,t_0+k;T} (x,y) := \frac{1}{n} \sum_{t \in T(k)} \I{X_{t,T} \leq x, X_{t+k,T} \leq y},
\]
and let
\begin{eqnarray*}
	r_{n,1}(k) &:=& \hat F_{t_0,t_0+k;T}(\hat q_{t_0,T}(\tau_1), \hat q_{t_0,T}(\tau_2)) - \hat F_{t_0,t_0+k;T}(q^\unu(\tau_1),q^\unu(\tau_2)) 
	\\
	&& \quad\quad\quad- \frac{1}{n} \sum_{t \in T(k)} \Big( F_{t,t+k;T}(\hat q_{t_0,T}(\tau_1),\hat q_{t_0,T}(\tau_2)) - F_{t,t+k;T}(q^\unu(\tau_1),q^\unu(\tau_2)) \Big),
	\\
	r_{n,2}(k) &:=& \frac{1}{n} \sum_{t \in T(k)} \Big[F_{t,t+k;T}(\hat q_{t_0,T}(\tau_1),\hat q_{t_0,T}(\tau_2)) - F_{t;T}(\hat q_{t_0,T}(\tau_1))F_{t+k;T}(\hat q_{t_0,T}(\tau_2))
	\\
	&& \quad\quad\quad\quad\quad\quad\quad
	-\Big(F_{t,t+k;T}(q^\unu(\tau_1),q^\unu(\tau_2)) - F_{t;T}(q^\unu(\tau_1))F_{t+k;T}(q^\unu(\tau_2)) \Big) \Big],
	\\
	r_{n,3}(k) &:=& \frac{1}{n} \sum_{t \in T(k)} \Big( F_{t;T}(\hat q_{t_0,T}(\tau_1))F_{t+k;T}(\hat q_{t_0,T}(\tau_2)) - F_{t;T}(q^\unu(\tau_1))F_{t+k;T}(q^\unu(\tau_2)) \Big),
	\\
	r_{n,4}(k) &:=& \frac{\tau_1}{n} \sum_{t \in T(k)} \Big(\I{X_{t
			+k,T} \leq \hat q_{t_0,T}(\tau_2)} - \tau_2 \Big) + \frac{\tau_2}{n} \sum_{t \in T(k)} \Big(\I{X_{t,T} \leq \hat q_{t_0,T}(\tau_1)} - \tau_1 \Big)
	\\
	&& - \frac{\tau_1}{n} \sum_{t \in T(k)} \Big(\I{X_{t+k,T} \leq q^\unu(\tau_2)} - \tau_2 \Big) - \frac{\tau_2}{n} \sum_{t \in T(k)} \Big(\I{X_{t,T} \leq q^\unu(\tau_1)} - \tau_1 \Big).
\end{eqnarray*}
Observe that, due to the assumptions on $K_n,$
\begin{align*}
&2\pi \hat{\mathfrak{f}}_{t_0,T}(\omega,\tau_1,\tau_2) - \sum_{|k|\leq n-1} \hspace{-2mm}K_n(k) e^{-i \omega k}\frac{1}{n} \sum_{t \in T(k)} \Big(\I{X_{t,T} \leq  q^\unu(\tau_1)} - \tau_1 \Big)\Big( \I{X_{t+k,T} \leq q^\unu(\tau_2)} - \tau_2 \Big)
\\
&\qquad = \sum_{|k| \leq B_n} K_n(k) e^{-i \omega k} \Big(r_{n,1}(k) + r_{n,2}(k) + r_{n,3}(k) + r_{n,4}(k)\Big) \\
&\qquad =: R_{n,1} + R_{n,2} + R_{n,3} + R_{n,4},\text{ say.}
\end{align*}
To prove \eqref{eq:main1} it is sufficient to establish the following statements:
\begin{equation} \label{eq:boundqhat}
\max(|q^\unu(\tau_1) - \hat q_{t_0,T}(\tau_1)|,|q^\unu(\tau_2) - \hat q_{t_0,T}(\tau_2)|) = O_P(T^{-2/5}),
\end{equation}
\begin{align} \label{eq:bincrfk}
\sup_k \sup_{x \in X} \sup_{\|y\|\leq \eps_n} & \Big|\hat F_{t_0,t_0+k;T}(x)
- \hat F_{t_0,t_0+k;T}(x+y) - \frac{1}{n} \sum_{t \in T(k)}\hspace{-1mm}[F_{t,t+k;T}(x) - F_{t,t+k;T}(x+y)]\Big| \nonumber \\
& = O_\p(\rho_n(\eps_n)),
\end{align}
for any $\eps_n = o(1)$ and any bounded set $X \subset \R^2$ with $\|v\|$ denoting the maximum norm of the vector $v$, and
\begin{align} \label{eq:uincrfk}
&\sup_{x \in Z} \sup_{|y|\leq \eps_n} \hspace{-2mm} \Big|\frac{1}{n}\sum_{|t-t_0|\leq m_T} \Big( \I{X_{t,T} \leq x} - \I{X_{t,T} \leq x+y} - F_{t;T}(x) + F_{t;T}(x+y) \Big)\Big| = O_\p(\rho_n(\eps_n))
\end{align}
for any $\eps_n = o(1)$ and any bounded set $Z \subset \R$
where $\rho_n$ is defined in Assumption (A2). We first analyze the asymptotic behavior of the four remainder terms $R_{n,1},R_{n,2},R_{n,3},R_{n,4}$, then turn to the proofs for \eqref{eq:boundqhat} - \eqref{eq:uincrfk}.

\bigskip
\noindent
\textbf{\textit{Discussion of remainder term $R_{n,1}$.}}\\ From (\ref{eq:boundqhat}) and (\ref{eq:bincrfk}), we obtain $\sup_k |r_{n,1}(k)| = O_P(\rho_n(T^{-2/5}))$ hence under (A3)
\begin{equation} \label{Rn1}
|R_{n,1}| = o_P(B_n^{1/2}n^{-1/2}).
\end{equation}

\noindent
\textbf{\textit{Discussion of remainder term $R_{n,2}$.}} \\
Under (A3)(i) and (A4)(i), $|r_{n,2}(k)| = O(\min(|k|^{-\gamma}, T^{-2/5})),$ and thus
\begin{equation} \label{Rn2}
|R_{n,2}| \leq \sum_{|k| \leq B_n} O(\min(|k|^{-\gamma}, T^{-2/5})) = O(T^{-(1-\gamma^{-1})2/5}) = o(B_n^{1/2}n^{-1/2}).
\end{equation}
To see the this, note that, for $\eps \to 0,$
\begin{equation} \label{eq:sumk}
\sum_{k \geq 1} \min(k^{-\gamma},\eps) \leq \sum_{1\leq k \leq \eps^{-1/\gamma}} \eps + \sum_{k \geq \eps^{-1/\gamma}} k^{-\gamma} = O(\eps^{1-1/\gamma}) + O(\eps^{(-1/\gamma)(1-\gamma)}) = O(\eps^{1-\gamma^{-1}}).
\end{equation}
\noindent
\textbf{\textit{Discussion of remainder term $R_{n,3}.$}}\\
Start by observing that
\begin{eqnarray*} 
	\frac{1}{n}\sum_{t \in T(k)} F_{t;T}(x)F_{t+k;T}(y)  &=& \frac{1}{n} \sum_{|t-t_0|\leq m_T} \hspace{-4pt} F_{t;T}(x)F_{t+k;T}(y) + O(k/n)
	\\
	&=& \frac{1}{n} \sum_{|t-t_0|\leq m_T}\hspace{-4pt} G^{t/T}(x)G^{(t+k)/T}(y) + O(k/n) + O(1/T)
	\\
	&=& \frac{1}{n} \sum_{|t-t_0|\leq m_T} \hspace{-4pt}G^{t/T}(x)G^{t/T}(y) + O(k/T) + O(k/n) + O(1/T)
	\\
	&=& \frac{T}{2m_T}\int_{-m_T/T}^{m_T/T}\hspace{-4pt} G^{\unu + u}(x)G^{\unu + u}(y) du + O(k/n), 
\end{eqnarray*}
where we have used a first-order Taylor expansion of the function $u \mapsto G^u(x)$. This yields
\begin{align*}
r_{n,3}(k) = \frac{T}{2m_T}\int_{-m_T/T}^{m_T/T} & G^{\unu + u}(\hat q_{t_0,T}(\tau_1))G^{\unu + u}(\hat q_{t_0,T}(\tau_2))\\
& \quad - G^{\unu + u}(q^\unu(\tau_1))G^{\unu + u}(q^\unu(\tau_2)) du + O(B_n/n), 
\end{align*}
uniformly in $|k|\leq B_n$. Observe that, by Lemma \ref{lem:OrderDFT} under condition~(K), we have 
\[ 
\sup_{\omega \in \tilde{\mathcal{F}}_n(\eps)} \Big|\sum_{|k|\leq n-1} K_n(k) e^{-i k\omega}\Big| = O(1).
\]
This implies
\begin{align*}
R_{n3} =& \Big(\sum_{|k|\leq B_n} K_n(k) e^{-i \omega k} \Big) \frac{T}{2m_T}\int_{-m_T/T}^{m_T/T} G^{\unu + u}(\hat q_{t_0,T}(\tau_1))G^{\unu + u}(\hat q_{t_0,T}(\tau_2)) 
\\
& \hspace{5cm} - G^{\unu + u}(q^\unu(\tau_1))G^{\unu + u}(q^\unu(\tau_2)) du + O(B_n^2/n)
\\
=& \ O\Big(\max\big(|q^\unu(\tau_1) - \hat q_{t_0,T}(\tau_1)|,|q^\unu(\tau_2) - \hat q_{t_0,T}(\tau_2)|\big)\Big) + O(B_n^2/n)
\end{align*}
uniformly in $\omega \in \tilde{\mathcal{F}}_n(\eps),$ almost surely. Recalling
(\ref{eq:boundqhat}) and Assumption (A2) we thus obtain 
\begin{equation} \label{Rn3}
|R_{n,3}| = o_P(B_n^{1/2}n^{-1/2}).
\end{equation}
\noindent
\textbf{\textit{Discussion of remainder term $R_{n,4}.$}}\\
Observe that, uniformly in $|k| \leq B_n$ and $ y \in\R$ we have
\begin{align*}
&\frac{1}{n}\sum_{t \in T(k)} \I{X_{t,T} \leq y}  =  \frac{1}{n}\sum_{|t-t_0|\leq m_T} \I{X_{t,T} \leq y} 
+ O_P(B_n/n),&
\\
&\frac{1}{n}\sum_{t \in T(k)} \I{X_{t+k,T} \leq y}  =  \frac{1}{n}\sum_{|t-t_0|\leq m_T} \I{X_{t,T} \leq y}
+ O_P(B_n/n).&
\end{align*}
Thus, uniformly in $|k|\leq B_n$,
 $
r_{n,4}(k) = D_n + O_P(B_n/n)
$, 
where
\[
D_n := \sum_{|t-t_0|\leq m_T} \frac{\tau_1}{n}\Big(\I{X_{t,T} \leq \hat q_{t_0,T}(\tau_2)} - \I{X_{t,T} \leq q^\unu(\tau_2)} \Big) + \frac{\tau_2}{n}\Big(\I{X_{t,T} \leq \hat q_{t_0,T}(\tau_1)} - \I{X_{t,T} \leq q^\unu(\tau_1)} \Big)
\]
does not depend on $k$. In particular, by Lemma \ref{lem:OrderDFT} this implies
\[
|R_{n,4}| \leq O_P(B_n^2/n) + |D_n| \sup_{\omega \in \tilde{\mathcal{F}}_n(\eps)} \Big|\sum_{|k|\leq n-1} K_n(k) e^{-i k\omega}\Big| = O_P(B_n^2/n) + |D_n|O(B_n^{-1}).
\]
To conclude with $R_{n,4}$, note that combining (\ref{eq:boundqhat}), (\ref{eq:uincrfk}), and Assumption (A4)(i) we obtain $|D_n| = O_P(\rho_n(T^{-2/5})) + O_P(T^{-2/5})$. Together with (A2), this entails $R_{n,4} = o_P(B_n^{1/2}/n^{1/2})$ which, combined with \eqref{Rn1}-\eqref{Rn3}, yields \eqref{eq:main1}. It remains to establish \eqref{eq:boundqhat} - \eqref{eq:uincrfk}.

\bigskip
\noindent
\textbf{\textit{Proof of (\ref{eq:boundqhat})}}
Letting $M_T = T^{4/5}$ in Lemma \ref{lem:empcdf}, we obtain the weak convergence of~$\sqrt{2T^{4/5}}(\tilde F_{t_0;T}(x) - \bar F(x)),$ where
\[
\bar F(x) := \frac{1}{2T^{4/5}}\sum_{|t-t_0|\leq T^{4/5}} F_{t;T}(x) = \frac{1}{2T^{4/5}}\sum_{|t-t_0|\leq T^{4/5}} G^{t/T}(x) + O(1/T).
\]
to a centered Gaussian process with almost surely continuous sample paths. 
Next, observe that, uniformly in $x,$
\[
\bar F(x) = \frac{T}{2T^{4/5}}\int_{-T^{4/5}/T}^{T^{4/5}/T} G^{\unu+u}(x) du + O(1/T) = G^\unu(x) + O((T^{4/5}/T)^2) + O(1/T)
\]
where we have used a second-order Taylor expansion of the function $u \mapsto G^{\unu+u}(x)$. The claim (statement \eqref{eq:boundqhat}) follows by compact differentiability of the quantile mapping, see Lemma~12.8 in \cite{Kosorok2008}. \\ 

\noindent
\textbf{\textit{Proof of (\ref{eq:bincrfk}) and (\ref{eq:uincrfk})}} Statement (\ref{eq:uincrfk}) can be established by similar arguments as~(\ref{eq:bincrfk}), and its proof is omitted for the sake of brevity. 
Let~$x = (x_1,x_2)$, $y = (y_1,y_2),$ and define\vspace{-3mm}
\begin{multline*}
n W_{t,k} (x,y) := \I{X_{t,T} \leq x_1, X_{t+k,T} \leq x_2} - \I{X_{t,T} \leq x_1+y_1, X_{t+k,T} \leq x_2+y_2}
\\
- \p(X_{t,T} \leq x_1, X_{t+k,T} \leq x_2) + \p(X_{t,T} \leq x_1+y_1, X_{t+k,T} \leq x_2+y_2).  \vspace{-2mm}
\end{multline*}
With this notation, we have\vspace{-3mm}
\begin{multline*}
\sum_{t\in T(k)} W_{t,k}(x,y) = \hat F_{t_0,t_0+k;T}(x_1 + y_1, x_2 + y_2) - \hat F_{t_0,t_0+k;T}(x_1,x_2) \\
- \frac{1}{n}\sum_{t \in T(k)} F_{t,t+k;T}(x_1+y_1,x_2+y_2) - F_{t,t+k;T}(x_1,x_2).
\end{multline*}
Cover the bounded set $\{(x,y): x \in X, \|y\| \leq \eps_n \}$ with $O(n^4)$ spheres of radius $1/2n$ and centers $(v,w)$ such that $\|w\| \leq \eps_n$, and denote the set of resulting centers by $Z$. Observe that there exists a constant $C$ independent of $k$ such that
\begin{align*}
\sup_{\|(v,w)-(x,y)\| \leq 1/n} & |W_{t,k}(v,w)-W_{t,k}(x,y)| \\
& \leq n^{-1}(\I{|X_{t,T}-v_1| \leq 1/n} +\I{ |X_{t+k,T}-v_2| \leq 1/n}  \\
& \hspace{2cm} + \I{|X_{t,T}-v_1-w_1| \leq 1/n} + \I{|X_{t+k,T}-v_2-w_2| \leq 1/n} +C)\\
& :=  V_{t,k}(v,w). 
\end{align*}
Therefore, 
\begin{align*} \label{eqn:Lem1first}
\sup_{x \in X} \sup_{\|y\| < \eps_n} \big|\sum_{t \in T(k)} W_{t,k}(x,y)\big| \leq \max_{(v,w) \in Z} \big| \sum_{t \in T(k)} W_{t,k}(v,w)\big| + \max_{(v,w) \in Z} \big|\sum_{t \in T(k)} V_{t,k}(v,w)\big|.
\end{align*}
We now use blocking to show that both terms in the right-hand side are of order $\Op(\rho_n(\eps_n)),$ uniformly in $k$. Since the random variables $X_{t,T}$ are $\beta$-mixing, so are the random variables~$W_{t,k}$ and $V_{t,k},$ and the $\beta$-mixing coefficients $\beta_j^{[W]}$ of $W_{t,k}$ are bounded by $\beta^{[X]}_{0 \vee (j-B_n)}.$ The same holds for the $\beta$-mixing coefficients of $V_{t,k}$. Furthermore, with $\mathring{V}_{t,k} := V_{t,k} - \E(V_{t,k}),$ it follows that
\begin{gather*}
\#\{W_{t,k}(v,w)|(v,w) \in Z\} = \#\{\mathring{V}_{t,k}(v,w)|(v,w) \in Z\} = O(n^4),\\
\max_{(v,w) \in Z} |W_{t,k}(v,w)| \leq n^{-1}, \quad   \quad \max_{(v,w) \in Z} |V_{t,k}(v,w)| = O(n^{-1}),
\end{gather*}
and 
$$
\E(W_{t,k}(v,w)) = \E(\mathring{V}_{t,k}(v,w)) = 0,
$$
so that the classes $\{W_{t,k}(v,w)|(v,w) \in Z\}$ and $\{\mathring{V}_{t,k}(v,w)|(v,w) \in Z\}$ satisfy conditions $(i)-(iii)$ in Lemma \ref{rem_block} with $m_f(n) = O(n^5)$ and $M_n = 1/n$. Set 
\[
a_n = \lceil (n^{\frac{1}{\delta + 1}} \vee k_n)\log(n) \rceil, \quad \mu_n = \lfloor \frac{n}{2 a_n} \rfloor\ \quad  \text{and} \quad r_n = \rho_n(\eps_n), 
\] 
so that conditions (a) and (b) of that lemma are satisfied as well, for $n$ large enough, by the random variables $(W_{t,k})_{t \in T(k)}$ and $(V_{t,k})_{t \in T(k)},$ for any $k$. A Taylor expansion yields 
\[
\sup_{t,k,(v,w)\in Z} |\E W_{t,k}(v,w)W_{s,k}(v,w)| = O(\eps_n)
\]
for any $s,t,$ and
\[
\sup_{t,k,(v,w)\in Z} |\E W_{t,k}(v,w)W_{s,k}(v,w)| = O(\eps_n^2)
\]
for any $s,t$ such that $t,\ t+k,\ s$ and $s+k$ are four distinct indices. Note that, for a given~$k$, there exist $O(a_n)$ pairs $(s,t)$ with $t_1 \leq s,t \leq t_2$ such that at least two of the four indices~$(t,t+k,s,s+k)$ coincide. Thus, for sufficiently large~$n$ and all $t_2-t_1 = a_n,$
\begin{equation} \label{eq:bvarw}
\sup_{k \leq B_n}\sup_{(v,w)\in Z} \Var\Big(\sum_{t = t_1}^{t_2} W_{t,k}(x,y) \Big) \leq c_1\Big(\frac{a_n}{n^2}\big(\eps_n + a_n\eps_n^2\big)\Big).
\end{equation}
Applying Lemma \ref{rem_block} to the triangular array $\{W_{t,k}(v,w)\}$ yields
\begin{align*}
\p\Big(\sup_{k \leq B_n}\sup_{(v,w)\in Z} & \big| \sum_{t\in T(k)} W_{t,k}(v,w)\big| > D\rho_n(\eps_n)\Big)\\
&\leq  O(n^5)\exp\Big(-\frac{\log(2)}{2}\Big(\frac{D^2\rho_n(\epsilon_n)^2}{4\mu_n V_n}\wedge \frac{D\rho_n(\eps_n)}{2 a_n n^{-1}}\Big)\Big),
\end{align*}
where \mbox{$V_n := c_1(\frac{a_n}{n^2}(\eps_n + a_n\eps_n^2))$} and $D$ beeing an arbitrary constant.
Now, the definition of~$\rho_n(\eps_n)$ implies that $D$ can be chosen in such a way that the right-hand side of the above inequality tends to zero for $n \rightarrow \infty$, i.e., for $D$ sufficiently large,
\begin{align*}
\p\Big(\sup_{k \leq B_n} \sup_{(v,w) \in Z} | \sum_{t\in T(k)} W_{t,k}(v,w)| > D\rho_n(\eps_n)\Big) = o(1).
\end{align*}
The same   analysis as before yields
\[
\sup_{k \leq B_n}\sup_{(v,w) \in \mathcal{Z}} \Var(\sum_{t = t_1}^{t_2}\mathring{V}_{t,k}(v,w)) = O(\frac{a_n}{n^3}+\frac{a^2_n}{n^4}) = O(\frac{a_n}{n^3}),
\]
for all $t_2-t_1 = a_n$; yet another application of Lemma \ref{rem_block} entails, for a suitable constant~$D,$
\[
\p\Big(\sup_{k \leq B_n} \sup_{(v,w) \in Z} | \sum_{t \in T(k)} V_{t,k}(v,w)| > D\rho_n(\epsilon_n)\Big) = o(1).
\]
This completes the proof of \eqref{eq:main1}. \hfill $\qed$

\subsection{Proof of \eqref{eq:main2}}
First, note that 
\begin{multline*}
\sum_{|k|\leq n-1} K_n(k) e^{-i \omega k} \frac{1}{n} \Big[\sum_{t \in T(k)} \Big( \I{X_{t,T} \leq q^\unu(\tau_1)} - \tau_1\Big) \Big( \I{X_{t+k,T} \leq q^\unu(\tau_2)} - \tau_2\Big) \\
- \sum_{|t-t_0| \leq m_T-B_n} \Big( \I{X_{t,T} \leq q^\unu(\tau_1)} - \tau_1\Big) \Big( \I{X_{t+k,T} \leq q^\unu(\tau_2)} - \tau_2\Big) \Big]\\
= O_P(B_n^2 / n) = o_P( \sqrt{B_n/n} ).
\end{multline*}
By simple algebra, we obtain
\begin{align*}
&\frac{1}{n} \sum_{|t-t_0| \leq m_T-B_n} \Big[\Big(\I{X_{t,T} \leq q^\unu(\tau_1)} - F_{t;T}(q^\unu(\tau_1)) \Big) \Big( \I{ X_{t+k,T} \leq q^\unu(\tau_2)}  - F_{t+k;T}(q^\unu(\tau_2)) \Big) \\
& \hspace{2cm} - \Big(\I{X_{t,T} \leq q^\unu(\tau_1)} - \tau_1 \Big) \Big( \I{ X_{t+k,T} \leq q^\unu(\tau_2)} - \tau_2 \Big) \Big] \\
&= \frac{1}{n} \sum_{|t-t_0| \leq m_T-B_n} \Big[\Big(\I{X_{t,T} \leq q^\unu(\tau_1)} - \tau_1 \Big) \Big( \tau_2 - F_{t+k;T}(q^\unu(\tau_2)) \Big) \\
& \hspace{2cm}+ \Big(\I{ X_{t+k,T} \leq q^\unu(\tau_2)} - F_{t+k;T}(q^\unu(\tau_2)) \Big) \Big( \tau_1 - F_{t;T}(q^\unu(\tau_1)) \Big) \Big]
=: a_{1,n} + a_{2,n}.
\end{align*}
Let $A_{i,n} := \sum_{|k|\leq n-1} K_n(k) e^{-i \omega k} a_{i,n}$: the proof consists in showing that $\E |A_{i,n}|^2 = o(B_n/n)$, $i=1,2$. We have 
\begin{align}\label{eq:bound2}
 \E|A_{1,n}|^2  \nonumber  \nonumber
& = \E \Bigg| \sum_{|k|\leq n-1} K_n(k) e^{-i \omega k} \frac{1}{n} \sum_{|t-t_0| \leq m_T-B_n} \Big(\I{X_{t,T} \leq q^\unu(\tau_1)} - \tau_1 \Big) \Big( \tau_2 - F_{t+k;T}(q^\unu(\tau_2)) \Big) \Bigg|^2 \\ \nonumber
& =  \frac{1}{n^2} \sum_{|t_1-t_0| \leq m_T-B_n} \sum_{|t_2-t_0| \leq m_T-B_n} \E \Big[ \Big(\I{X_{t_1,T} \leq q^\unu(\tau_1)} - \tau_1 \Big) \Big(\I{X_{t_2,T} \leq q^\unu(\tau_1)} - \tau_1 \Big) \Big]\\ 
& \hspace{3cm} \times \sum_{|k_1|\leq n-1}  K_n(k_1)  e^{-i \omega k_1} \Big( \tau_2 - F_{t_1+k_1;T}(q^\unu(\tau_2)) \Big) \\ \nonumber
& \hspace{3cm} \times \sum_{|k_2|\leq n-1} K_n(k_2) e^{i \omega k_2} \Big( \tau_2 - F_{t_2+k_2;T}(q^\unu(\tau_2)) \Big);
\end{align}
in view of Lemma~\ref{lem:OrderDFT} and the fact that 
\begin{align*}
& \E [ (\I{X_{t_1,T} \leq q^\unu(\tau_1)} - \tau_1) (\I{X_{t_2,T} \leq q^\unu(\tau_1)} - \tau_1 )]
\\
& \hspace{6mm}= F_{t_1,t_2;T}(q^\unu(\tau_1),q^\unu(\tau_1))  - \tau_1 F_{t_2;T}(q^\unu(\tau_1))  - \tau_1 F_{t_1;T}(q^\unu(\tau_1)) + \tau_1^2
\\
& \hspace{6mm} =  \cum(\I{X_{t_1,T} \leq q^\unu(\tau_1)},\I{X_{t_2,T} \leq q^\unu(\tau_1)}) +F_{t_1;T}(q^\unu(\tau_1))F_{t_2;T}(q^\unu(\tau_1))
\\
& \hspace{26mm}   - \tau_1 F_{t_2;T}(q^\unu(\tau_1))  - \tau_1 F_{t_1;T}(q^\unu(\tau_1)) + \tau_1^2
\\
&\hspace{6mm} = \cum(\I{X_{t_1,T} \leq q^\unu(\tau_1)},\I{X_{t_2,T} \leq q^\unu(\tau_1)}) + O(n^2/T^2),
\end{align*}
the right-hand side of \eqref{eq:bound2} is bounded by
\begin{align*}
&\frac{1}{n^2}\hspace{-2mm} \sum_{\substack{|t_1-t_0| \leq m_T-B_n \\ |t_2-t_0| \leq m_T-B_n}} \hspace{-3mm}\Big[\cum (\I{X_{t_1,T} \leq q^\unu(\tau_1)}, \I{X_{t_2,T} \leq q^\unu(\tau_1)}) + O\Big(\Big[\frac{n}{T}\Big]^2\Big)\Big]  O\Big( \Big[\frac{n}{T B_n^{d-1}} + B_n/T\Big]^2\Big) \\
& = O\Big( n^{-1} + n^2/T^2 \Big) O\Big( \Big[\frac{n}{T B_n^{d-1}} + B_n/T\Big]^2\Big) .
\end{align*}
Turning to $A_{2,n}$, note that
\begin{align*}
& \E|A_{2,n}|^2
= \E \Big| \sum_{|k|\leq n-1} \hspace{-2mm} K_n(k) e^{-i \omega k}\\
& \hspace{2cm} \times \frac{1}{n} \sum_{|t-t_0| \leq m_T-B_n} \hspace{-2mm} \Big(\I{X_{t+k,T} \leq q^\unu(\tau_2)} -  F_{t+k;T}(q^\unu(\tau_2)) \Big) \Big( \tau_1 - F_{t;T}(q^\unu(\tau_1)) \Big) \Big|^2 \\
& =  \frac{1}{n^2} \sum_{\substack{|k_1|\leq n-1\\|k_2|\leq n-1}} K_n(k_1) K_n(k_2) e^{- i \omega (k_1-k_2)} \sum_{\substack{|t_1-t_0| \leq m_T-B_n \\ |t_2-t_0| \leq m_T-B_n}} \Bigg[ \Big( \tau_2 - F_{t_1;T}(q^\unu(\tau_2)) \Big)\\
& \hspace{2cm} \times \Big( \tau_2 - F_{t_2;T}(q^\unu(\tau_2)) \Big)\cum \Big( \I{X_{t_1+k_1,T} \leq q^\unu(\tau_2)}, \I{X_{t_2+k_2,T} \leq q^\unu(\tau_1)} \Big) \Bigg] \\
& \leq \frac{1}{n^2} \sum_{\substack{|k_1|\leq n-1\\|k_2|\leq n-1}} \quad \sum_{\substack{|t_1-t_0| \leq m_T-B_n\\|t_2-t_0| \leq m_T-B_n}}
O\Big(\frac{n^2}{T^2}\Big) \Big| \cum \Big( \I{X_{t_1+k_1,T} \leq q^\unu(\tau_2)}, \I{X_{t_2+k_2,T} \leq q^\unu(\tau_1)} \Big) \Big| \\
& \leq O(1/T^2) \sum_{|t_1-t_0| \leq m_T-B_n} \sum_{|k_1| \leq B_n}  O(B_n) \sum_{m \in\Z} \kappa_{2}(m) 
= O(nB_n^2/T^2) = o(B_n/n),
\end{align*}
where the second inequality follows from the fact that, for each fixed value of $t_1, k_1$ and each~$m \in \Z$ there are at most~$O(B_n)$ values of $k_2,t_2$ such that $t_1 + k_1 - t_2-k_2 = m$ and  Assumption (A3)(iii), which implies that the sum over $m$ is finite. \hfill $\qed$

\subsection{Proof of \eqref{eq:asynorm}}

%
%
%
%
To start with, let us state the following lemma.
\begin{lemma} \label{lem:boundvarblock}
	For any $a_n \to \infty$ such that $a_n/n=o(1), B_n/a_n = o(1)$ we have, for all~$\omega_1,\omega_2$ in~$\{\omega,-\omega\}$, 
	\begin{align*} \label{eq:boundvarblock}
	& \sup_{|s-t_0| \leq m_T} \Big| \E\Big[\sum_{t_1=s}^{s+a_n} W_{t_1,T}(\omega_1) \overline{\sum_{t_2=s}^{s+a_n} W_{t_2,T}(\omega_2)} \Big]\\
	& \hspace{2.5cm} - 2 \pi \frac{a_nB_n}{n^2} \int K^2(u)du \Big( \I{\omega_1=\omega_2} \mathfrak{f}^{\unu}(\omega_1, \tau_1, \tau_1)\mathfrak{f}^{\unu}(-\omega_2, \tau_2, \tau_2)\\
	& \hspace{7.5cm}+ \I{\omega_1= - \omega_2} \mathfrak{f}^{\unu}(\omega_1, \tau_1, \tau_2)\mathfrak{f}^{\unu}(\omega_2, \tau_2, \tau_1) \Big)\Big|\\
	& = o(B_na_n/n^2).
	\end{align*}
\end{lemma}
\noindent \textbf{Proof.}
Observe that 
\begin{align*}
& \E\Big[\sum_{t_1=s}^{s+a_n} W_{t_1,T}(\omega_1)\sum_{t_2=s}^{s+a_n} W_{t_2,T}(\omega_2) \Big] \\
& = \!\! \ \frac{1}{4\pi^2}\!\sum_{|k_1| \leq B_n} \sum_{|k_2| \leq B_n} \! K_n(k_1) K_n(k_2) {\rm e}^{- {\rm i} (k_1 \omega_1 + k_2 \omega_2) }
	\frac{1}{n^2} \sum_{t_1 = s}^{s+a_n} \sum_{t_2 = s}^{s+a_n}
\Cov \big( Y_{t_1,\tau_1} Y_{t_1+k_1,\tau_2}, Y_{t_2, \tau_1} Y_{t_2+k_2, \tau_2} \big) \\
& = \  \frac{1}{4\pi^2}\sum_{|k_1| \leq B_n} \sum_{|k_2| \leq B_n} K_n(k_1) K_n(k_2) {\rm e}^{- {\rm i} (k_1 \omega_1 + k_2 \omega_2)}\\
& \hspace{3cm} \times \frac{1}{n^2} \sum_{t_1 = s}^{s+a_n} \sum_{t_2 = s}^{s+a_n} 
\Big[ \cum \big( Y_{t_1,\tau_1}, Y_{t_1+k_1,\tau_2}, Y_{t_2,\tau_1}, Y_{t_2+k_2,\tau_2} \big) \\
& \hspace{6cm} + \cum \big( Y_{t_1,\tau_1}, Y_{t_2,\tau_1} \big) \cum \big( Y_{t_1+k_1,\tau_2}, Y_{t_2+k_2,\tau_2} \big)\\
&\hspace{6cm}		 + \cum \big( Y_{t_1,\tau_1}, Y_{t_2+k_2,\tau_2} \big) \cum \big( Y_{t_1+k_1,\tau_2}, Y_{t_2,\tau_1} \big) \Big] \\
& =: \ C_{1,n} + D_{1,n} + D_{2,n}.
\end{align*}
For $C_{1,n}$, note that
\begin{equation*}
\begin{split}
|C_{1,n}| &\leq \frac{1}{4\pi^2}\frac{\|K_n\|_\infty^2}{n^2} \sum_{|k_1| \leq B_n} \sum_{|k_2| \leq B_n} \sum_{t_1 = s}^{s+a_n} \sum_{t_2 = s}^{s+a_n} | \cum \big( Y_{t_1,\tau_1}, Y_{t_1+k_1,\tau_2}, Y_{t_2,\tau_1}, Y_{t_2+k_2,\tau_2} \big) | \\
& \leq \frac{1}{4\pi^2}\frac{\|K_n\|_\infty^2}{n^2} \sum_{t_1 = s}^{s+a_n}  \sum_{t_2, \ldots, t_4 = 1}^T | \cum \big( Y_{t_1,\tau_1}, Y_{t_2,\tau_2}, Y_{t_3,\tau_1}, Y_{t_4,\tau_2} \big) |
= O(a_n/n^2)
\end{split}
\end{equation*}
since the inner sum is bounded by Assumption~(A3)(iii) with $p=4,$ uniformly in~$t_1$. For the second inequality, note that $(t_1, k_1, t_2, k_2) \mapsto (t_1, t_1+k_1, t_2, t_2+k_2)$ is injective, as it has~$(s_1, s_2, s_3, s_4) \mapsto (s_1, s_2-s_1, s_3, s_4-s_3)$ as an inverse. 

Next, define $Y_{t,\tau}^\unu := \I{X_{t}^{\vartheta} \leq q^\unu(\tau)} - \tau$~and
\begin{equation*}
\begin{split}
D^{\vartheta}_{1,n} & := \frac{1}{4\pi^2}\sum_{|k_1| \leq B_n} \sum_{|k_2| \leq B_n} K_n(k_1) K_n(k_2) {\rm e}^{- {\rm i} (k_1 \omega_1 + k_2 \omega_2)}
\\
& \qquad \qquad \qquad \times  \frac{1}{n^2} \sum_{t_1 = s}^{s+a_n}\sum_{t_2 = s}^{s+a_n}
\cum \big( Y^{\vartheta}_{t_1,\tau_1}, Y^{\vartheta}_{t_2,\tau_1} \big) \cum \big( Y^{\vartheta}_{t_1+k_1,\tau_2}, Y^{\vartheta}_{t_2+k_2,\tau_2} \big), \\
D^{\vartheta}_{2,n} & := \frac{1}{4\pi^2}\sum_{|k_1| \leq B_n} \sum_{|k_2| \leq B_n} K_n(k_1) K_n(k_2) {\rm e}^{- {\rm i} (k_1 \omega_1 + k_2 \omega_2)}
\\
& \qquad \qquad \qquad \times \frac{1}{n^2}\sum_{t_1 = s}^{s+a_n}\sum_{t_2 = s}^{s+a_n}
\cum \big( Y^{\vartheta}_{t_1,\tau_1}, Y^{\vartheta}_{t_2+k_2,\tau_2} \big) \cum \big( Y^{\vartheta}_{t_1+k_1,\tau_2}, Y^{\vartheta}_{t_2,\tau_1} \big). \\
\end{split}
\end{equation*} 
After some computation, in view of local stationarity, there exists a constant $C$ such that, uniformly in $\tau_1,\tau_2,t_1,t_2,k_1,k_2,$
\begin{align*}
&\Big|\cum \big( Y_{t_1,\tau_1}, Y_{t_2,\tau_1} \big) \cum \big( Y_{t_1+k_1,\tau_2}, Y_{t_2+k_2,\tau_2} \big)\\
& \hspace{4cm}	- \cum \big( Y^{\vartheta}_{t_1,\tau_1}, Y^{\vartheta}_{t_2,\tau_1} \big) \cum \big( Y^{\vartheta}_{t_1+k_1,\tau_2}, Y^{\vartheta}_{t_2+k_2,\tau_2} \big)\Big| \leq C n/T.
\end{align*}
Note that
\[\sup_{t_1, k_1} \sum_{t_2 = s}^{s+a_n}  \kappa_{2}(t_1-t_2) \sum_{|k_2| \leq B_n} \kappa_{2}(t_1+k_1-(t_2+k_2)) < \infty,\]
which implies that
\[
\sup_{t_1, k_1} \sum_{t_2 = s}^{s+a_n}  \sum_{|k_2| \leq B_n} 
\min\Big( \frac{nC}{T}, \kappa_{2}(t_1-t_2) \kappa_{2}(t_1+k_1-(t_2+k_2))\Big) = o(1).
\]
This, along with assumption (A3)(iii) yields
\begin{align*}
| D_{1,n} &- D_{1,n}^{\vartheta} | \\
& \leq \frac{1}{4\pi^2 n^2} \sum_{t_1 = s}^{s+a_n} \sum_{|k_1| \leq B_n} \sum_{t_2 = s}^{s+a_n}  \sum_{|k_2| \leq B_n} 
\min\Big( \frac{nC}{T}, \kappa_{2}(t_1-t_2) \kappa_{2}(t_1+k_1-(t_2+k_2))\Big) \\
& = o(a_nB_n / n^2).
\end{align*}
A similar argument shows that $|D_{2,n} - D_{2,n}^{\vartheta}| = o(a_nB_n / n^2)$. Summarizing, we have shown that
\[
\E\Big[\sum_{t_1=s}^{s+a_n} W_{t_1,T}(\omega_1)\sum_{t_2=s}^{s+a_n} W_{t_2,T}(\omega_2) \Big] = D_{1,n}^{\vartheta} + D_{2,n}^{\vartheta} + o(a_nB_n / n^2).
\]
Now, arguments similar to the ones used to show that $C_{1,n} = O(a_n/n^2)$ yield
\[
D_{1,n}^{\vartheta} + D_{2,n}^{\vartheta} =  \E\Big[\sum_{t_1=s}^{s+a_n} W_{t_1}^\unu(\omega_1)\sum_{t_2=s}^{s+a_n} W_{t_2}^\unu(\omega_2) \Big] + o(a_nB_n / n^2)
\] 
where $W_{t}^\unu(\omega) := n^{-1}\sum_{|k|\leq n-1} K_n(k) e^{-i \omega k}  (Y_{t,\tau_1}^\unu Y_{t+k,\tau_2}^\unu - \E[Y_{t,\tau_1}^\unu Y_{t+k,\tau_2}^\unu])$.
Let
\[
h^\unu(\omega,\tau_1,\tau_2) := \frac{1}{2\pi}\frac{1}{a_n}\sum_{|k|\leq n-1} K_n(k) e^{-i \omega k} \sum_{t \in S_k(s,a_n)} (Y_{t,\tau_1}^\unu Y_{t+k,\tau_2}^\unu - \E[Y_{t,\tau_1}^\unu Y_{t+k,\tau_2}^\unu]).
\]
Proceeding as in \cite{rosenblatt1984}, pp. 1173-1174, we have that
\begin{equation} \label{eq:var1}
\Var\Big(\sum_{t=s}^{s+a_n} W_{t}^\unu(\omega) - \frac{a_n}{n}h^\unu(\omega,\tau_1,\tau_2)\Big) = O\Big(\frac{B_n^2}{n^2}\Big)
\end{equation}
uniformly in $|s-t_0|\leq m_T$ where $S_k(s,a_n) := \{t: s \leq t\leq s+a_n, s\leq t+k \leq s+a_n\}$. Now,~$h^\unu(\omega,\tau_1,\tau_2)$ is the usual lag-window estimator (centered by its expectation) of the cross-spectrum between $(Y_{t,\tau_1})_{s\leq t\leq s+a_n}$ and $(Y_{t,\tau_2})_{s\leq t\leq s+a_n}$. Thus, classical results from spectral density estimation yield
\begin{align} \label{eq:var2} 
\E[h^\unu(\omega_1,\tau_1,\tau_2) \overline{h^\unu(\omega_2,\tau_1,\tau_2)}] =& 2\pi \frac{B_n}{a_n} \int K^2(u) {\rm d} u \Big( \I{\omega_1=\omega_2} \mathfrak{f}^{\unu}(\omega_1, \tau_1, \tau_1)\mathfrak{f}^{\unu}(\omega_1, \tau_2, \tau_2) \nonumber \\ 
& + \I{\omega_1= - \omega_2} \mathfrak{f}^{\unu}(\omega_1, \tau_1, \tau_2)\mathfrak{f}^{\unu}(-\omega_1, \tau_2, \tau_1) \Big)+ o\Big(\frac{B_n}{a_n}\Big) .
\end{align}
This, combined with \eqref{eq:var1} and the fact that $W_{t,T}^\unu(\omega)$ and $h^\unu(\omega_2,\tau_1,\tau_2)$ are centered, entails
\[
\E\Big[\sum_{t_1=s}^{s+a_n} W_{t_1}^\unu(\omega_1)\sum_{t_2=s}^{s+a_n} W_{t_2}^\unu(\omega_2) \Big] - \frac{a_n^2}{n^2}\E[h^\unu(\omega_1,\tau_1,\tau_2)h^\unu(\omega_2,\tau_1,\tau_2)] = O\Big(\frac{ B_n^{3/2} a_n^{1/2} + B_n^2}{n^2} \Big).
\]
Since $B_n = o(a_n)$  by assumption, this with \eqref{eq:var2} completes the proof of Lemma \ref{lem:boundvarblock}. \hfill $\qed$

Next. observe that
\begin{align*}
& \tilde{\mathfrak{f}}_{t_0,T}(\omega, \tau_1, \tau_2) - \E \tilde{\mathfrak{f}}_{t_0,T}(\omega, \tau_1, \tau_2) \\
& = \frac{1}{2\pi}\sum_{|t-t_0| \leq m_T-B_n} \frac{1}{n}\sum_{|k|\leq n-1} K_n(k) e^{-i \omega k}  (Y_{t,\tau_1} Y_{t+k,\tau_2} - \E[Y_{t,\tau_1} Y_{t+k,\tau_2}])
\\
& =: \sum_{|t-t_0| \leq m_T-B_n} W_{t,T}(\omega).
\end{align*}
By construction, the random variables $W := \{W_{t,T}(\omega)\}_{|t-t_0| \leq m_T-B_n}$ form a triangular array of $\beta$-mixing random variables with $\beta$-mixing coefficients \[\beta^{[W]}(j) \leq \beta^{[X]}(0 \vee j-B_n).\] To establish the central limit theorem, we will apply the blocking technique from Section~$\ref{sec:blocking}$ with different block lengths $p_n,q_n$. Choose $p_n,q_n$ such that
\begin{equation} \label{eq:qpn}
q_n/p_n \to 0, \quad B_n/q_n \to 0, \text{ and } p_n/n \to 0. 
\end{equation}
Now decompose
\begin{align*}
\tilde{\mathfrak{f}}_{t_0,T}(\omega, \tau_1, \tau_2) - \E \tilde{\mathfrak{f}}_{t_0,T}(\omega, \tau_1, \tau_2) &= \sum_{j=1}^{\mu_n} \sum_{t \in \Gamma_j} W_{t,T}(\omega) + \sum_{j=1}^{\mu_n} \sum_{t \in \Delta_j} W_{t,T}(\omega) + \sum_{t \in R} W_{t,T}(\omega)\\
&=: S^n_{\Gamma} + S^n_{\Delta} + S^n_{R}, \quad  \text{say}.
\end{align*}
By construction, $S^n_{R}$ contains at most $O(p_n+q_n)$ summands. Lemma \eqref{lem:boundvarblock} thus implies that
\[
\Var\Big(\sum_{t \in R} W_{t,T}(\omega) \Big) = O\Big(\frac{(p_n+q_n)B_n}{n^2}\Big) = o(B_n/n),
\]
and, therefore, $S^n_{R} = o_P(B_n^{1/2}/n^{1/2})$. Next, observe that, by Lemma~\ref{lem:yu_alt},
\[
\p\Big( B_n^{1/2}n^{-1/2}\big|S^n_{\Delta}\big| \geq \eps \Big) = \p\Big(B_n^{1/2}n^{-1/2} \big|\sum_{j=1}^{\mu_n} \sum_{t \in \Delta_j} \xi_{t,T}(\omega) \big| \geq \eps \Big) + (\mu_n-1)\beta^{[W]}_{p_n}.
\]
The second term on the right-hand side of the above expression converges to zero by the assumptions on $p_n$ and $\beta^{[X]}$. To show that the first term also converges to zero, observe that, by construction $\E\xi_{t,T}(\omega) = 0$. The definition of $\xi_{t,T}(\omega),$ combined with Lemma \ref{lem:boundvarblock} and~$q_n/p_n = o(1)$, yields
\[
\Var(\frac{n^{1/2}}{B_n^{1/2}}\sum_{j=1}^{\mu_n} \sum_{t \in \Delta_j} \xi_{t,T}(\omega)) 
= \frac{n}{B_n}\sum_{j=1}^{\mu_n} \Var(\sum_{t \in \Delta_j} W_{t,T}(\omega))
= \frac{n}{B_n}O\Big(\frac{\mu_nB_nq_n}{n^2}\Big) = o(1). 
\]
Thus it remains to show that $\frac{n^{1/2}}{B_n^{1/2}} S^n_{\Gamma}$ converges in distribution. To this end, observe that, for any measurable set $A,$ we have, by Lemma~\ref{lem:yu_alt} and the assumptions on $\beta^{[W]},$
\[
\p\Big(\frac{n^{1/2}}{B_n^{1/2}} S^n_{\Gamma} \in A \Big)
= \p\Big(\frac{n^{1/2}}{B_n^{1/2}}\sum_{j=1}^{\mu_n} \sum_{t \in \Gamma_j} \xi_{t,T}(\omega) \in A \Big) + o(1).
\] 
Thus, it suffices to establish the weak convergence of $\frac{n^{1/2}}{B_n^{1/2}} \sum_{j=1}^{\mu_n} \sum_{t \in \Gamma_j} \xi_{t,T}(\omega)$. To do so, consider the triangular array of independent random vectors 
\[\Big(\frac{n^{1/2}}{B_n^{1/2}} \sum_{t \in \Gamma_j} \big(\Re \xi_{t,T}(\omega), \Im \xi_{t,T}(\omega)\big)^T \Big)_{j=1,...,\mu_n}.\]
Applying the Cram\'{e}r-Wold device, let us show that for any~$\lambda_1,\lambda_2 \in \R$ such that \mbox{$|\lambda_1|+|\lambda_2|\neq 0,$} the triangular array of independent random variables
\[
\Big(\frac{n^{1/2}}{B_n^{1/2}} \sum_{t \in \Gamma_j} \lambda_1 \Re \xi_{t,T}(\omega) + \lambda_2 \Im \xi_{t,T}(\omega)\Big)_{j=1,...,\mu_n}
\]
satisfies the Lyapunov condition. By construction $(W_{t,T}(\omega))_{t \in \Gamma_j} \overset{d}{=} (\xi_{t,T}(\omega))_{t \in \Gamma_j}$, so that 
\begin{align*}
& \E\Big[\Big(\sum_{t \in \Gamma_j} \Re \xi_{t,T}(\omega)\Big)^4 \Big]  = \E\Big[\Big(\sum_{t \in \Gamma_j} \Re W_{t,T}(\omega)\Big)^4 \Big]\\
& = 3\Big(\Var(\sum_{t \in \Gamma_j} \Re W_{t,T}(\omega)) \Big)^2 + \sum_{t_1,...,t_4 \in \Gamma_j}\cum(\Re W_{t_1,T}(\omega),...,\Re W_{t_4,T}(\omega)).
\end{align*}
A similar representation holds for the imaginary parts of $W_{t,T}$. Similar arguments as   on pages~1177-1178 of \cite{rosenblatt1984} show that 
\begin{align} \label{eq:rosen}
&\nonumber \sup_j \sum_{t_1,...,t_4 \in \Gamma_j} \Big( |\cum(\Re W_{t_1,T}(\omega),...,\Re W_{t_4,T}(\omega))| + |\cum(\Im W_{t_1,T}(\omega),...,\Im W_{t_4,T}(\omega))| \Big)\\
&\hspace{8mm}= O(q_n^2 B_n^2/n^4).
\end{align}
To verify this, note that, exactly as in \cite{rosenblatt1984},  the cumulants in (\ref{eq:rosen})   can be expressed in terms of cumulants of the random variables $Y_{t,\tau_j}, j = 1,2$, $t \in \Gamma_j$ by summation over indecomposable partititions. Apply (A3)(iii) to bound those cumulants uniformly, then follow the same arguments as in \cite{rosenblatt1984} to bound the sums. Then  \eqref{eq:rosen} entails, for any $\lambda_1,\lambda_2 \in \R,$
\[
\sum_{j=1}^{\mu_n}\E\Big[\Big(\sum_{t \in \Gamma_j} \lambda_1 \Re \xi_{t,T}(\omega) + \lambda_2 \Im \xi_{t,T}(\omega) \Big)^4\Big] = O(\mu_nq_n^2 B_n^2/n^4)
\]
and, by Lemma \eqref{lem:boundvarblock}, for any $\lambda_1,\lambda_2 \in \R$ with $|\lambda_1|+|\lambda_2|\neq 0$,
\[
\Big( \sum_{j=1}^{\mu_n} \Var(\sum_{t \in \Gamma_j} \lambda_1 \Re \xi_{t,T}(\omega) + \lambda_2 \Im \xi_{t,T}(\omega)) \Big)^2 \geq c_0(\lambda_1,\lambda_2) \mu_n^2q_n^2 B_n^2/n^4
\]
for some $c_0(\lambda_1,\lambda_2) > 0$ for sufficiently large $n$. Thus the conditions of Lyapunovs central limit theorem are satisfied as $\mu_n \to \infty$. This completes the proof of \eqref{eq:asynorm}. \hfill \qed \\

\subsection{Proof of \eqref{eq:bias}}
The proof of \eqref{eq:bias} relies on the following lemma (see \cite{priestley1981}, page 459 for similar arguments).
\begin{lemma} \label{lem:bias}
	Uniformly in $|u-\unu| \leq n/T$ and $x,y$ in a neighborhood of $\tau_1,\tau_2$, we have
	\begin{equation*}\label{eq:bias1}
	\frac{1}{2\pi} \sum_{|k|\leq n-1} K_n(k) e^{-i \omega k} \gamma_k^{u}(x,y) = \mathfrak{f}^u(\omega,x,y) - C_K(r)B_n^{-r} \mathfrak{d}^{(r)}_\omega\mathfrak{f}^u(\omega,x,y) + o(B_n^{-r}). 
	\end{equation*}
\end{lemma}
\noindent\textbf{Proof.} Choose some $L_n \to \infty$ such that $L_n/B_n \to 0$. Then,
\begin{align*}
\frac{1}{2\pi}\sum_{|k|\leq n-1}& K_n(k) e^{-i \omega k} \gamma_k^{u}(x,y) - \mathfrak{f}^{u}(\omega,x,y)\\
& = B_n^{-r} \frac{1}{2\pi}\sum_{|k|\leq L_n} \frac{K(k/B_n)-1}{|k/B_n|^r} |k|^r e^{-i \omega k} \gamma_k^{u}(x,y) \\
& \quad + B_n^{-r}\frac{1}{2\pi}\sum_{B_n \geq |k| > L_n} \frac{K(k/B_n)-1}{|k/B_n|^r} |k|^r e^{-i \omega k} \gamma_k^{u}(x,y)\\
& \qquad- B_n^{-r} \frac{1}{2\pi}\sum_{|k| > B_n}  e^{-i \omega k} \frac{B_n^r}{|k|^r}|k|^r\gamma_k^{u}(x,y).
\end{align*}
By Assumption (K) and (A3)(ii) $\sup_{v} \frac{|K(v)-1|}{|v|^r}$ and $\sum_{k \in \Z} |k|^r |\gamma_k^{u}(x,y)|$ are  bounded. Therefore, the last term in the above expression is 
\[O(B_n^{-r}) \sum_{|k| > B_n} |k|^r |\gamma_k^{u}(x,y)| = o(B_n^{-r}),\] and the second term is
\begin{align*}
\Bigg|B_n^{-r}\frac{1}{2\pi} \sum_{B_n \geq |k| > L_n} & \frac{K(k/B_n)-1}{|k/B_n|^r} |k|^r e^{-i \omega k} \gamma_k^{u}(x,y)\Bigg| \\
&\leq O(B_n^{-r})\sup_{v} \frac{|K(v)-1|}{|v|^r}\sum_{B_n \geq |k| > L_n}  |k|^r |\gamma_k^{u}(x,y)| = o(B_n^{-r}),
\end{align*}
since $L_n \to \infty.$ Finally, for the first term, observe that
\begin{align} \label{eq:firstterm}
& \frac{1}{2\pi} B_n^{-r}\sum_{|k|\leq L_n} \frac{K(k/B_n)-1}{|k/B_n|^r} |k|^r e^{-i \omega k} \gamma_k^{u}(x,y) + C_K(r)B_n^{-r} \mathfrak{d}^{(r)}_\omega\mathfrak{f}^u(\omega,x,y)
\\
& \nonumber \quad =  \frac{1}{2\pi} B_n^{-r} \sum_{|k|\leq L_n} \Big(\frac{K(k/B_n)-1}{|k/B_n|^r} + C_K(r)\Big)|k|^r e^{-i \omega k} \gamma_k^{u}(x,y)\\
& \nonumber \qquad + \frac{1}{2\pi} C_K(r)B_n^{-r} \sum_{|k| > L_n} |k|^r e^{-i \omega k} \gamma_k^{u}(x,y).
\end{align}
The first term in the right-hand side of \eqref{eq:firstterm} is of order $o(B_n^{-r})$ since, by Assumption (K), $L_n/B_n \to 0$ implies $\frac{K(k/B_n)-1}{|k/B_n|^r} \to -C_K(r)$ and $|k|^r |\gamma_k^{u}(x,y)|$ is absolutely summable, while the second term is $o(B_n^{-r})$  since $L_n \to \infty $ and $|k|^r |\gamma_k^{u}(x,y)|$ is absolutely summable. Note that, under the assumptions made, all arguments hold uniformly in $u,x,y$. This completes the proof. \hfill $\qed$

We can now prove \eqref{eq:bias}. First, note that
\begin{align*}
\E \tilde{\mathfrak{f}}_{t_0,T}(\omega, \tau_1, \tau_2) = \frac{1}{2\pi}\frac{1}{n}&\sum_{|k|\leq n-1} K_n(k) e^{-i \omega k} \\
&\times \sum_{|t-t_0|\leq m_T-B_n} \gamma_k^{t/T}(G^{t/T}(q^\unu(\tau_1)),G^{t/T}(q^\unu(\tau_2))) +  O(B_n/T).
\end{align*}
Next, observe that by \ref{lem:bias} and the continuity of $(u,x,y) \mapsto \mathfrak{d}^{(r)}_\omega\mathfrak{f}^u (\omega,x,y),$
\[
\sup_{|t-t_0|\leq m_T-B_n} \Big| \mathfrak{d}^{(r)}_\omega\mathfrak{f}^{t/T}(\omega,G^{t/T}(q^\unu(\tau_1)),G^{t/T}(q^\unu(\tau_2))) - \mathfrak{d}^{(r)}_\omega\mathfrak{f}^\unu(\omega,\tau_1,\tau_2) \Big| = o(1)
\]
since $G^{u}(q^\unu(\tau)) \to \tau$ for $u \to \unu$. Thus
\begin{align*}
\E \tilde{\mathfrak{f}}_{t_0,T}(\omega, \tau_1, \tau_2) = & - C_K(r)B_n^{-r} \mathfrak{d}^{(r)}_\omega\mathfrak{f}^\unu(\omega,\tau_1,\tau_2)
\\
& + \frac{1}{n} \sum_{|t-t_0|\leq m_T-B_n} \mathfrak{f}^{t/T}(\omega,G^{t/T}(q^\unu(\tau_1)),G^{t/T}(q^\unu(\tau_2))) + o(B_n^{-r}). 
\end{align*}
On the other hand,
\begin{align*}
& \frac{1}{n} \sum_{|t-t_0|\leq m_T-B_n} \mathfrak{f}^{t/T}(\omega,G^{t/T}(q^\unu(\tau_1)),G^{t/T}(q^\unu(\tau_2)))
\\
& = \frac{T}{2m_T} \int_{-m_T/T}^{m_T/T} \mathfrak{f}^{\unu+u}(\omega,G^{\unu+u}(q^\unu(\tau_1)),G^{\unu+u}(q^\unu(\tau_2))) + O(1/n)
\\
& = \mathfrak{f}^{\unu}(\omega,q^\unu(\tau_1),q^\unu(\tau_2))\\
& \hspace{2cm} + \frac{n^2}{2T^2}\frac{\partial^2}{\partial u^2} \mathfrak{f}^u(\omega,G^{u}(q^\unu(\tau_1)),G^{u}(q^\unu(\tau_2)))\Big|_{u=\unu} + O(1/n) + o(n^2/T^2).
\end{align*}
Statement \eqref{eq:bias} follows. \hfill \qed


\end{document}